\documentclass[preprint]{elsarticle}
\usepackage[latin1]{inputenc}
\usepackage[T1]{fontenc}
\usepackage[bookmarksnumbered=true]{hyperref}
\usepackage{amsfonts,amstext,amsmath,amsthm,amssymb}
\usepackage{amscd}
\usepackage{fullpage}
\usepackage{color}
\usepackage{graphicx}
\usepackage{verbatim}

\newtheoremstyle{slplain}
  {1.0\baselineskip\@plus.2\baselineskip\@minus.2\baselineskip}
  {.0\baselineskip\@plus.2\baselineskip\@minus.2\baselineskip}
  {\slshape}
  {}
  {\bfseries}
  {}
  { }
  {}

\DefineNamedColor{named}{JungleGreen} {cmyk}{0.99,0,0.52,0}
\DefineNamedColor{named}{Gray}{cmyk}{0,0,0,0.50}
\DefineNamedColor{named}{ProcessBlue}   {cmyk}{0.96,0,0.0,0}
\DefineNamedColor{named}{Purple}        {cmyk}{0.45,0.86,0,0}

\usepackage[normalem]{ulem}
\DefineNamedColor{named}{orange} {rgb}{1,0.55,0}

\newtheorem{example}{example}[section]
\newtheorem{theorem}[example]{Theorem}
\newtheorem{proposition}[example]{Proposition}
\newtheorem{definition}[example]{Definition}
\newtheorem{lemma}[example]{Lemma}
\newtheorem{corollary}[example]{Corollary}

\providecommand{\abs}[1]{\left|#1\right|}
\providecommand{\norm}[1]{\left\Vert#1\right\Vert}
\providecommand{\ip}[2]{\left\langle#1,#2\right\rangle}
\providecommand{\br}[1]{\left\lbrace#1\right\rbrace}
\providecommand{\fl}[1]{\left\lfloor#1\right\rfloor}
\providecommand{\cl}[1]{\left\lceil#1\right\rceil}

\theoremstyle{definition}
\newtheorem{remark}[example]{Remark}

\newcommand{\thickhline}{\noalign{\hrule height 0.8pt}}
\newcommand\T{\rule{0pt}{2.6ex}}
\newcommand\B{\rule[-1.2ex]{0pt}{0pt}}

\newcommand*{\cH}{\mathcal{H}}
\newcommand*{\cS}{\mathcal{S}}
\newcommand*{\cT}{\mathcal{T}}

\newcommand*{\bbN}{\mathbb{N}}
\newcommand*{\bbZ}{\mathbb{Z}}
\newcommand*{\bbR}{\mathbb{R}}

\begin{document}
\begin{frontmatter}



\title{On optimal wavelet reconstructions from Fourier samples: linearity and universality of the stable sampling rate}


\author[purdue]{B. Adcock\corref{cor2}} 
\ead{adcock@purdue.edu}

\author[cambridge,vienna]{A. C. Hansen} 
\ead{a.hansen@damtp.cam.ac.uk}

\author[cambridge]{C. Poon\corref{cor1}} 
\ead{cmhsp2@cam.ac.uk}

\cortext[cor1]{Corresponding author}

\address[purdue]{Department of Mathematics, Purdue University}
\address[cambridge]{Department of Applied Mathematics and Theoretical Physics, University of Cambridge}
\address[vienna]{Institut f\"{u}r Mathematik, Universit\"{a}t Wien,}

\begin{abstract}

In this paper we study the problem of computing wavelet coefficients of compactly supported functions from their Fourier samples.  For this, we use the recently introduced framework of generalized sampling.  Our first result demonstrates that using generalized sampling one obtains a stable and accurate reconstruction, provided the number of Fourier samples grows linearly in the number of wavelet coefficients recovered.  For the class of Daubechies wavelets we derive the exact constant of proportionality.

Our second result concerns the optimality of generalized sampling for this problem.  Under some mild assumptions we show that generalized sampling cannot be outperformed in terms of approximation quality by more than a constant factor.  Moreover, for the class of so-called perfect methods, any attempt to lower the sampling ratio below a certain critical threshold necessarily results in exponential ill-conditioning.  Thus generalized sampling provides a nearly-optimal solution to this problem.
\end{abstract}


\end{frontmatter}

\section{Introduction}\label{s:introduction}
One of the most fundamental problems in sampling theory is the issue of how to recover an object -- an image or signal, for example -- from a finite, and typically fixed, collection of its measurements.  This problem lies at the heart of countless algorithms, with applications ranging from medical imaging to astronomy.

An important instance of this problem is the recovery of a compactly supported function from pointwise measurements of its Fourier transform.  This problem occurs notably in Magnetic Resonance Imaging (MRI), as well as other applications such as radar. The classical approach for this problem is to recover $f$ by computing a discrete Fourier transform (DFT) of the given data.  However, this approach suffers from a number of drawbacks, including the sensitivity to motion and the presence of unpleasant Gibbs ringing \cite{HealyWeaverWaveletIEEE,WeaverEtAlWaveletEncoding}.  Such phenomena can present serious issues in applications.

\subsection{Wavelets in imaging}
It is known that many real-life images can be much more efficiently represented by using wavelets than by their Fourier series.  Images may be sparse in wavelets, or their coefficients may have improved decay properties.    Representing  medical images in this way also has several other benefits over the classical Fourier representation.  These include better compressibility, improved feature detection (see \cite{UnserAldroubiWaveletReview,UnserAldroubiLaineEditorial} and references therein), and easier and more effective denoising \cite{LaineWaveletsBiomed,NowakWaveletDenoise,WeaverEtAlWaveletFiltering}.  For these reasons, the use of wavelets in biomedical imaging applications has been a significant area of research for several decades \cite{LaineWaveletsBiomed,UnserAldroubiWaveletReview,UnserAldroubiLaineEditorial}.

Seeking to exploit these beneficial properties, an approach to recover wavelet coefficients directly in MRI was introduced in 1992 by Weaver et al \cite{HealyWeaverWaveletIEEE,WeaverEtAlWaveletEncoding} (see also \cite{GelmanWood3DWaveletEncoding,LaineWaveletsBiomed,PanychEtAlWaveletEncoded} and references therein).  This is known as \textit{wavelet-encoded} MRI.  In this technique, the MR scanner itself is modified to sample wavelet coefficients along one dimension, with Fourier sampling, followed by a one-dimensional DFT, applied in the other.  The resulting reconstructed image suffers less from Gibbs ringing, has fewer motion artefacts, and can in principle be acquired more rapidly \cite{PanychWaveletEncoding,PanychEtAlWaveletEncoded}.  For a medical perspective on wavelet encoding, and a discussion on how it can be combined with other imaging techniques such as parallel MRI, see \cite{KyriakosEtAlGeneralizedEncoding}.

Unfortunately, there are a number of disadvantages to wavelet encoding, which limit its applicability.  These include low signal-to-noise ratio \cite{PanychWaveletEncoding,WeaverEtAlWaveletEncoding}, and the extra complications encountered in the acquisition process due to having to modify the MR scanner \cite{LaineWaveletsBiomed}.  Moreover, the state-of-the-art wavelet encoding allows only for reconstructions of wavelet coefficients of a 2D image in one direction, and thus does not permit one to take full advantage of general wavelets.

Nonetheless, the intensity of work on wavelets in MRI, and in particular on wavelet encoding techniques, indicates the importance of the problem of computing wavelet coefficients of biomedical imaging.  It also serves to highlight the fact that this problem remains largely unsolved. 

With this in mind, the purpose of this paper is to introduce and analyse a different solution to this problem, known as \textit{generalized sampling}.  Unlike wavelet encoding, which is primarily an engineering exercise in which the scanner itself is modified to produce different samples, we take the mathematical viewpoint and consider the samples as being fixed Fourier samples, and then seek to reconstruct wavelet coefficients directly via a post-processing algorithm.  Our main conclusion is that one can perform wavelet encoding in applications such as MRI by generalized sampling without altering the scanner at all. This allows for the use of arbitrary wavelets and removes any hardware restrictions.

The typical MRI problem concerns the recovery of two- or three-dimensional images from Fourier measurements.  In this paper, we shall consider only the one-dimensional case.  As we explain further in Section \ref{s:conclusion}, both the technique of generalized sampling and its analysis can be extended to the higher-dimensional setting.  This is a topic of ongoing work.  The development and analysis of the one-dimensional case, i.e.\ the topic of this paper, can be viewed as a vital first step in this direction.

\begin{remark}\label{r:DFT}
The reader may wonder at this stage why wavelet encoding is necessary. Why could one not simply recover wavelet coefficients from standard MRI data by applying the DFT and DWT (discrete wavelet transform) in turn? There are two reasons. First, the use of DFT yields a discrete (pixel-based) version of the truncated Fourier series. Hence, by applying the DWT one (at best) obtains the wavelet coefficients of the truncated Fourier series and not the actual wavelet coefficients of the image itself. Second, the recovery algorithm using DFT and DWT would be as follows. The "wavelet coefficients" are obtained by
$$
x = \mathrm{DWT} \cdot \mathrm{DFT}^{-1} y,
$$
where $y$ is a vector of the Fourier samples.
However, when mapping these coefficients back to the pixel domain, one gets
$$
\tilde x = \mathrm{DWT}^{-1}x = \mathrm{DFT}^{-1} y,
$$
which is exactly what we would get in the first place using DFT. In particular, nothing is gained here in terms of the quality of the reconstructed image.  By contrast, wavelet encoding techniques seek to reconstruct the true wavelet coefficients directly.  This yields a different reconstruction with qualities determined by the wavelet used, and not by the original Fourier series.
\end{remark}

\subsection{Generalized sampling}

In sampling theory, the mathematical problem of recovering the coefficients of a signal or image in a particular basis from samples taken with respect to another basis has been studied for several decades \cite{unser2000sampling}.  Motivating this is the fact that many images and signals can be better represented in terms of a different basis (e.g. splines \cite{UnserSplinesFit} or the aforementioned wavelets) than the basis in which they are sampled (e.g. the Fourier basis).  Some of the earliest work on this problem in its abstract form was carried out by Unser \& Aldroubi, who introduced a mathematical reconstruction framework known as \textit{consistent reconstructions} for shift-invariant sampling and reconstruction spaces \cite{unser1994general} (see also \cite{unserzerubia}).  This was later considered by Eldar et al, who extended this framework to frames in arbitrary Hilbert spaces \cite{EldarRobConsistSamp,eldar2003FAA,eldar2003sampling,eldar2005general}.  Further developments to more general types of signal models were introduced in \cite{LuDoUnionSubspace} (see also \cite{BlumensathUnionSubspace,EldarUnionSubspace}).

Whilst consistent reconstructions are quite popular in engineering applications, there are a number of issues.  As discussed in \cite{BAACHSampTA,BAACHShannon,EldarMinimax,UnserHirabayashiConsist}, consistent reconstructions have the significant drawback of being, in general, neither numerically stable nor convergent as the number of samples is increased.  Hence, when applied to the important problem of recovering wavelet coefficients of MR images, they can result in severe amplification of noise and round-off error.

Nonetheless, it transpires that these issues can be overcome completely by using a different approach, known as \textit{generalized sampling}.  Introduced by Adcock \& Hansen in \cite{BAACHShannon,BAACHAccRecov}, based on elements from \cite{hansen2011}, this framework allows one to recover a signal $f$ modelled as an element of a separable Hilbert space $\mathcal{H}$ in terms of any Riesz basis $\{ \varphi_j \}^{\infty}_{j=1}$ from samples $\{ \ip{f}{s_j} \}^{\infty}_{j=1}$ taken with respect to any other Riesz basis $\{ s_j \}^{\infty}_{j=1}$ of $\mathcal{H}$.  The resulting reconstruction is both convergent and numerically stable, and therefore an obvious candidate for the wavelet recovery problem.  The extension of this framework to frames, as opposed to bases, was presented in \cite{AHHTillposed}.  See also \cite{BAACHOptimality}.

Keeping this in mind, the aim of this paper is to show that generalized sampling effectively solves the longstanding problem of recovering wavelet coefficients from Fourier samples in the one-dimensional case.  Our main results are explained in more detail in the next section.

\subsection{Main results}
Generalized sampling obtains a reconstruction by performing a simple least-squares procedure.  The fundamental principle which gives this method its stability and accuracy (as opposed to a consistent reconstruction) is that the number of computed coefficients $N$ in the reconstruction basis  $\{ \varphi_j \}^{\infty}_{j=1}$ (i.e. the wavelet basis) should be allowed to differ from the number $M$ of acquired samples $\{ \ip{f}{s_j} \}^{M}_{j=1}$ (i.e. Fourier samples).  In \cite{BAACHOptimality}, this was posed in terms of the so-called \textit{stable sampling rate} $\Theta(N;\theta)$.  Given $N$ coefficients to be recovered, sampling \textit{at a rate} $M \geq \Theta(N;\theta)$ ensures a numerically stable and quasi-optimal reconstruction of $f$ (see Section \ref{s:GS} for definitions), with the stability and quasi-optimality constants depending on the fixed parameter $\theta$.  

Understanding the behaviour of $\Theta(N;\theta)$ is critically important from a practical standpoint.  In the problem we consider in this paper, for example, it allows one to determine \textit{a priori} how many Fourier samples are required to compute $N$ wavelet coefficients in a manner that is stable and accurate (i.e. the computed wavelet coefficients closely approximate the exact wavelet coefficients).  Clearly, it is both wasteful and time-consuming to acquire more samples than necessary.  Hence, the main goal of this paper is to obtain good estimates for the stable sampling rate in the context of reconstructing in compactly supported Multiresolution Analysis (MRA) wavelet bases from one-dimensional Fourier-encoded data. Precise definitions of the reconstruction and sampling spaces can be found in Section \ref{sec:setup:recostruction}

The first result we prove in this paper is that the stable sampling rate is linear  in this setting.  Thus, if $N$ wavelet coefficients are required, one only needs $\mathcal{O}(N)$ Fourier samples of $f$ to apply generalized sampling.  In this sense, wavelets give rise to \textit{ideal} bases for the Fourier samples reconstruction problem: up to a constant factor, there is a one-to-one ratio correspondence between Fourier samples and wavelet coefficients.  Hence generalized sampling not only solves the long-standing problem of how to recover wavelet coefficients from MR data, but it also does so in a way that is, up to a constant factor, optimal.  

This result suggests that little can be gained in terms of reconstruction quality by altering the MR scanner, as is done in wavelet encoding techniques.  The problem of recovering wavelet coefficients can be readily solved without doing this by post-processing of the standard Fourier-encoded MR data with generalized sampling.  We remark that this conclusion is due completely to the linear scaling of $\Theta(N;\theta)$.  Had the scaling been more severe, as can be the case for other reconstruction bases -- orthogonal polynomials, for example, have quadratic stable sampling rates, $\Theta(N;\theta) = \mathcal{O}(N^2)$ (a result due originally to Hrycak \& Gr\"ochenig \cite{hrycakIPRM}, see also \cite{BAACHAccRecov}) -- then generalized sampling may well not be as good an approach to the problem as alternatives based on modifying the sampling process.

Given that wavelets have linear stable sampling rates, it is natural to ask how large the ratio $\eta(\theta) = M/N$, which we henceforth refer to as the \textit{stable sampling ratio}, is required to be.  Specifically, is it possible to have the optimal ratio $\eta(\theta) = 1$ for some moderate value of $\theta$, and thus get a stable, accurate reconstruction using an equal number of wavelets as Fourier samples?  Our second result shows that in general this is not the case.  Indeed, every pair of Fourier and wavelet bases is associated with a critical threshold $\eta^*$ below which the reconstruction becomes exponentially unstable. On the other hand, for certain wavelet bases, such as Daubechies wavelets, a ratio of at least $\eta^*$ will ensure complete stability.

The third issue we address in this paper is the question of optimality of generalized sampling: that is, whether or not it can be outperformed by a different method.  This question is equivalent to asking whether the stable sampling rate is a quantity intrinsic to generalized sampling, or whether it is in fact \textit{universal}.  In other words, does the stable sampling rate place a fundamental limit on the number of Fourier samples required to recover $N$ wavelet coefficients in a stable, accurate manner, regardless of the method used?

Optimality of generalized sampling was first discussed in \cite{BAACHOptimality}.  Using a general result proved therein, we show that the stable sampling rate is indeed universal for all so-called \textit{perfect} methods (i.e. methods which recover finite sums of wavelets in a reasonable way; see Section \ref{s:GS} for a definition).  As a result of this, we show that for wavelet reconstructions, any perfect method with ratio less than $\eta^*$ must be exponentially unstable. Hence, there is always a limit to the amount of improvement over generalized sampling that any perfect method can offer.

Unfortunately, perfect methods represent only a subclass of all possible reconstruction techniques.  Hence it cannot be claimed that the stable sampling rate is truly universal.  Indeed, perfectness of a method implies that it recovers all functions in a particular class rather well.  This leads to the following question: is it possible to devise a different method which outperforms generalized sampling for a single function $f$?  Using our results on the linear scaling of the stable sampling rate, we show under a mild assumption that such a method can at best give a reconstruction whose approximation error is a constant factor smaller than that of generalized sampling.  Thus, although it is possible to outperform generalized sampling in terms of approximation error, only the constant can be improved and not the asymptotic rate.  In this sense, generalized sampling is, up to a constant factor, an oracle for the problem.

\subsection{Related works}
Similar ideas for reconstructions in MRI were initially introduced by Pruessmann et al under the name of Sensitivity Encoding in \cite{pruessmann1999sense}.   They considered reconstructions in terms of voxel shapes from Fourier-encoded data by solving a least squares problem and this method has since been used in more general wavelet reconstructions. However, the least squares problem can become ill-posed and various authors have sought to resolve this by imposing some quadratic regularization constraints.  In a later work,  \cite{guerquin2011fast} introduced an $l^1$ regularization term to resolve this ill-posedness (see also \cite{daubechies2004iterative}). Whilst the mentioned works provide algorithms for the computation of the reconstructions, this ill-posedness and the error from the true image is not well understood.  The main contribution of \cite{BAACHShannon,BAACHOptimality} is to provide an abstract framework, known as generalized sampling, under which reconstruction schemes including these can be formally analysed.  By considering the one-dimensional problem of wavelet reconstructions from Fourier samples in the framework of generalized sampling, we provide a rigorous analysis of the error and stability of the resultant scheme. We demonstrate that when the number of Fourier samples and number of wavelet coefficients are chosen in accordance with a linear stable sampling rate, the generalized sampling scheme is stable and convergent and there is no need for extra regularization constraints. The work here may thus be seen as a starting point for a theoretical understanding of the scheme presented in \cite{pruessmann1999sense}. 

\subsection{Outline}
The outline for the remainder of this paper is as follows.  In Section \ref{s:GS}, we recap the generalized sampling framework of \cite{BAACHShannon,BAACHAccRecov,BAACHOptimality}. In Section \ref{s:examples}, we present two examples to illustrate the use of generalized sampling.  The main results of the paper are presented and discussed in Section \ref{s:main}, and proofs are given in Sections \ref{s:prf1}--\ref{s:prf3}.  In Section \ref{s:numexp}, we provide numerical results.

\section{Generalized sampling}\label{s:GS}

\subsection{Generalized sampling}

In this section, we recap the main details of generalized sampling from \cite{BAACHShannon,BAACHAccRecov}, and in particular \cite{BAACHOptimality}. Let $\mathcal{H}$ be a separable Hilbert space with inner product $\ip{\cdot}{\cdot}$ and norm $\norm{\cdot}$.  Suppose that $\mathcal{S}$ and $\mathcal{T}$ are closed subspaces of $\mathcal{H}$ satisfying the subspace condition
\begin{equation}\label{subspacecondition}
\mbox{$\mathcal{T} \cap \mathcal{S}^{\perp} = \{0\}$ and $\mathcal{T} + \mathcal{S}^{\perp}$ is closed in $\mathcal{H}$.}
\end{equation}
Let $\{ s_j \}^{\infty}_{j=1}$ be an orthonormal basis for $\cS$, and for $f \in \cH$, let 
\begin{equation*}
\hat{f}_j = \ip{f}{s_j},\quad j \in \mathbb{N},
\end{equation*}
be the \textit{samples} of $f$.  The reconstruction problem is to recover $f$ with an element $\tilde{f} \in \cT$ from its samples $\{ \hat{f}_j \}^{\infty}_{j=1}$.

In practice, one does not have access to the whole set $\{ \hat{f}_j \}^{\infty}_{j=1}$ of samples, nor can one process infinite amounts of information.  Hence, in computations we consider the problem of recovering $f$ from its first $M$ samples
\begin{equation*}
\hat{f}_1,\ldots,\hat{f}_M.
\end{equation*}
Also, it is usual to assume that there exists a sequence $\{ \mathcal{T}_N \}^{\infty}_{N=1}$ of finite-dimensional subspaces of $\cT$ satisfying
\begin{equation}\label{T_N}
\cT_1 \subseteq \cT_2 \subseteq \cdots \subseteq \cT,\quad \overline{\bigcup^{\infty}_{N=1} \cT_N} = \cT.
\end{equation}
For example, if $\{ \varphi_j \}^{\infty}_{j=1}$ is a frame or a Riesz basis for $\cT$, then one typically has
\begin{equation*}
\cT_N = \mbox{span} \left \{ \varphi_1,\ldots,\varphi_N \right \}.
\end{equation*}
The reconstruction problem is now formulated as follows: given $N \in \bbN$, compute a reconstruction $\tilde{f}_{N,M} \in \cT_N$ of $f$ from the samples $\{ \hat{f}_j \}^{M}_{j=1}$.

In order to formulate what constitutes a `good' reconstruction, we consider the following two definitions \cite{BAACHOptimality}:
\begin{definition}
Let $F_{N,M} : \cH	 \rightarrow \cT_N$. 
The quasi-optimality constant $\mu = \mu(F_{N,M})$ is the least constant such that
\begin{equation*}
\| f - F_{N,M}(f) \| \leq \mu \| f - Q_N f \|,\quad \forall f \in \cH,
\end{equation*}
where $Q_N : \cH \rightarrow \cT_N$ is the orthogonal projection onto $\cT_N$.  If no such constant exists, we write $\mu = \infty$.  We say that $F_{N,M}$ is quasi-optimal if $\mu(F_{N,M})$ is small.
\end{definition}
Note that $Q_{N}f$ is the best approximation in norm to $f$ from $\mathcal{T}_N$.  So quasi-optimality means that the difference in norm between $f$ and $F_{N,M}(f)$ is at most a constant factor $\mu$ of the difference between $f$ and its best approximation in the subspace $\mathcal{T}_N$.

We also define the condition number of a reconstruction:
\begin{definition}
Let $F_{N,M} : \cH	 \rightarrow \cT_N$ be a mapping such that, for each $f \in \mathcal{H}$, $F_{N,M}(f)$ depends only on the samples $\{ \hat{f}_j \}^{M}_{j=1}$.
 The condition number of $\kappa(F_{N,M})$ is given by
\begin{equation*}
\kappa(F_{N,M})=\sup_{f \in \cH} \lim_{\epsilon \rightarrow 0^+} \sup_{\substack{g \in \cH \\ 0 < \| \hat{g} \|_{l^2} \leq \epsilon}} \frac{\| F_{N,M}(f+g) - F_{N,M}(f) \|}{\| \hat{g} \|_{l^2}},
\end{equation*}
where $\hat{g} = \{ \hat{g}_j \}^{M}_{j=1} \in \mathbb{C}^M$.  The mapping $F_{N,M}$ is well-conditioned if $\kappa(F_{N,M})$ is small and ill-conditioned otherwise.
\end{definition}

We say that the reconstruction $F_{N,M}$ is `good' if it is stable and quasi-optimal.  In other words, if the \textit{reconstruction constant}
\begin{equation*}
C(F_{N,M}) = \max \{ \kappa(F_{N,M}) , \mu(F_{N,M}) \},
\end{equation*}
is small.  

As we shall explain in a moment, the key to obtaining a good reconstruction is to allow the parameter $M$, the number of samples, to vary independently from $N$.  To this end, suppose now we write $P_M : \cH \rightarrow \cS_M$ for the orthogonal projection onto the subspace $\cS_M =\mathrm{span}\br{s_1,\ldots,s_M}$, i.e.
\begin{equation*}
P_M g = \sum^{M}_{j=1} \ip{g}{s_j} s_j,\quad g \in \cH.
\end{equation*}
The method of tackling the reconstruction problem proposed in \cite{BAACHShannon} is to let $\tilde{f}_{N,M} = F_{N,M}(f) \in \cT_N$ be defined by 
\begin{align}\label{eq:gensamp_soln}
 \ip{P_M \tilde{f}_{N,M}}{\varphi_j}=\ip{P_M f}{\varphi_j}, \ \ j=1,\ldots,N.
\end{align}
Note that solving (\ref{eq:gensamp_soln}) is equivalent to finding $\alpha^{[N,M]}=\br{\alpha_1^{[N,M]},\ldots,\alpha_N^{[N,M]}}\in \mathbb{C}^N$ as the least-squares solution to the problem
\begin{equation}\label{least_sq_formulation}
U^{[N,M]} \alpha^{[N,M]} = \hat{f}^{[M]},
\end{equation}
where $\hat{f}^{[M]} = \{ \ip{f}{s_1},\ldots,\ip{f}{s_M} \}$ and $U^{[N,M]}$ is the $M$ by $N$ matrix whose $(i,j)^{th}$ entry is $\ip{\varphi_j}{s_i}$.  The reconstruction $\tilde{f}_{N,M}$ is then given by $\sum^{N}_{j=1} \alpha^{[N,M]}_{j} \varphi_j$. 

Furthermore, the uniqueness of the solution to (\ref{least_sq_formulation}), the condition number and the quasi-optimality of generalized sampling are all determined by the subspace angle between $\cT_N$ and $\cS_M$, namely, the value  $C_{N,M}=\sqrt{\inf_{\substack{\varphi \in \cT_N \\ \| \varphi \|=1}} \ip{P_M \varphi}{\varphi}}$. In \cite{BAACHShannon,BAACHOptimality}, it was established that when $C_{N,M}>0$, the solution is uniquely
\begin{equation}\label{eq:gs_exact_rep}
\alpha^{[N,M]} = \left(\left( U^{[N,M]}\right)^* U^{[N,M]}\right)^{-1}\left(U^{[N,M]}\right)^*\hat{f}^{[M]}
\end{equation}
and the reconstruction constant $C(F_{N,M})$ of generalized sampling satisfies
\begin{equation*}
C(F_{N,M}) = \kappa(F_{N,M}) = \mu(F_{N,M} ) = \frac{1}{C_{N,M}}.
\end{equation*}
Moreover, since $P_M \rightarrow P$ strongly on $\cH$ as $M \rightarrow \infty$ (where $P : \cH \rightarrow \cH$ is the projection onto $\mathcal{S}$), one has, via (\ref{subspacecondition}) and (\ref{T_N}), that
\begin{equation*}
C_{N,M} \rightarrow 1,\quad M \rightarrow \infty,
\end{equation*}
for fixed $N \in \bbN$.  Thus, one obtains a good reconstruction by allowing $M$ to be sufficiently large in comparison to $N$.  

To quantify how large $M$ is required to be, the concept of the \textit{stable sampling rate} was introduced in \cite{BAACHOptimality}:
\begin{definition}
For $N \in \bbN$ and $\theta \in (1,\infty)$, the stable sampling rate is given by
\begin{equation*}
\Theta(N;\theta) = \min \left \{ M \in \bbN : \frac{1}{C_{N,M}} < \theta \right \}.
\end{equation*}
\end{definition}
This notion of the stable sampling rate is important as it determines the number of samples required for guaranteed, quasi-optimal and numerically stable reconstructions. In particular, for all $M \geq \Theta(N;\theta)$, we have that $\tilde{f}_{N,M}$ is quasi-optimal to $f$ from $\mathcal{T}_N$ with constant a most $\theta$, and the condition number $\kappa(F_{N,M})$ is at worst $\theta$.

\subsection{Optimality of generalized sampling}\label{ss:optimal}
In \cite{BAACHOptimality} the question of optimality of generalized sampling was also discussed.  We now recap the main results proved, since they will be of use later.  We first recall the definition of a perfect method:
\begin{definition}
Let $G_{N,M} : \cH \rightarrow \cT_N$ be a mapping such that, for each $f \in \mathcal{H}$, $G_{N,M}(f)$ depends only on the samples $\{ \hat{f}_j \}^{M}_{j=1}$. If $G_{N,M}(f)=f$ for all $f\in\mathcal{T}_N$, then $G_{N,M}$ is said to be perfect.
\end{definition}
Observe that the notion of perfectness is strictly weaker than quasi-optimality.  Also, we remark that generalized sampling is a perfect method, as can be seen from (\ref{eq:gensamp_soln}).

The first result of \cite{BAACHOptimality} concerns such methods:
\begin{theorem}\label{t:perfectstability}
For $M \geq N $ let $G_{N,M} :\mathcal{H}\to\mathcal{T}_N$ be a perfect reconstruction method such that, for each $f \in \mathcal{H}$, $G_{N,M}(f)$ depends only on the samples $\{ \hat{f}_j \}^{M}_{j=1}$.  Then the condition number
\begin{equation*}
\kappa(G_{N,M}) \geq \kappa(F_{N,M}),
\end{equation*}
where $F_{N,M}$ is the generalized sampling reconstruction.
\end{theorem}
This result implies the following: for any perfect reconstruction method, one must sample at a rate higher than that of generalized sampling -- namely, the stable sampling rate -- to obtain a stable reconstruction.  In other words, generalized sampling cannot be improved upon in terms of its stability (at least for perfect methods).

The case of non-perfect methods was also studied in \cite{BAACHOptimality}.  The following result was proved:
\begin{theorem}\label{t:optimalitylinearSSR}
Suppose that the stable sampling rate $\Theta(N;\theta)$ is linear in $N$ for a particular sampling and reconstruction problem.  Let $f \in \mathcal{H}$ be fixed, and suppose that there exists a sequence of mappings
$$
G_M : \{ \hat{f}_j \}^{M}_{j=1} \mapsto G_M(f) \in \mathcal{T}_{\Psi_f(M)},
$$
where $\Psi_f : \bbN \rightarrow \bbN$ with $\Psi_f(M) \leq c M$.   Suppose also that there exist constants $c_1(f),c_2(f),\alpha_f > 0$ such that
\begin{equation}
\label{algconv}
c_1(f) N^{-\alpha_f} \leq \| f - Q_N f \| \leq c_2(f) N^{-\alpha_f},\quad \forall N \in \bbN.
\end{equation}
Then, given $\theta \in (1,\infty)$, there exist constants $c(\theta) \in (0,1)$ and $c_f(\theta) > 0$ such that
\begin{equation}
\label{decay}
\| f - F_{c(\theta) M,M}(f) \| \leq c_f(\theta) \| f - G_M(f) \|,\quad \forall M \in \bbN,
\end{equation}
where $F_{N,M}$ is the generalized sampling reconstruction.
\end{theorem}

This theorem demonstrates that for problems with linear stable sampling rates, even if one is allowed to design a method that depends on $f$ in a completely non-trivial way, it is still not possible to obtain a faster asymptotic rate of convergence than that of generalized sampling.  As we explain in Section \ref{s:main}, the stable sampling rate is linear for wavelets, making this theorem directly applicable.

Observe that a consequence of this theorem is that generalized sampling is, up to a constant, an oracle for the wavelet coefficient reconstruction problem.  Suppose there was some method that, for a particular $f$ satisfying (\ref{algconv}), could recover the first $N=M$ wavelet coefficients of $f$ \textit{exactly} (i.e. with no error) from $M$ Fourier samples.  The conclusion of the above corollary is that generalized sampling commits an error that is at worst a constant factor larger than that of this method.

\subsection{The wavelet reconstruction and Fourier sampling spaces}\label{sec:setup}
In the remainder of this paper we focus on the problem of recovering wavelet coefficients from Fourier samples.  To this end, we now specify the corresponding sampling space $\mathcal{S}$, with its corresponding sampling vectors $\{s_j\}_{j\in \mathbb{N}}$, as well as the reconstruction space $\mathcal{T}$ with the reconstruction vectors $\{\varphi_j\}_{j\in\mathbb{N}}$.  Throughout we let $\mathcal{H} = L^2(\mathbb{R})$ with its usual inner product and will consider the recovery of functions in $\cH$ that are compactly supported on $[0,a]$ for some $a\geq1$.

\subsubsection{The wavelet reconstruction space}\label{sec:setup:recostruction}
The results of this paper are for the case where the reconstruction space is generated by compactly supported Multiresolution Analysis (MRA) wavelets in the sense that the wavelet $\psi$ is associated with an MRA generated by the scaling function $\phi$ such that the following holds:
\begin{enumerate}
\item[(i)] $\br{V_j: j\in\bbZ}$ is a system of nested closed subspaces of $L^2(\bbR)$ with  $V_j \subset V_{j+1}$,
\item[(ii)] $f\in V_j$ if and only if $f(2\cdot ) \in V_{j+1}$,
\item[(iii)] $\bigcap_{j\in\bbZ} V_j = \br{0}$,
\item[(iv)] $\overline{\bigcup_{j\in\bbZ} V_{j}}  = L^2(\bbR)$,
\item[(v)] $\br{\phi(\cdot -k): k\in\bbZ}$ is an orthonormal basis for $V_0$.
\end{enumerate}
However, note that there is a more general notion of an MRA, where condition (v) is replaced with
\begin{enumerate}
\item[(v') ]$\br{\phi(\cdot -k): k\in\bbZ}$ is a Riesz basis for $V_0$.
\end{enumerate}  Furthermore, it can be shown that this weaker notion is equivalent to the assumption of orthonomality \cite[p.44]{eugenio1996first}. Consequently, all the results of this paper can be generalized to compactly supported MRA wavelet systems with the Riesz basis assumption only. In particular, this would include spline wavelets such as the semi-orthogonal wavelets \cite{chui1992compactly, unser1992asymptotic} and the Cohen-Daubechies-Feauveau bi-orthogonal wavelets \cite{cohen2006biorthogonal}.

Suppose now that the reconstruction space $\mathcal{T}$ is generated by a mother wavelet $\psi$ and a scaling function $\phi$ such that $\mathrm{supp}(\psi)=\mathrm{supp}(\phi)=[0,a]$.  Then the only wavelets of interest are those whose support intersects $[0,a]$. In particular, for 
\begin{align*}
 \phi_{j,k} &= 2^{j/2}\phi(2^j\cdot -k), \qquad j,k\in\mathbb{Z},\\
\psi_{j,k} &= 2^{j/2}\psi(2^j\cdot -k), \qquad j,k\in\mathbb{Z},
\end{align*}
the wavelets of interest are
\begin{align*}
\Omega_a = \br{\phi_{0,k} : \abs{k} = 0,1,\ldots,\left\lceil a\right\rceil -1}\cup\br{\psi_{j,k} : j\in\mathbb{Z}+, k\in \mathbb{Z},-\left\lceil a\right\rceil +1\leq k \leq 2^j \left\lceil a\right\rceil -1}.
\end{align*}
So, 
$$
\mathcal{T} =\overline{\mathrm{span}\br{\varphi :\varphi\in\Omega_a}},
$$
and for sufficiently large $T_1$ and $T_2$, namely, $T_1\geq \left\lceil a\right\rceil -1$ and $T_2\geq 2\left\lceil a\right\rceil -1$:
\begin{align}\label{eq:Omega_a}
L^2[0,a]\subset \mathcal{T}\subset L^2[-T_1,T_2].
\end{align}
Elements of $\Omega_a$ are ordered as follows:
\begin{equation}\label{ordering}
\{\varphi_j\}_{j\in\mathbb{N}} = \{\phi_{0,-\left\lceil a\right\rceil+1}, \phi_{0,-\left\lceil a\right\rceil+2},\ldots, \phi_{0,\left\lceil a\right\rceil-1}, \psi_{0,-\left\lceil a\right\rceil+1},\ldots, \psi_{0,\left\lceil a\right\rceil-1}, \psi_{1,-\left\lceil a\right\rceil+1}, \ldots, \psi_{1,2^j \left\lceil a\right\rceil -1},\ldots\},
\end{equation}
and thus 
\begin{equation}\label{ordering2}
\mathcal{T}_{N} = \mathrm{span}\br{\varphi_j : j=1,\ldots,N}.
\end{equation}
Although one can in principle consider arbitrary values of $N \in \bbN$, it is natural instead to consider only those $N$ for which $\mathcal{T}_N$ contains all wavelets up to a certain scale.  To this end, we now write
\begin{equation}\label{NRdef}
N_R = 2^R \left\lceil a\right\rceil + (R+1)(\left\lceil a\right\rceil -1),\quad R \in \bbN.
\end{equation}
We will verify in Lemma \ref{lem:decomposition} that the subspace $\mathcal{T}_{N_R}$ consists of all wavelets $\psi_{j,k}$ of scale $0 \leq j \leq R-1$.

\subsubsection{The Fourier sampling space}
Given $[-T_1,T_2]$ (the support of $\cT$), we let $\epsilon \leq 1/(T_1+T_2)$ be the \textit{sampling density} (or \textit{sampling distance}).  Note that via the Nyquist-Shannon sampling theorem, $1/(T_1+T_2)$ is the corresponding Nyquist criterion for functions supported on $[-T_1,T_2]$. We now define the sampling vectors by
$$
s_l^\epsilon = \sqrt{\epsilon}e^{2\pi i l \epsilon \cdot}\chi_{[-T_1/(\epsilon(T_1 +T_2)),T_2/(\epsilon(T_1 +T_2))]},
$$
and the sampling space by 
\begin{align*} 
\mathcal{S}^\epsilon= \overline{\mathrm{span}\br{s_l^\epsilon : l\in\mathbb{Z}}} = \left \{ f \in L^2(\mathbb{R}) : \mathrm{supp}(f) \subseteq [-T_1/(\epsilon(T_1 +T_2)),T_2/(\epsilon(T_1 +T_2))] \right \},
\end{align*}
and the space spanned by the first $M$ sampling vectors by
\begin{align}\label{eq:S_m}
\mathcal{S}_{M}^\epsilon = \mathrm{span}\br{s_l^\epsilon :-\left\lfloor\dfrac{M}{2}\right\rfloor \leq l\leq \left\lceil\dfrac{M}{2}\right\rceil -1}.
\end{align}
Moreover, $P^\epsilon$ and $P^\epsilon_M$ will denote the orthogonal projections from $\cH$ onto $\cS^\epsilon$ and $\cS^\epsilon_M$ respectively.
Where there is no ambiguity about the value of the sampling density, we will drop the $\epsilon$ notation and simply write $\cS$, $\cS_M$, $P$ and $P_M$ instead.

\begin{remark}
Observe that for all $\epsilon\leq 1/(T_1+T_2)$, $\cT\subset\cS^\epsilon$. So, $\cT+(\cS^\epsilon)^\perp$ is a closed subspace of $\cH$ and $\cT\cap(\cS^\epsilon)^\perp=\br{0}$. Thus, the subspace condition (\ref{subspacecondition}) of generalized sampling is satisfied.
\end{remark}

\section{Examples}\label{s:examples}
In this section, we present two examples to illustrate the use generalized sampling in practice and its advantage is clear from the figures corresponding to our examples.
Recalling the summary of generalized sampling in Section \ref{s:GS}, an effective reconstruction requires choosing the correct ratio between $M$, the number of samples  and $N$, the number of reconstruction vectors to be approximated. In particular, the choice of $M$ and $N$ is in accordance with the stable sampling rate and  the generalized sampling reconstruction is the unique least squares solution to (\ref{least_sq_formulation}) with representation (\ref{eq:gs_exact_rep}).

\begin{figure}
\centering
$\begin{array}{ccc}
\includegraphics[ trim=0.8cm 0.2cm 0.8cm 0.2cm,clip=true,width=0.3\textwidth]{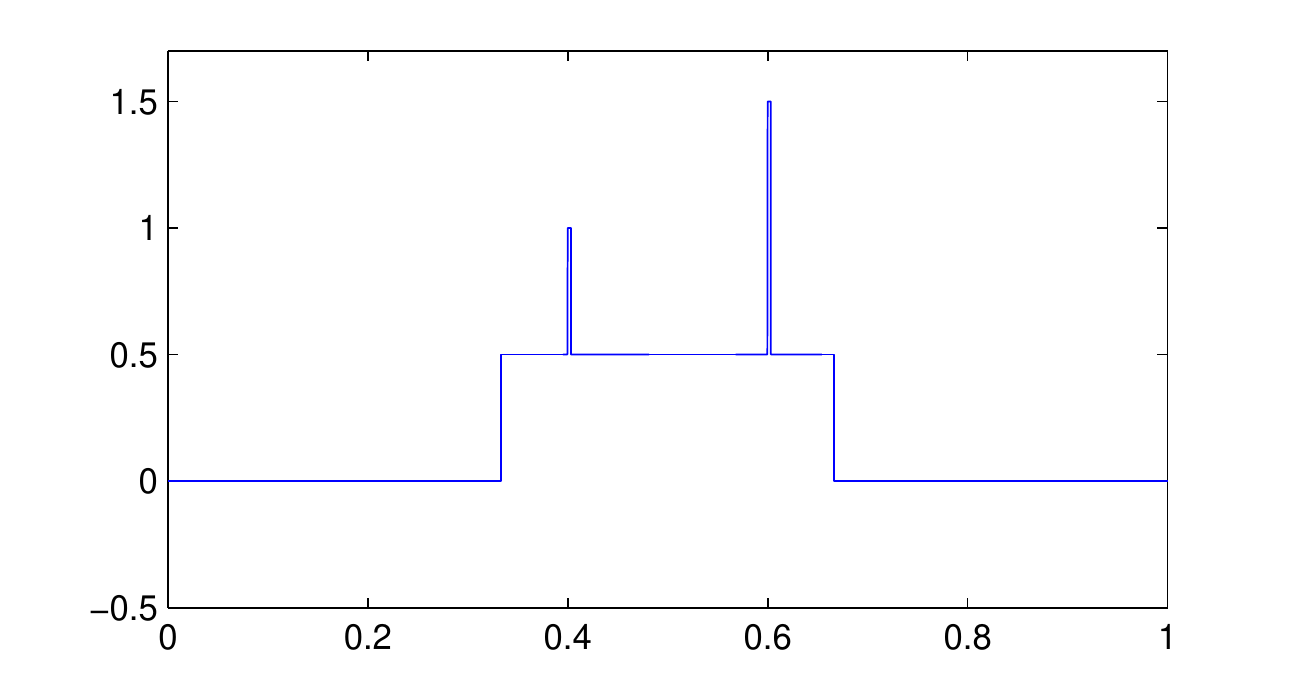} &
\includegraphics[ trim=0.8cm 0.2cm 1cm 0.2cm,clip=true,width=0.3\textwidth]{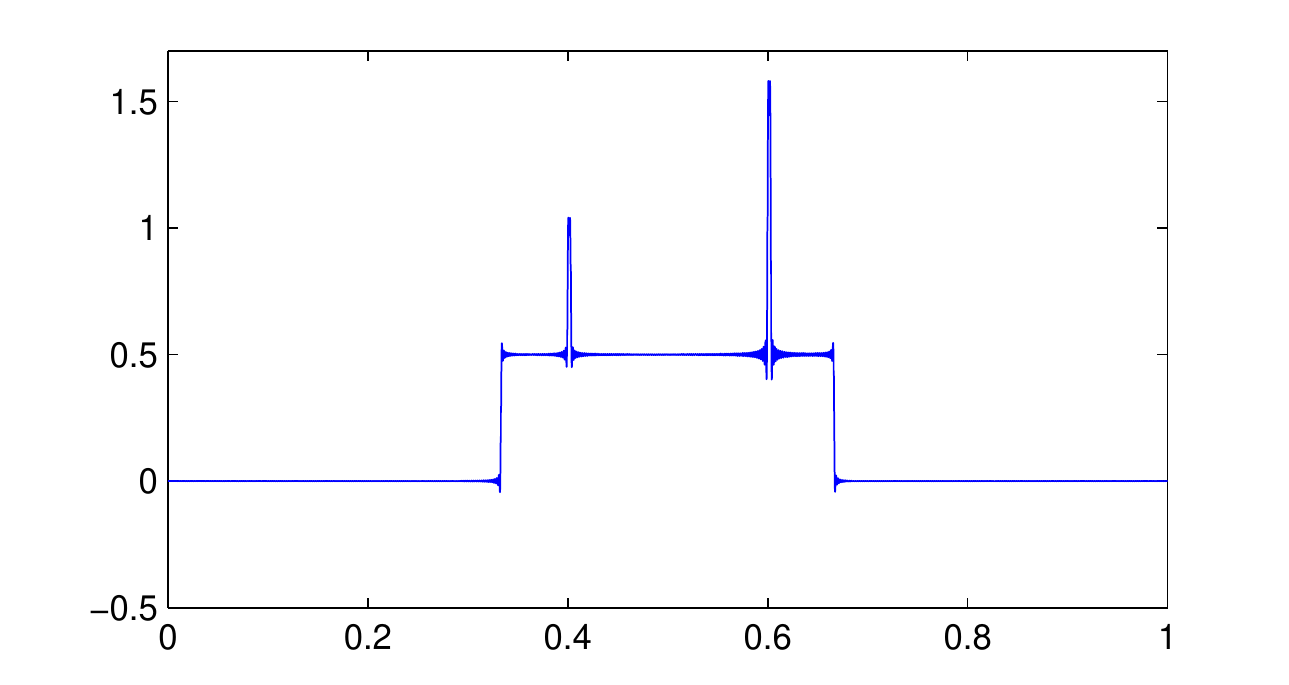} &
\includegraphics[ trim=0.8cm 0.2cm 1cm 0.2cm,clip=true,width=0.3\textwidth]{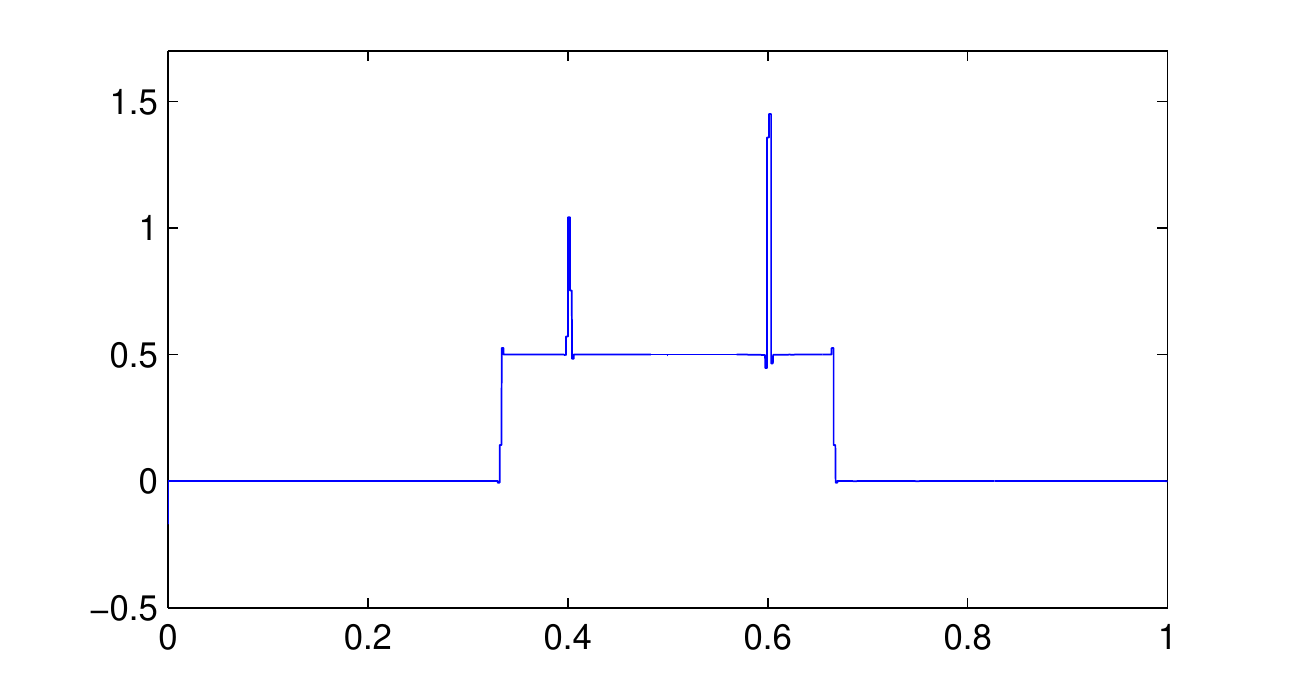} \\
\includegraphics[ trim=0.8cm 0.2cm 0.8cm 0.2cm,clip=true,width=0.3\textwidth]{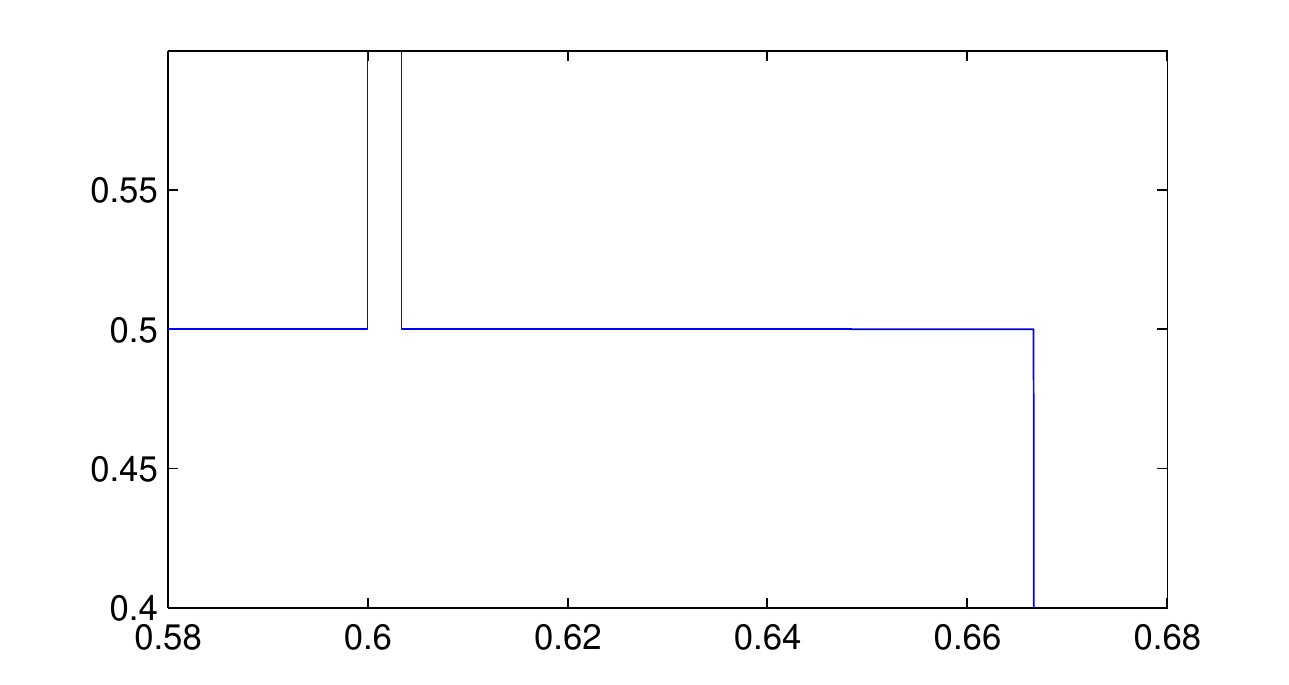} &
\includegraphics[ trim=0.8cm 0.2cm 0.8cm 0.2cm,clip=true,width=0.3\textwidth]{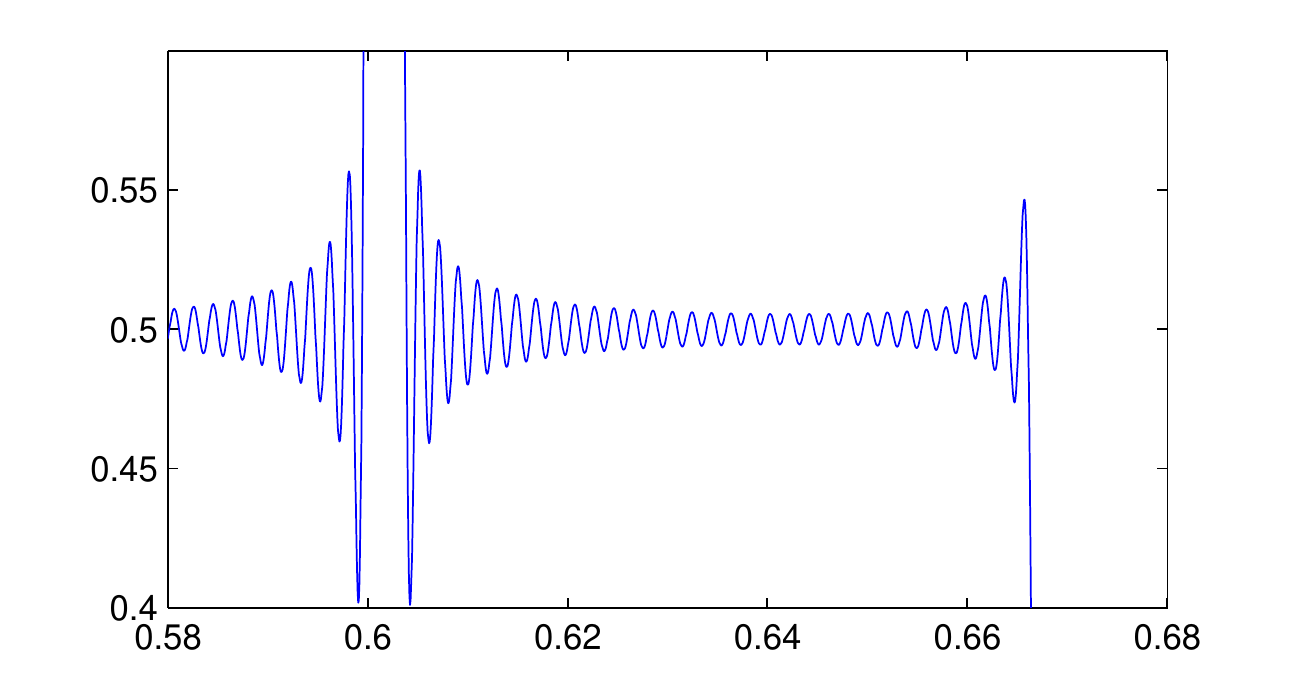} &
\includegraphics[ trim=0.8cm 0.2cm 0.8cm 0.2cm,clip=true,width=0.3\textwidth]{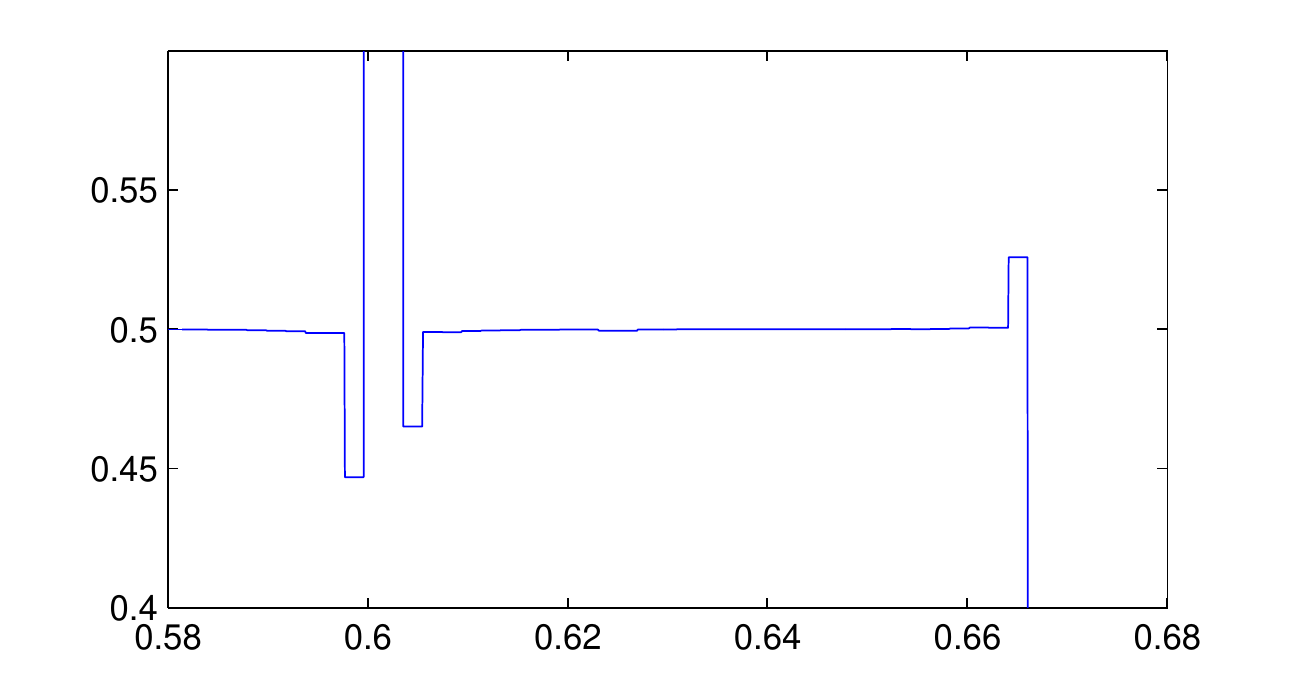}
\end{array}$
\caption{The top row shows $f$ (left), $f_M$ (middle) and $f^{[N,M]}$ (right). The bottom row shows $f$ (left), $f_M$ (middle) and $f^{[N,M]}$ (right) on the interval $[0.58, 0.68]$. 
\label{fig:comparison}}
\end{figure}

\begin{figure}
\centering
$\begin{array}{ccc}
\includegraphics[ trim=0.8cm 0.2cm 0.8cm 0.2cm,clip=true, width=0.3\textwidth]{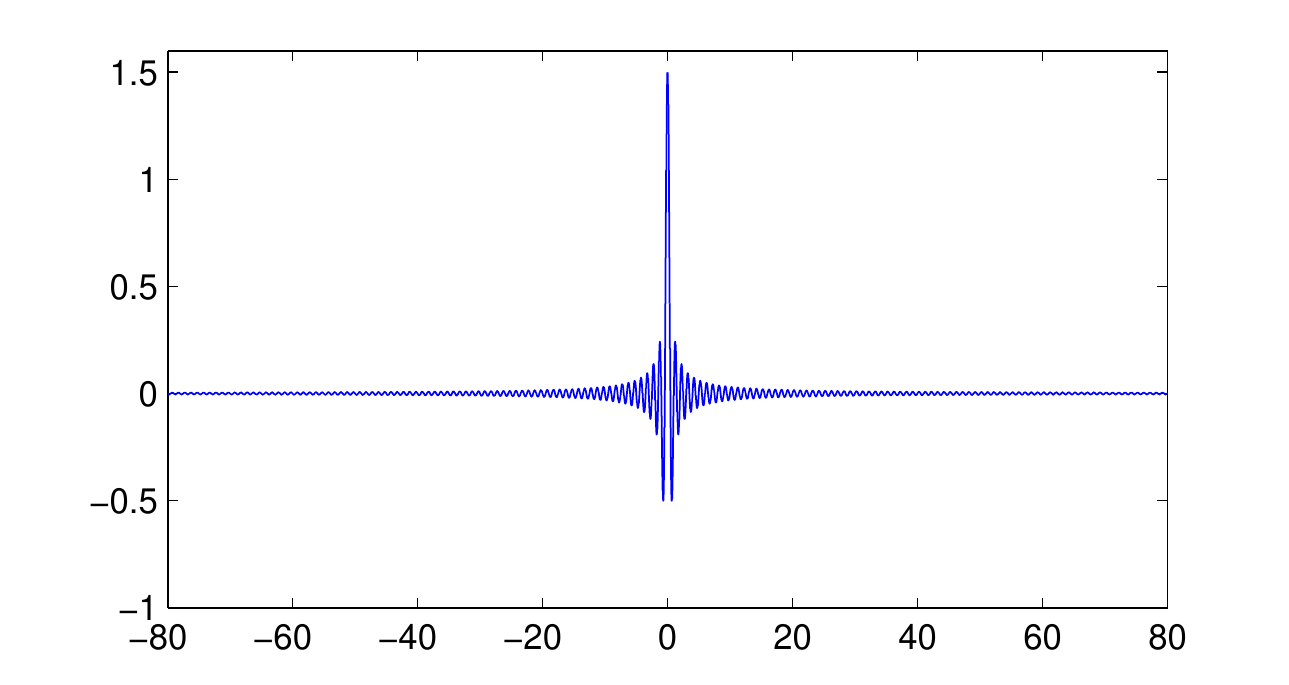} &
\includegraphics[  trim=0.8cm 0.2cm 0.8cm 0.2cm,clip=true, width=0.3\textwidth]{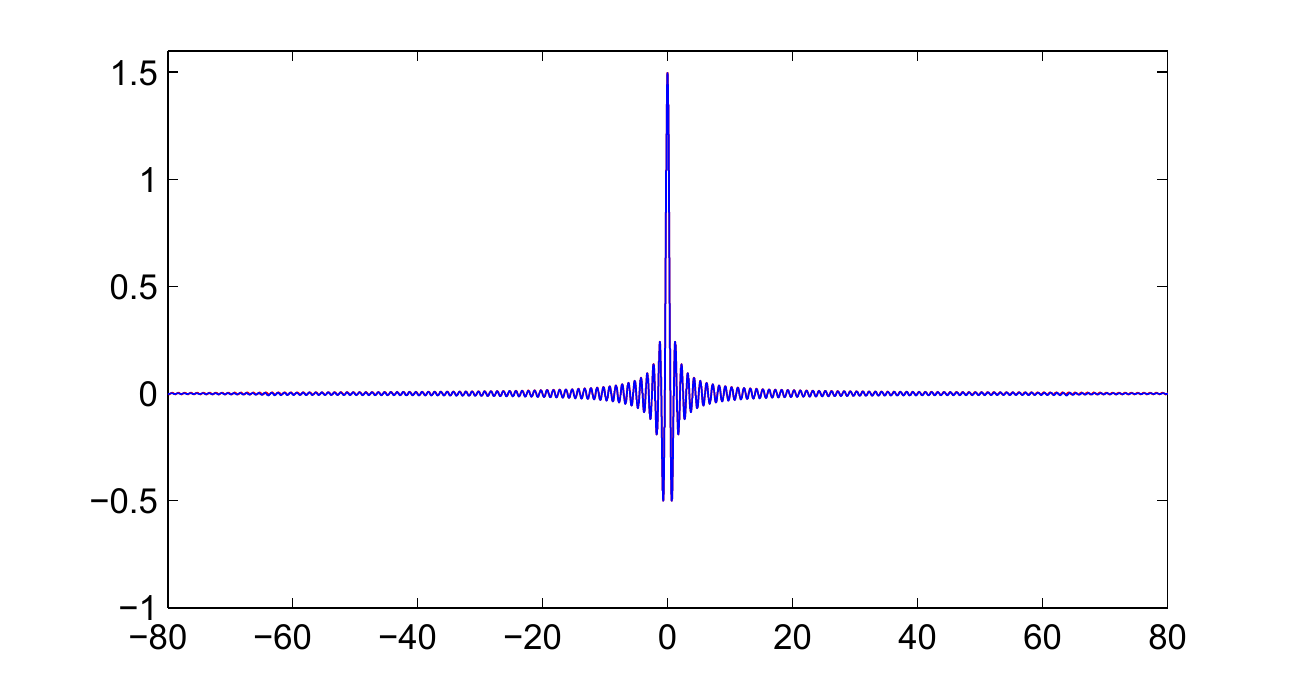} &
\includegraphics[ trim=0.8cm 0.2cm 0.8cm 0.2cm,clip=true,   width=0.3\textwidth]{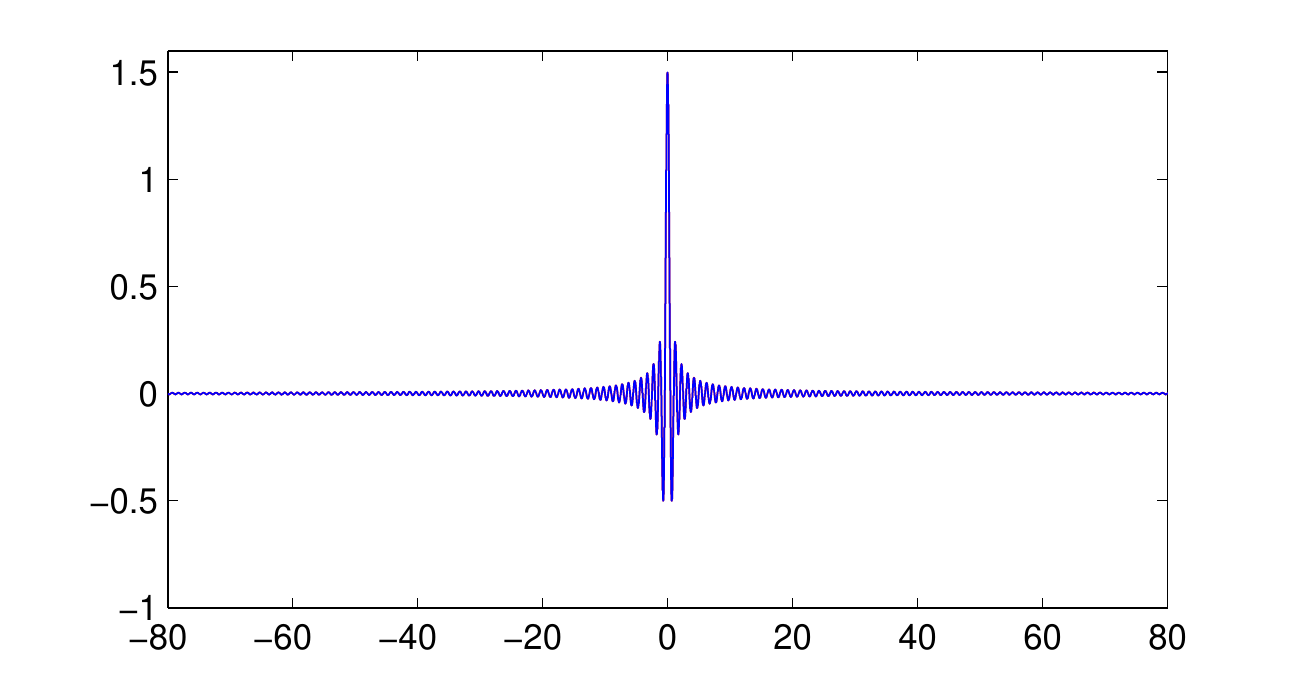} \\
\includegraphics[  trim=0.8cm 0.2cm 0.8cm 0.2cm,clip=true, width=0.3\textwidth]{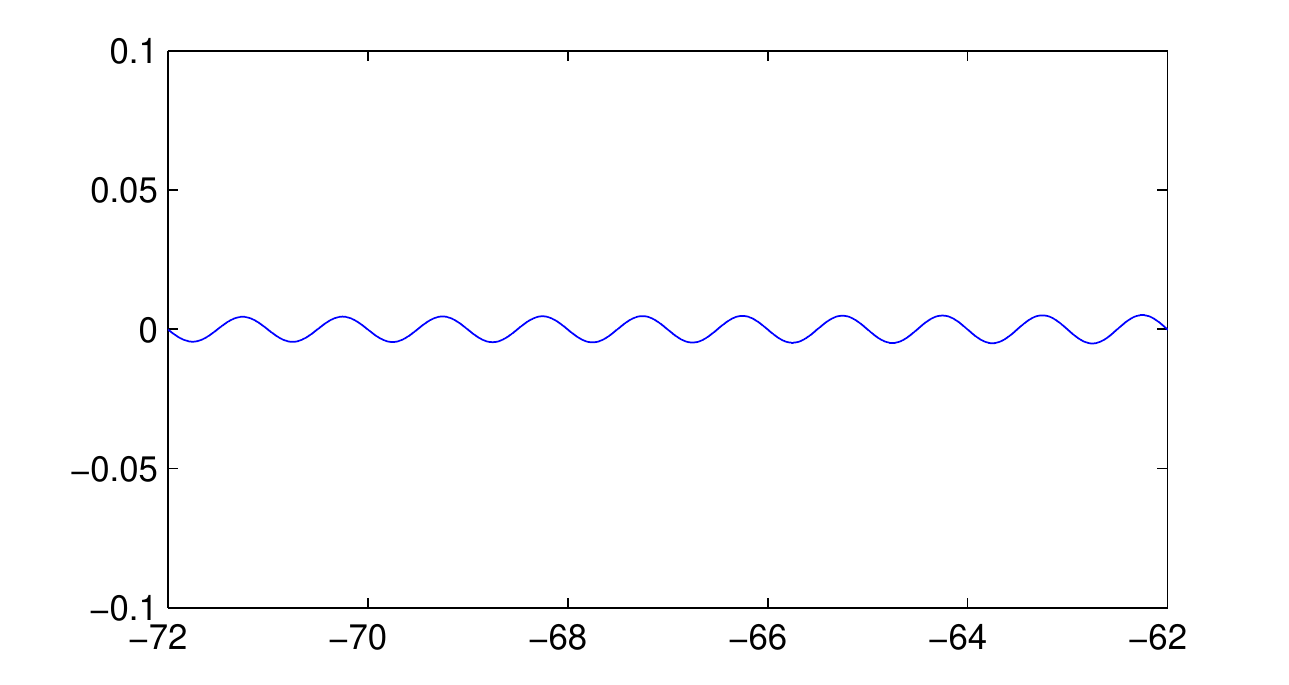} &
\includegraphics[  trim=0.8cm 0.2cm 0.8cm 0.2cm,clip=true, width=0.3\textwidth]{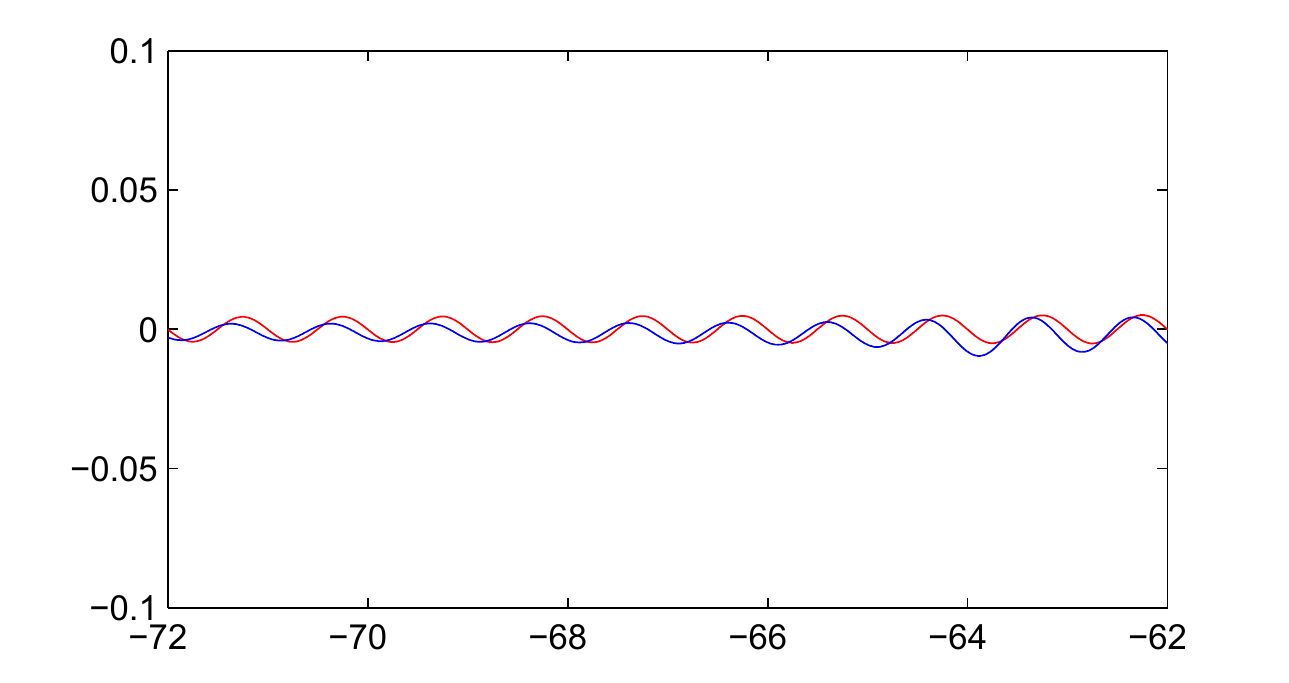} &
\includegraphics[  trim=0.8cm 0.2cm 0.8cm 0.2cm,clip=true, width=0.3\textwidth]{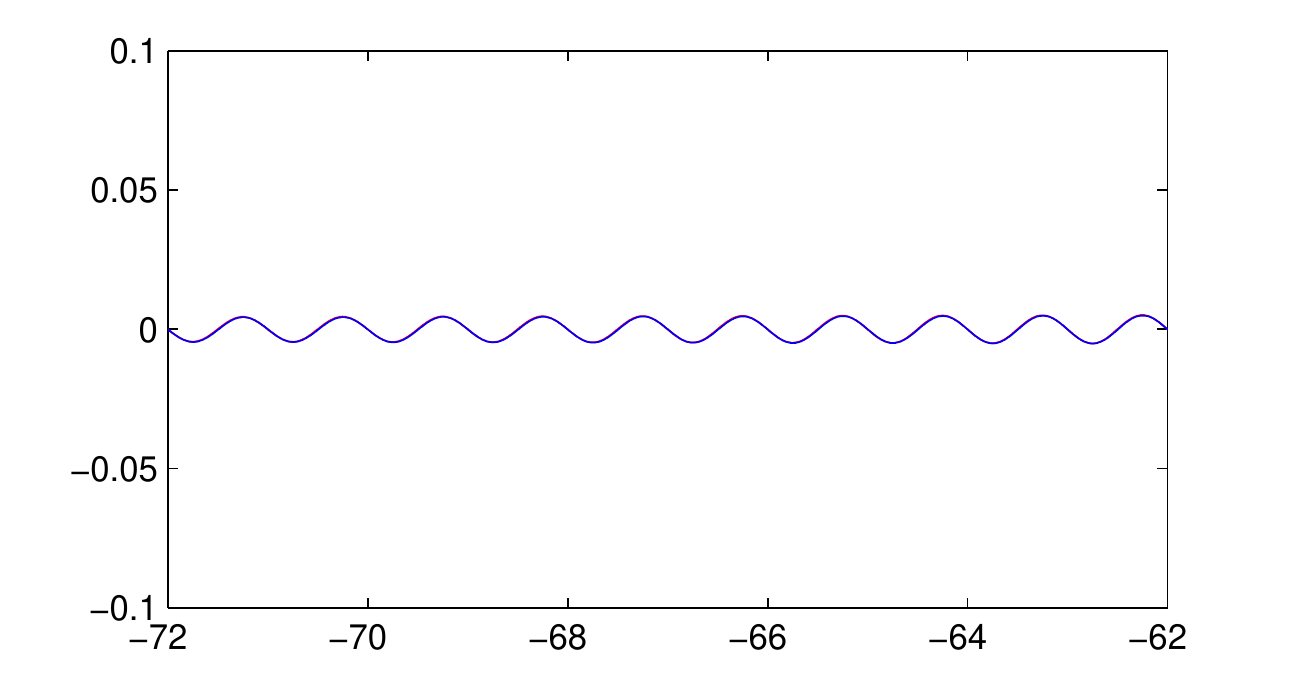} 
\end{array}$
\caption{The top row shows $f$ (left), $f_M$ in blue and $f$ in red (middle) and $f^{[N,M]}$  in blue and $f$  in red  (right). The bottom row shows $f$ (left),  $f_M$ in blue and $f$ in red (middle) and $f^{[N,M]}$  in blue and $f$  in red (right) on the interval $[-72, -62]$.
\label{fig:comparison_sinc}}
\end{figure}

\begin{enumerate}
\item \textbf{Reconstruction using Haar wavelets:}
Let us first consider the reconstruction of the function
$$
f(x) = \frac{1}{2}\chi_{[1/3, 2/3]} + \frac{1}{2} \chi_{[2/5,2/5+1/300]} +  \chi_{[3/5,3/5+1/300]}
$$
from the finite vector of measurements
$$
\mathcal{F}_M = \left(\hat{f}(-\pi M/2), \hat{f}(-\pi (M/2-1) ), \hdots, \hat f(\pi (M/2-1)), \hat f(\pi M/2 ) \right), \quad M=2048.
$$
By the Nyquist-Shannon Sampling Theorem, we may directly approximate $f$ by its truncated Fourier series $f_N$ as follows
$$
 f_M =  \frac{1}{2} \sum_{n=-M/2}^{M/2} \hat f(\pi n ) e^{\pi i n  \cdot}.
$$
In this example, we contrast the generalized sampling reconstruction in Haar wavelets with the direct approximation $f_M$ from the \emph{same} measurements $\mathcal F_M$. 
Recall that the Haar scaling function and wavelet are defined by
\begin{align*}
\phi = \chi_{[0,1)}, \qquad \psi = \chi_{[0,1/2)} - \chi_{[1/2,1)}
\end{align*} 
and applying the construction of Section \ref{sec:setup}, the reconstruction space is $\cT =\overline{\mathrm{span} \br{\varphi\in\Omega_1}}$ where $$
\Omega_1 = \br{\phi} \cup \br{\psi_{j,k}: j\in\mathbb{Z}_+, 0\leq k\leq 2^j-1}.
$$
By implementing generalized sampling with $N=512$ such that $\Theta(N; 1.2)\leq 2048$, the generalized sampling reconstruction $f^{[N,M]}$ is,  up to a factor of $1.2$, the best approximation from the first $512$ wavelets.
 The exact function $f$, the truncate Fourier series approximation $f_M$ and the generalized sampling solution $f^{[N,M]}$ are displayed in Figure \ref{fig:comparison}. We remark that similar figures were generated in \cite{WeaverEtAlWaveletEncoding} to justify the use of wavelet encoding for MRI and in proving that the stable sampling rate is linear, we show that that wavelet coefficients can be accurately approximated via post-processing and there is little to be gained in modifying the sampling process.

\item \textbf{Reconstruction of a bandlimited function:}
Consider $$f(x)= \frac{e^{-ix}+x(2 i e^{-ix}-1)}{x^2},$$ then $\hat f(x) = (x+1)\chi_{[0,1]}(x)$ and by the Nyquist-Shannon Sampling Theorem, we may approximate $f$ directly from its pointwise samples
$$\mathcal{F}_{M} = \left({f}(-\pi M/2 ), {f}(-\pi (M/2-1) ), \hdots,  f(\pi (M/2-1)),  f(\pi M/2 ) \right), \quad M=512
$$
with
$$
 f_M =  \sum_{n=-M/2}^{M/2} f(\pi n ) \mathrm{sinc}(\cdot + \pi  n ).
$$
By applying the generalized sampling scheme with $N=128$, we obtain a Haar wavelet reconstruction of $\hat f$ which we denote by $\hat f^{[N,M]}$. Then, taking the inverse Fourier transform of $\hat f^{[N,M]} $ gives the  reconstruction $f^{[N,M]}$. The graphs of $f$, $f_M$ and $f^{[N,M]}$ are displayed in Figure \ref{fig:comparison_sinc}.
\end{enumerate}

\section{Main results}\label{s:main}
We now state the main results of this paper.  Proofs are provided in Sections \ref{s:prf1} and \ref{s:prf2}.

\subsection{Linearity of the stable sampling rate}
The first result of this paper is that the stable sampling rate for wavelet reconstructions from Fourier samples is linear for any compactly supported MRA wavelet basis.  In other words, up to a constant factor there is a one-to-one correspondence between Fourier samples and wavelet coefficients.  In particular, all information about a function that can be retrieved from its wavelet coefficients can still be retrieved even in the situation where only Fourier samples are available (and not the wavelet coefficients themselves).

More formally, we have the following theorem:
\begin{theorem}\label{them;main}
Let $\mathcal{S}$ and $\mathcal{T}$ be the sampling and reconstruction spaces defined in Section \ref{sec:setup} and recall $N_R$ from (\ref{NRdef}). Let $N\leq N_R$ for $R\in\mathbb{N}$.  Then for all $\theta \in (1,\infty)$ there exists $S_{\theta}\in\mathbb{N}$, independent of $R$, such that for 
$
M=\left\lceil\dfrac{S_{\theta}2^{R+1}}{\epsilon}\right\rceil,
$
we have 
$
C_{N,M} \geq \dfrac{1}{\theta}.
$
In particular, 
$
\Theta(N; \theta)\leq \left\lceil\dfrac{2S_{\theta} N}{\epsilon\lceil a\rceil}\right\rceil$. Hence, $\Theta(N;\theta) = \mathcal{O}(N)$ for any $\theta \in (1,\infty)$.
\end{theorem}
Since the stable sampling rate is linear for wavelets, it makes sense to introduce the notion of a \textit{stable sampling ratio}.  We define
\begin{equation}\label{eq;sampling_ratio}
\eta(\theta) = \limsup_{N \rightarrow \infty} \frac{\Theta(N;\theta)}{N},\qquad \theta \in (1,\infty).
\end{equation}
Note the difference between $\Theta(N;\theta)$, which determines how many samples are required for each $N$, and $\eta(\theta)$, which stipulates asymptotically how many are required as $N \rightarrow \infty$.  We will also discuss sampling ratios in the context of other methods.  To this end, suppose that an arbitrary method $G$ uses $\Theta_G(N) \in \bbN$ samples to reconstruct the first $N$ wavelet coefficients.  We define the sampling ratio for that method as
$$
\eta_G = \limsup_{N \rightarrow \infty} \frac{\Theta(N)}{N}.
$$
Since the stable sampling rate is linear, we shall only consider methods $G$ for which $\eta_G$ is defined (all other methods necessarily give worse reconstructions asymptotically as $N \rightarrow \infty$).

\subsection{Universality of the stable sampling rate}

The second collection of results concerns the universality of the stable sampling rate, or equivalently, the optimality of generalized sampling amongst all methods which recover $N$ wavelet coefficients from $M \geq N$ Fourier samples.  Our first result is simply a corollary of Theorem \ref{t:perfectstability} for the wavelet reconstruction problem from Fourier samples:

\begin{corollary}
\label{c:universal}
For $N \in \bbN$, let $G_N$ be a sequence of perfect reconstruction methods with sampling ratio $\eta_G \geq 1$.  If $\eta_G$ is such that $\kappa(G_{N}) \leq \theta$ for some $\theta \in (1,\infty)$ and all sufficiently large $N$, then $\eta_G \geq \eta(\theta)$, where $\eta(\theta)$ is the stable sampling ratio for generalized sampling.
\end{corollary}

This corollary states that, for any perfect method, the stable sampling ratio $\eta(\theta)$ cannot be lowered.  In particular, any perfect method requires at least the same number of Fourier samples to achieve as stable a reconstruction as that of generalized sampling.

Despite this result, in some cases it might seemingly be acceptable to forgo complete stability to obtain a better reconstruction.  Our next theorem, which is specific to the wavelet reconstruction problem, shows that this cannot be done in practice:

\begin{theorem}
\label{thrm:main2}
Let $G_{N}$ be as in Corollary \ref{c:universal} with sampling rate $\eta_G \geq 1$.  If $\eta_G < \frac{1}{\epsilon \lceil a \rceil}$, where $\epsilon$ is as in Section \ref{sec:setup},  then $\kappa(G_{N})$ is unbounded and $\kappa(G_{N_R})$ becomes exponentially large as $N_R\rightarrow \infty$, where $N_R$ is as defined in (\ref{NRdef}).
\end{theorem}

This theorem demonstrates that any attempt to improve upon generalized sampling by lowering the sampling rate will result in extremely poor stability, and consequently extreme sensitivity to noise and round-off error.  Prior to this result, one may have hoped that sampling below the critical threshold $\eta = \frac{1}{\epsilon \lceil a \rceil}$ might only result in mildly growing condition numbers.  This theorem demonstrates that this is not the case: stability rapidly declines dramatically once $\eta < \frac{1}{\epsilon \lceil a \rceil}$. 

Corollary \ref{c:universal} and Theorem \ref{thrm:main2} establish the \textit{universality} of the stable sampling rate, and the pitfalls of trying to circumvent the stability barrier $\eta \geq \frac{1}{\epsilon \lceil a \rceil }$.  However, they are valid only for perfect methods.  Recall that the question of non-perfect methods was addressed by Theorem \ref{t:optimalitylinearSSR}.  In terms of the sampling ratio, this implies that any non-perfect method which has a lower sampling ratio for a particular function $f$ satisfying (\ref{algconv}) can only outperform generalized sampling by a constant factor.  Note that the problem of recovering wavelet coefficients from Fourier samples certainly satisfies the assumptions of Theorem \ref{t:optimalitylinearSSR}:  as we prove, the stable sampling rate is linear, and for typical functions $f$, it is usually the case that the wavelet coefficients decay algebraically (which implies (\ref{algconv})).

\subsection{Sharp results for the Daubechies wavelets}
Although Theorem \ref{them;main} establishes linearity of the stable sampling rate for any compactly supported MRA wavelet  basis, it does not provide the precise constant of proportionality.  Nor is it straightforward to determine an upper bound, since the quantity $S_{\theta}$ is not given explicitly.  Although one can in theory estimate $S_{\theta}$ by carefully following the steps of the proof, we shall not do this.  Instead, in this section we show that for the important case of Daubechies wavelets the constant can be determined exactly.

\begin{remark}
The fact that the constant may not be known in general does not necessary prohibit implementation of generalized sampling.  As discussed in \cite{BAACHOptimality}, the stable sampling rate is explicitly computable, and thus the constant can actually be determined  \textit{a priori} for each particular case through numerical means.
\end{remark}

Our main result is as follows:

\begin{theorem}\label{prop:daubechies_case}
Let $\mathcal{S}$ and $\mathcal{T}$ be the sampling and reconstruction spaces defined in Section \ref{sec:setup}, where $\mathcal{T}$ is generated by a Daubechies wavelet, and recall $N_R$ from (\ref{NRdef}). Then, there exists $\theta\in (1,\infty)$ and $R_0\in\mathbb{N}$ such that for all $R\geq R_0$, the stable sampling rate is  
$
\Theta(N_R; \theta) = \left\lceil 2^R/\epsilon\right\rceil$. 
In particular, when  $1/\epsilon\in\mathbb{Z}$ it suffices to let 
$
\theta>\left(\inf_{\xi\in [-\pi,\pi]}\abs{\hat{\phi}(\xi)}\right)^{-1}
$. 
Moreover, in addition to this, for Haar wavelets, where $a=1$, we have that $\Theta(N_R; \theta) \leq \left\lceil 2^R/\epsilon\right\rceil$ for all $R\in\mathbb{N}$.
\end{theorem}

\begin{remark}
Note also that for such values of $\theta$ and $R$ in Theorem \ref{prop:daubechies_case}, if $N$ is such that $N_{R-1}+1\leq N\leq N_R$, then
$\Theta(N;\theta)\leq \left\lceil 2^R/\epsilon\right\rceil
$.
Therefore, we have that
$$
\dfrac{1}{\epsilon\lceil a\rceil}\leq \eta(\theta) \leq \lim_{R\to\infty} \frac{\left\lceil 2^R/\epsilon\right\rceil}{N_{R-1}+1} = \dfrac{2}{\epsilon\lceil a\rceil}.
$$ 
However, our numerical results in Section \ref{s:numexp} suggest that the optimal ratio is $(\epsilon\lceil a\rceil)^{-1}$ and is attained only when $N=N_R$.
\end{remark}

\section{Proof of Theorem \ref{them;main}}\label{s:prf1}
The proof of Theorem \ref{them;main} requires a series of lemmas and propositions that will be presented below. The actual proof can be found at the very end of this section.

\subsection{Expressing wavelets in terms of the scaling function}
We will demonstrate in this section that due to standard MRA properties, given any $N\in\mathbb{N}$, all basis elements of $\mathcal{T}_N$ may be expressed as a linear combination of finitely many basis elements of $\br{\phi_{R,k}: k \in \mathbb{Z}}$ for some $R\in\mathbb{N}$.
Let therefore, for $j \in \mathbb{Z}_+,$
\begin{align*}
V_j &= \overline{\mathrm{span}\br{\phi_{j,k}: k \in \mathbb{Z}}},\\
W_j &= \overline{\mathrm{span}\br{\psi_{j,k}: k \in \mathbb{Z}}},\\
V^{(a)}_0 &= \mathrm{span}\br{\phi_{0,k}: k \in \mathbb{Z}, \abs{k}\leq \left\lceil a\right\rceil -1},\\
W^{(a)}_j &= \mathrm{span}\br{\psi_{j,k}: k \in \mathbb{Z}, -\left\lceil a\right\rceil +1\leq k \leq 2^j \left\lceil a\right\rceil -1}.\\
\end{align*}
The following lemma relates $\mathcal{T}_N$ to the two latter types of subspaces.

\begin{lemma}\label{lem:decomposition}
For $R \in \mathbb{N}$, let 
\begin{equation}\label{eq:A_R}
\begin{split}
A_{R,1}&=-(2^R+1)\left\lceil a\right\rceil+2^R+1,\\
A_{R,2}&= 2^{R+1}\left\lceil a\right\rceil -2^R-1,\\
\mathbb{V}_{R,a} &= \mathrm{span}\br{\phi_{R,k}: A_{R,1} \leq k\leq A_{R,2}}.
\end{split}
\end{equation} Then, the following holds:
\begin{itemize}
\item[(i)] 
\begin{equation}\label{span}
 V^{(a)}_0\oplus \left(\mathop{\oplus}_{j=0}^{R-1}W^{(a)}_j\right) \subset \mathbb{V}_{R,a}.
\end{equation}
\item[(ii)] Let $N=N_R$, as defined in (\ref{NRdef}). Then
\begin{align*}
\mathcal{T}_N = V_0^{(a)}\oplus W_0^{(a)}\oplus\cdots\oplus W_{R-1}^{(a)} \subset \mathbb{V}_{R,a},
\end{align*}
where $\mathcal{T}_N$ is defined in (\ref{ordering}) and (\ref{ordering2}). 
Moreover, if $\norm{\phi}_\infty$ and $\norm{\psi}_\infty$ exist, then given any $\varphi\in \cT_N$ such that $\norm{\varphi}=1$ and $R\geq \log_2(\lceil a\rceil -1)$, the following holds:
$$
\varphi=\sum_{j=A_{R,1}}^{A_{R,2}} \alpha_j \phi_{R,j}, \qquad 
\sum_{j=A_{R,1}}^{A_{R,2}}\abs{\alpha_j}^2 =1
$$
and
$$
\sum_{j=A_{R,2}-\lceil a\rceil +1}^{A_{R,2}}\abs{\alpha_j}^2 \leq \frac{\left(\norm{\phi}_\infty+\norm{\psi}_\infty\right)^2 \lceil a\rceil(\lceil a\rceil+1)}{2^{R+1}}. 
$$
This bound  shows that although each element of $\mathcal{T}_N$ may be expressed as a linear combination of elements in $\br{\phi_{R,k}: A_{R,1} \leq k\leq A_{R,2}}$, the contribution of $\phi_{R,k}$ for $A_{R,2}-\cl{a}\leq k\leq A_{R,2}$ is insignificant when $R$ is large. This fact will be used in the proof of Theorem \ref{thrm:main2}.

\item[(iii)]
\begin{equation}\label{eq:span2}
\phi_{R,k}\in V_0^{(a)}\oplus W_0^{(a)}\oplus\cdots\oplus W_{R-1}^{(a)} \qquad \text{whenever } 0\leq k \leq (2^R-1) \left\lceil a\right\rceil
\end{equation}
\end{itemize}
\end{lemma}

\begin{proof}
To prove (i) we start by observing that
 in MRA,  for each $R \in \mathbb{N}$ we have that
$$
V_{R} = V_0 \oplus \left( \mathop{\oplus}_{l=0}^{R-1}W_l \right), \qquad
V_{R} \supset V^{(a)}_0 \oplus \left( \mathop{\oplus}_{l=0}^{R-1} W^{(a)}_l \right). 
$$
Thus,
since $\br{\phi_{R,k} : k\in \mathbb{Z}}$ is an orthonormal basis for the closed subspace $V_R$, it follows that, 
given $l\in\mathbb{Z}$ such that $\abs{l}\leq \left\lceil a\right\rceil -1$,
\begin{align*}
\phi_{0,l}=\sum_{k\in\mathbb{Z}} \beta_k \phi_{R,k}, \qquad \beta_k =  \int_{\mathbb{R}} \phi_{0,l}(x)\overline{\phi_{R,k}(x)} \mathrm{d}x.
\end{align*}
Note that $\phi$ has compact support, so finitely many $\beta_k$'s are non-zero. In particular, $\beta_k=0$ if $k$ is such that $\mathrm{measure}\left(\mathrm{supp}(\phi_{0,l})\cap\mathrm{supp}(\phi_{R,k})\right)=0$. So, $\beta_k\neq0$ only if 
\begin{align*}
2^Rl-\lceil a\rceil+1\leq k\leq 2^R(\lceil a\rceil +l)-1
\end{align*}
and inserting $\abs{l} \leq \lceil a\rceil -1$, we find that
\begin{align}\label{eq;cond_k_1}
-(2^R+1)\left\lceil a\right\rceil  +2^R+1 \leq k\leq 2^{R+1}\left\lceil a\right\rceil -2^R-1.
\end{align}
Similarly, given $j,l\in\mathbb{Z}$ such that  $0\leq j\leq R-1$ and $-\left\lceil a\right\rceil +1 \leq l\leq 2^j\left\lceil a\right\rceil -1$,
\begin{align*}
\psi_{j,l}=\sum_{k\in\mathbb{Z}} \gamma_k \phi_{R,k} , \qquad \gamma_k =  \int_{\mathbb{R}} \psi_{j,l}(x)\overline{\phi_{R,k}(x)} \mathrm{d}x.
\end{align*}
Note that $\gamma_k=0$ if $k$ is such that $\mathrm{measure}\left(\mathrm{supp}(\psi_{j,l})\cap\mathrm{supp}(\phi_{R,k})\right)=0$. 
Thus, $\gamma_k \neq 0$ only if
 $$2^R\left(\dfrac{l}{2^j}\right)-\left\lceil a\right\rceil +1 \leq k \leq 2^R\left(\dfrac{l+\left\lceil a\right\rceil}{2^j}\right)-1.$$
Hence, $\gamma_k\neq0$ only if $k$ satisfies
\begin{align}\label{eq;cond_k_2}
-(2^{R-j}+1)\left\lceil a\right\rceil  +2^{R-j}+1\leq k\leq (2^R + 2^{R-j})\left\lceil a\right\rceil  -2^{R-j}-1.
\end{align}
Since we have shown that $\beta_k$ and $\gamma_k$ can be non-zero only if (\ref{eq;cond_k_1}) and (\ref{eq;cond_k_2}) are satisfied, we have demonstrated that all elements in $V_0^{(a)}\oplus W_0^{(a)}\oplus\cdots\oplus W_{R-1}^{(a)}$ may be represented as a linear combination of elements in $\mathbb{V}_{R,a}$, and we have proved (\ref{span}).

To prove (ii), note that $W_j^{(a)}$ has $(2^j +1)\left\lceil a\right\rceil -1$ basis elements. So, $V_0^{(a)}\oplus W_0^{(a)}\oplus\cdots\oplus W_{R-1}^{(a)}$ has precisely
\begin{align*}
2\left\lceil a\right\rceil -1 + \sum_{j=0}^{R-1}\left((2^j +1)\left\lceil a\right\rceil -1\right)=2^R \left\lceil a\right\rceil + (R+1)(\left\lceil a\right\rceil -1)
\end{align*}
basis elements.
Thus whenever $N=2^R \left\lceil a\right\rceil + (R+1)(\left\lceil a\right\rceil -1)$, it follows by the ordering in (\ref{ordering}), that
\begin{align*}
\mathcal{T}_N &= V_0^{(a)}\oplus W_0^{(a)}\oplus\cdots\oplus W_{R-1}^{(a)}.
\end{align*}
Hence, $\varphi\in\cT_N$ and $\norm{\varphi}=1$ implies that 
\begin{align*}
\varphi = \sum_{\abs{l}\leq \lceil a\rceil -1} b_l\phi_{0,l}+\sum_{j=0}^{R-1}\sum_{l=-\lceil a\rceil +1}^{2^j\lceil a\rceil -1}c_{j,l}\psi_{j,l} = \sum_{k=A_{R,1}}^{A_{R,2}}\alpha_k\phi_{R,k}
\end{align*} 
for some complex numbers $\br{\alpha_k}$, $\br{b_l}$ and $\br{c_{j,l}}$, where 
\begin{align*}
&\sum_{\abs{l}\leq \lceil a\rceil -1} \abs{b_l}^2 +\sum_{j=0}^{R-1}\sum_{l=-\lceil a\rceil +1}^{2^j\lceil a\rceil -1}\abs{c_{j,l}}^2=\sum_{j=A_{R,1}}^{A_{R,2}}\abs{\alpha_j}^2 =1
\end{align*}
by the orthonormality of the scaling functions and wavelets.
\\
By a similar argument to the proof of (i), it is straightforward to verify that when $2^R\geq \lceil a\rceil-1$,
\begin{align*}
\mathrm{span}\br{\phi_{0,l},\psi_{0,l} : l=-\lceil a\rceil +1,\ldots,\lceil a\rceil -2} \oplus \, \bigoplus_{j=1}^{R-1} W^{(a)}_{j}
&\subset \mathrm{span} \br{\phi_{R,k}: A_{R,1}\leq k\leq A_{R,2}-\lceil a\rceil}.
\end{align*}
Thus $$\varphi - b_{\cl{a}-1}\phi_{0,\cl{a}-1} - c_{0,\cl{a}-1}\psi_{0,\cl{a}-1}\in\mathrm{span} \br{\phi_{R,k}: A_{R,1}\leq k\leq A_{R,2}-\lceil a\rceil}$$
and it follows by orthonormality of the system $\br{\phi_{R,k}:k\in\bbZ}$ that for $k=A_{R,2}-\cl{a}+1,\ldots,A_{R,2}$, 
$$
\alpha_k =  b_{\lceil a\rceil -1}\ip{\phi_{0,\lceil a\rceil -1}}{\phi_{R,k}}+  c_{0,\lceil a\rceil -1}\ip{\psi_{0,\lceil a\rceil -1}}{\phi_{R,k}}.
$$
Let $B_k =  b_{\lceil a\rceil -1}\ip{\phi_{0,\lceil a\rceil -1}}{\phi_{R,k}}$ and $C_k = c_{0,\lceil a\rceil -1}\ip{\psi_{0,\lceil a\rceil -1}}{\phi_{R,k}}$ and suppose that both $\norm{\phi}_\infty$ and $\norm{\psi}_\infty$ exist.
\\
Then, for $j=0,\ldots,\lceil a\rceil-1$, since $\abs{b_{\lceil a\rceil -1}}\leq 1$ and $\norm{\phi_{R,A_{R,2}-j}}\leq 1$
\begin{align*}
\abs{B_{A_{R,2}-j}} \leq \norm{\phi_{R,A_{R,2}-j}}\norm{\phi_{0,\lceil a\rceil-1}\chi_{I_j}} \leq \norm{\phi}_\infty\sqrt{\dfrac{j+1}{2^R}}
\end{align*}
where $I_j = \mathrm{supp}\phi_{R,A_{R,2}-j}\cap\mathrm{supp}\phi_{0,\lceil a\rceil-1} \subset \left[2\lceil a\rceil -1 -(j+1)/2^R,2\lceil a\rceil -1\right]$. Thus
\begin{align}\label{eq:B_estimate}
\sum_{j=0}^{\lceil a\rceil -1} \abs{B_{A_{R,2}-j}}^2 \leq \frac{\norm{\phi}_\infty^2\lceil a\rceil(\lceil a\rceil +1)}{2^{R+1}}.
\end{align}
Similarly, for $j=0,\ldots,\lceil a\rceil-1$,
\begin{align*}
\abs{C_{A_{R,2}-j}} \leq \norm{\psi}_\infty \sqrt{\dfrac{j+1}{2^R}}
\end{align*}
and
\begin{align}\label{eq:C_estimate}
\sum_{j=0}^{\lceil a\rceil -1} \abs{C_{A_{R,2}-j}}^2 \leq \frac{\norm{\psi}_\infty^2\lceil a\rceil(\lceil a\rceil +1)}{2^{R+1}}.
\end{align} 
Hence, if $R\geq \log_2(\lceil a\rceil -1)$ and both $\norm{\phi}_\infty$ and $\norm{\psi}_\infty$ exist, then by the Cauchy Schwarz inequality, and estimates (\ref{eq:B_estimate}), (\ref{eq:C_estimate}),
\begin{align*}
\sum_{k=A_{R,2}-\lceil a\rceil+1}^{A_{R,2}} \abs{\alpha_k}^2 &=\sum_{k=A_{R,2}-\lceil a\rceil+1}^{A_{R,2}}\abs{B_k+C_k}^2\\
 &\leq \sum_{k=A_{R,2}-\lceil a\rceil+1}^{A_{R,2}} \left(\abs{B_k}^2+\abs{C_k}^2 \right)+ 2 \sqrt{\sum_{k=A_{R,2}-\lceil a\rceil+1}^{A_{R,2}}\abs{B_k}^2}\sqrt{\sum_{k=A_{R,2}-\lceil a\rceil+1}^{A_{R,2}}\abs{C_k}^2}\\
&= \frac{\left(\norm{\phi}_\infty+\norm{\psi}_\infty\right)^2 \lceil a\rceil(\lceil a\rceil+1)}{2^{R+1}}
\end{align*}
and this completes the proof of (ii).

Finally, to prove (iii) and (\ref{eq:span2}), note that
$$
\phi_{R,k}=\sum_{l\in\mathbb{Z}}\alpha_l\phi_{0,l}+\sum_{j=0}^{R-1}\sum_{l\in\mathbb{Z}} \beta_{j,l}\psi_{j,l}
$$
where $\alpha_l=\ip{\phi_{R,k}}{\phi_{0,l}}$ and $\beta_{j,l}=\ip{\phi_{R,k}}{\psi_{j,l}}$.
If $0\leq k \leq (2^R-1) \left\lceil a\right\rceil$, then 
$$
\mathrm{supp}(\phi_{R,k})\subset\left[0,\left\lceil a\right\rceil\right].
$$
So, if $\phi_{0,l}\not\in V_0^{(a)}$, then $\mathrm{measure}\left( \mathrm{supp}(\phi_{0,l})\cap\left[0,\left\lceil a\right\rceil\right]\right)=0$ and $\alpha_l=0$. 
Similarly, if $\psi_{j,l}\not\in W_j^{(a)}$, then $\mathrm{measure}\left(\mathrm{supp}(\psi_{j,l})\cap\left[0,\left\lceil a\right\rceil\right]\right)=0$ and $\beta_{j,l}=0$.
Hence, 
$$
\phi_{R,k}\in V_0^{(a)}\oplus W_0^{(a)}\oplus\cdots\oplus W_{R-1}^{(a)},
$$
as required.
\end{proof}

\subsection{Useful results with trigonometric polynomials}
Our proof hinges on some precise estimates on the behaviour of trigonometric polynomials. These estimates are presented below.
\begin{lemma}\label{lem:DFT}
 Let $A_1,A_2\in\mathbb{Z}$ be such that $A_1\leq A_2$ and consider the trigonometric polynomial $\Phi(z) = \sum_{j= A_1}^{A_2}\alpha_j e^{2\pi i jz}$. If $L \in \mathbb{N}$ is such that $2L\geq A_2-A_1+1$, then
$$\sum_{j=0}^{2L-1}\dfrac{1}{2L}\abs{\Phi\left(\dfrac{j}{2L}\right)}^2 = \norm{\Phi}^2_{L^2([0,1])}= \sum_{j= A_1}^{A_2}\abs{\alpha_j}^2.$$ 
\end{lemma}
\begin{proof}
Given $N\in\mathbb{N}$, $x= (x_0,\ldots,x_{N-1}) \in \mathbb{Z}_{N}$, the Discrete Fourier Transform of $x$ is defined by $\hat{x} = (\hat{x}_0, \ldots, \hat{x}_{N-1})$, where
\begin{align*}
\hat{x}_k = \dfrac{1}{\sqrt{N}}\sum_{j=0}^{N-1} x_j e^{-\frac{2\pi i kj}{N}}.
\end{align*}
Recall that $x\mapsto\hat{x}$ is a unitary operator on $l^2(\mathbb{Z}_N)$ with $\norm{x}_{l^2(\mathbb{Z}_N)}=\norm{\hat{x}}_{l^2(\mathbb{Z}_N)}$, where
\begin{align*}
\norm{x}_{l^2(\mathbb{Z}_N)} = \sqrt{\sum_{j=0}^{N-1} \abs{x_j}^2}.
\end{align*}

The proof of this lemma is a direct application of the Discrete Fourier Transform, with $N=2L$.
Define $x = (x_0, \ldots, x_{N-1}) \in \mathbb{Z}_{2L}$ as follows
\begin{align*}
 x_{j+L} = \begin{cases}
                      \alpha_{j+A_1+L} &-L\leq j\leq -L+A_2-A_1\\
                      0 &\text{otherwise.}
                     \end{cases}
\end{align*}
Then
\begin{align*}
 \hat{x}_k &=\dfrac{1}{\sqrt{2L}}\sum_{j=0}^{2L-1} x_j e^{-\frac{2\pi i kj}{2L}}= \dfrac{1}{\sqrt{2L}}\sum_{j=-L}^{L-1} x_{j+L} e^{-\frac{2\pi i k(j+L)}{2L}}= \dfrac{e^{\frac{\pi i k A_1}{L}}}{\sqrt{2L}}\sum_{j=A_1}^{A_2} \alpha_j e^{\frac{2\pi i k j}{2L}}= \dfrac{e^{-\frac{\pi i k A_1}{L}}}{\sqrt{2L}}\Phi\left(\dfrac{k}{2L}\right).
\end{align*}
So, 
\begin{align*}
\sum_{k=0}^{2L-1} \dfrac{1}{2L}\abs{\Phi\left(\dfrac{k}{2L}\right)}^2 =  \sum_{k=0}^{2L-1}\abs{\hat{x}_k}^2 = \sum_{k=0}^{2L-1}\abs{x_k}^2 = \sum_{k = A_1}^{A_2} \abs{\alpha_k}^2.
\end{align*}

\end{proof}

The following theorem is a reworking of a result from \cite[Proposition 1]{Grochenig}.

\begin{theorem}\label{thm:grochenig}
Let $D\in\mathbb{N}$, $A\in\mathbb{R}$, $A\leq x_1 < \ldots < x_r< A+1$ and suppose that
\begin{align*}
 \delta = \max_{j=1,\ldots,r} x_{j+1}-x_j < \dfrac{1}{2D}
\end{align*}
where $x_{r+1}=x_1+1$. If $\Phi(x) = \sum_{j=D_1}^{D_2}\alpha_j e^{2\pi i jx}$ and $D_2-D_1\leq 2D$, then
\begin{align*}
 (1-2\delta D) \norm{\Phi}_{L^2[A, A+1)} \leq \left(\sum_{j=1}^r \nu_j \abs{\Phi(x_j)}^2\right)^{\frac{1}{2}} \leq (1+2\delta D) \norm{\Phi}_{L^2[A,A+1)}
\end{align*}
where $\nu_j = \dfrac{1}{2}(x_{j+1}-x_{j-1})$ and $x_0 = x_r-1$.
\end{theorem}

\subsection{Bounding the stable sampling rate}
We are now ready to prove the linearity of the stable sampling rate. However, before we can present the final proof we need a couple of technical lemmas and propositions.
The following lemma is an adaptation of \cite[Theorem 6.3.1]{daubechies1992ten}, the proof has simply been included for clarity.
\begin{lemma} \label{lem:wavelet_property}
 Let $\phi$ be a compactly supported scaling function generating an MRA. Let $I$ be any closed interval of length $2\pi$. Then, for each $\gamma \in (0,1)$, there exists $N$ such that for all $\xi \in I$, $$\sum_{\abs{l}\leq N}\abs{\hat{\phi}(\xi+2\pi l)}^2 \geq \gamma.$$ 
\end{lemma}
\begin{proof}
 First note that the orthonormality of $\br{\phi(\cdot -l): l\in \mathbb{Z}}$ is equivalent to $$\sum_{l\in \mathbb{Z}}\abs{\hat{\phi}(\xi+2\pi l)}^2 =1$$ for all $\xi \in \mathbb{R}$. In particular, for all $\xi \in I$, there exists $N_\xi$ such that for some $\tilde{\gamma}\in(\gamma,1)$, $$\sum_{\abs{l}\leq N_\xi}\abs{\hat{\phi}(\xi+2\pi l)}^2 \geq \tilde{\gamma}.$$
Since $\hat{\phi}$ is continuous, it follows that 
$$
x \mapsto \sum_{\abs{l}\leq N_\xi}\abs{\hat{\phi}(x+2\pi l)}^2
$$ 
is also continuous. Hence, there exists some $\delta_\xi$ such that 
$$
\sum_{\abs{l}\leq N_\xi}\abs{\hat{\phi}(\eta+2\pi l)}^2 \geq \gamma, \qquad \forall \, \eta \in (\xi-\delta_\xi, \xi + \delta_\xi) =:U_\xi.
$$
Note that 
$I\subset\cup_{\xi \in I}U_\xi$ and $I$ is compact, hence, $I = \cup_{\xi \in J}U_\xi$ for some finite subset $J \subset I$. 
Let $N=\max\br{N_\xi:\xi \in J}$. Then for all $\xi \in I$, 
$$
\sum_{\abs{l}\leq N}\abs{\hat{\phi}(\xi+2\pi l)}^2 \geq \gamma.
$$
\end{proof}

\begin{proposition}\label{prop:formula}
 For $R,l\in \mathbb{Z}$, let the vectors $\phi_{R,l}$, $s_l^\epsilon$ and the finite rank operator $P_M$ be defined as in \ref{sec:setup}. Suppose that $N_1,N_2\in\mathbb{Z}$ and $\varphi = \sum_{l=N_1}^{N_2}\alpha_l \phi_{R,l}$, with $\alpha_l \in \mathbb{C}$, such that $\varphi$ is compactly supported in $[-T_1,T_2]$. Then, for all $j\in\mathbb{Z}$,
\begin{align*}
 \ip{\varphi}{s_j^\epsilon} = \dfrac{\sqrt{\epsilon}}{\sqrt{2^R}} \Phi\left(\dfrac{\epsilon j}{2^R}\right)\hat{\phi}\left(-\dfrac{2\pi\epsilon j}{2^R}\right)
\end{align*}
where 
$\Phi(z) = \sum_{l=N_1}^{N_2} \alpha_l e^{2\pi i lz}
$. 
In particular,
\begin{align*}
\norm{P_M \varphi}^2 &= \sum_{j=-\left\lfloor\frac{M}{2}\right\rfloor}^{\left\lceil\frac{M}{2}\right\rceil-1}\abs{\ip{\varphi}{s_j^\epsilon}}^2
=\sum_{j=-\left\lfloor\frac{M}{2}\right\rfloor}^{\left\lceil\frac{M}{2}\right\rceil-1}\dfrac{\epsilon}{2^R}\abs{\Phi\left(\dfrac{\epsilon j}{2^R}\right)\hat{\phi}\left(-\dfrac{2\pi\epsilon j}{2^R}\right)}^2.
\end{align*}
\end{proposition}
\begin{proof}
Note that $\epsilon\leq \dfrac{1}{T_1+T_2}$ and $s_j^\epsilon=\sqrt{\epsilon}e^{2\pi i j \epsilon \cdot}\chi_{\left[-\frac{T_1}{\epsilon(T_1+T_2)},\frac{T_2}{\epsilon(T_1+T_2)}\right]}$. So, by the assumption on the support of $\varphi,$
\begin{align*}
\ip{\varphi}{s_j^\epsilon} &= \sqrt{\epsilon}\int_{-\frac{T_1}{\epsilon(T_1+T_2)}}^{\frac{T_2}{\epsilon(T_1+T_2)}}\varphi(x) e^{2\pi i\epsilon jx} \mathrm{d}x
=\sqrt{\epsilon} \hat{\varphi}(-2\pi\epsilon j)
= \sqrt{\epsilon} \sum_{l=N_1}^{ N_2}\alpha_l \hat{\phi}_{R,l}(-2\pi\epsilon j)\\
&=\dfrac{\sqrt{\epsilon}}{\sqrt{2^R}}\sum_{l=N_1}^{N_2}\alpha_l e^{\frac{2\pi i\epsilon jl}{2^R}}\hat{\phi}\left(-\dfrac{2\pi\epsilon j}{2^R}\right)
= \dfrac{\sqrt{\epsilon}}{\sqrt{2^R}} \Phi\left(\dfrac{\epsilon j}{2^R}\right)\hat{\phi}\left(-\dfrac{2\pi\epsilon j}{2^R}\right).
\end{align*}
\end{proof}

In order to prove Theorem \ref{them;main}, we will show that given $N\in\bbN$, for all $\epsilon \in (0,1/(T_1+T_2)]$ and $\gamma\in (0,1)$, there is some $M=O(N)$ such that the subspace angle $C_{N,M}$ is at least $\gamma$, namely
$$
\inf_{\varphi\in \mathcal{T}_N, \norm{\varphi}=1} \norm{P^{\epsilon}_{M} \varphi} \geq \gamma.
$$ The following result shows that it is sufficient to do so only for \textit{some} $\epsilon \in (0,1/(T_1+T_2)]$.

\begin{proposition} \label{prop:sampling_rate}
Given $\gamma\in(0,1)$ and $\epsilon_1,\epsilon_2\in\left(0,1/(T_1+T_2)\right]$, choose $\delta(\gamma)\in(0,1)$ and $C(\gamma) >1$ such that 
\begin{align} \label{eq:const_cond}
\sqrt{\delta(\gamma)^2-\dfrac{4}{\pi^2(C(\gamma)-1)}}-\sqrt{1-\delta(\gamma)^2}>\gamma.
\end{align}
Suppose that there exists $M_1$ such that 
\begin{equation}\label{eq;M_1}
\inf_{\varphi\in \mathcal{T}_N, \norm{\varphi}=1} \norm{P^{\epsilon_1}_{M_1} \varphi} \geq \delta(\gamma), \qquad N \in \mathbb{N}.
\end{equation}
Then, 
the following holds:
\begin{equation}\label{eq;M_2}
\inf_{\varphi\in \mathcal{T}_N, \norm{\varphi}=1} \norm{P^{\epsilon_2}_{M_2} \varphi} \geq \gamma , \qquad N \in \mathbb{N},
\end{equation}
whenever 
\begin{equation}\label{eq;M_22}
M_2=\left\lceil\dfrac{C(\gamma)M_1\epsilon_1}{\epsilon_2}\right\rceil.
\end{equation}
\end{proposition}

\begin{proof}

Without loss of generality, in this proof, $M_1$ and $M_2$ will be even. Also, it is easy to see that 
$\delta(\gamma)$ and $C(\gamma)$ always exist. 
Observe now that for any $M_2\in\mathbb{N}$, 
\begin{align*}
\inf_{\varphi\in \mathcal{T}_N, \norm{\varphi}=1} \norm{P^{\epsilon_2}_{M_2} \varphi} &\geq 
\inf_{\varphi\in \mathcal{T}_N, \norm{\varphi}=1}  \left(\norm{P^{\epsilon_2}_{M_2} P^{\epsilon_1}_{M_1}\varphi} -  \norm{P^{\epsilon_2}_{M_2} (P^{\epsilon_1}_{M_1})^{\perp} \varphi} \right)\\
&\geq \inf_{\varphi\in \mathcal{T}_N, \norm{\varphi}=1} \norm{P^{\epsilon_2}_{M_2} P^{\epsilon_1}_{M_1}\varphi}- \sqrt{1-\delta(\gamma)^2}, 
\end{align*}
where the last inequality follows from (\ref{eq;M_1}). Hence, to prove the proposition,
it suffices to determine $M_2$ such that 
$$
\inf_{\varphi\in \mathcal{T}_N, \norm{\varphi}=1} \norm{P^{\epsilon_2}_{M_2} P^{\epsilon_1}_{M_1}\varphi} - \sqrt{1-\delta(\gamma)^2} \geq \gamma.
$$
In order to understand why $M_2$ exists, first note that $P^{\epsilon_2}_n\to P^{\epsilon_2}$ strongly as $n\to \infty$ and since $\mathcal{B}=P^{\epsilon_1}_{M_1}\left(\br{\varphi\in \cT_N : \norm{\varphi}=1}\right)$ is finite dimensional, $P^{\epsilon_2}_n\to P^{\epsilon_2}$ uniformly on $\mathcal{B}$ as $n\to\infty$. Also, $\cT\subset \cS^{\epsilon_1}\subset\cS^{\epsilon_2}$, $P^{\epsilon_2}P^{\epsilon_1}\varphi = \varphi$ for all $\varphi\in\cT_N$. So, for all $\xi>0$, there exists $M_2$ such that  
$$
\sup_{\varphi\in \mathcal{T}_N, \norm{\varphi}=1} \norm{P^{\epsilon_2}_{M_2} P^{\epsilon_1}_{M_1}\varphi-P^{\epsilon_2}P^{\epsilon_1}_{M_1}\varphi} \leq \xi.
$$
So, for $\xi$ sufficiently small and $M_2$ sufficiently large,
\begin{align*}
\inf_{\varphi\in \mathcal{T}_N, \norm{\varphi}=1} \norm{P^{\epsilon_2}_{M_2} P^{\epsilon_1}_{M_1}\varphi}&\geq 1- \sup_{\varphi\in \mathcal{T}_N, \norm{\varphi}=1}\left( \norm{P^{\epsilon_2}_{M_2} P^{\epsilon_1}_{M_1}\varphi-P^{\epsilon_2}P^{\epsilon_1}_{M_1}\varphi} +\norm{P^{\epsilon_2} P^{\epsilon_1}_{M_1}\varphi-\varphi}\right)\\
&\geq 1-\xi-\sqrt{1-\delta(\gamma)^2} \geq \gamma.
\end{align*}
 Thus, by the choice of $\delta(\gamma)$, for sufficiently small $\xi$ and so for sufficiently large $M_2$,
 $$
 \inf_{\varphi\in \mathcal{T}_N, \norm{\varphi}=1} \norm{P^{\epsilon_2}_{M_2} P^{\epsilon_1}_{M_1}\varphi} \geq \gamma.
 $$
 
Having established the existence of $M_2$ we now demonstrate that (\ref{eq;M_2}) follows when $M_2$ takes the value in (\ref{eq;M_22}). We begin by letting 
$$
\mathcal{B}_{M_2} = \br{l\in\mathbb{Z}:l\geq\dfrac{M_2}{2} \text{ or }l\leq -\dfrac{M_2}{2}-1}.
$$ 
Then
\begin{equation*}
\begin{split}
\norm{\left(P^{\epsilon_2}_{M_2}\right)^\perp P^{\epsilon_1}_{M_1} \varphi}^2 &= \norm{\sum_{l\in \mathcal{B}_{M_2}}\ip{P^{\epsilon_1}_{M_1} \varphi}{s^{\epsilon_2}_l}s^{\epsilon_2}_l}^2
= \norm{\sum_{l\in \mathcal{B}_{M_2}}\ip{\sum_{j=-\frac{M_1}{2}}^{\frac{M_1}{2}-1}\ip{\varphi}{s^{\epsilon_1}_j}s^{\epsilon_1}_j}{s^{\epsilon_2}_l}s^{\epsilon_2}_l}^2\\
&= \sum_{l\in \mathcal{B}_{M_2}}\abs{\sum_{j=-\frac{M_1}{2}}^{\frac{M_1}{2}-1}\ip{\varphi}{s^{\epsilon_1}_j}\ip{s^{\epsilon_1}_j}{s^{\epsilon_2}_l}}^2
\leq \sum_{l\in \mathcal{B}_{M_2}}\left(\sum_{j=-\frac{M_1}{2}}^{\frac{M_1}{2}-1}\abs{\ip{\varphi}{s^{\epsilon_1}_j}}^2 \sum_{j=-\frac{M_1}{2}}^{\frac{M_1}{2}-1} \abs{\ip{s^{\epsilon_1}_j}{s^{\epsilon_2}_l}}^2\right) 
\end{split}
\end{equation*}
and since $\sum_{j=-\frac{M_1}{2}}^{\frac{M_1}{2}-1}\abs{\ip{\varphi}{s^{\epsilon_1}_j}}^2 \leq \norm{\varphi} =1$, it follows that
\begin{equation}\label{eq;P_Mest}
\norm{\left(P^{\epsilon_2}_{M_2}\right)^\perp P^{\epsilon_1}_{M_1} \varphi}^2 
\leq  \sum_{l\in \mathcal{B}_{M_2}}\sum_{j=-\frac{M_1}{2}}^{\frac{M_1}{2}-1} \abs{\ip{s^{\epsilon_1}_j}{s^{\epsilon_2}_l}}^2.
\end{equation}
Let $\epsilon_+ = \max \br{\epsilon_1,\epsilon_2}$, and note that
\begin{equation}\label{eq;P_Mest2}
\begin{split}
\abs{\ip{s^{\epsilon_1}_j}{s^{\epsilon_2}_l}} &=\abs{ \sqrt{\epsilon_1\epsilon_2}\int_{-\frac{1}{2\epsilon_+}}^{\frac{1}{2\epsilon_+}} e^{2\pi i \epsilon_1 jx}e^{-2\pi i \epsilon_2 lx}\mathrm{d}x }
= \sqrt{\epsilon_1\epsilon_2} \abs{ \dfrac{\sin\left(\dfrac{\pi (\epsilon_1 j-\epsilon_2 l)}{\epsilon_+}\right)}{\pi (\epsilon_1 j-\epsilon_2 l) }}.
\end{split}
\end{equation}
So, by substituting (\ref{eq;P_Mest2}) into (\ref{eq;P_Mest}), we have that
\begin{align*}
\norm{\left(P^{\epsilon_2}_{M_2}\right)^\perp P^{\epsilon_1}_{M_1} \varphi}^2 
&\leq  \sum_{l\in \mathcal{B}_{M_2}}\sum_{j=-\frac{M_1}{2}}^{\frac{M_1}{2}-1} \epsilon_1\epsilon_2 \abs{\dfrac{\sin\left(\dfrac{\pi (\epsilon_1 j-\epsilon_2 l)}{\epsilon_+}\right)}{\pi (\epsilon_1 j-\epsilon_2 l) }}^2
\leq \dfrac{\epsilon_1\epsilon_2}{\pi^2}\sum_{l\in \mathcal{B}_{M_2}}\sum_{j=-\frac{M_1}{2}}^{\frac{M_1}{2}-1} \dfrac{1}{\abs{\epsilon_1 j-\epsilon_2 l}^2}.
\end{align*}
Suppose that $M_2=\left\lceil C(\gamma)M_1\epsilon_1/\epsilon_2\right\rceil$ where $C(\gamma)$ stems from (\ref{eq:const_cond}), then
\begin{align*}
\norm{\left(P^{\epsilon_2}_{M_2}\right)^\perp P^{\epsilon_1}_{M_1} \varphi}^2
&\leq  \dfrac{\epsilon_1\epsilon_2}{\pi^2} M_1\sum_{l>\frac{M_2}{2}}\frac{2}{\abs{\epsilon_1 \frac{M_1}{2}-\epsilon_2 l}^2}
\leq \frac{2\epsilon_1 M_1}{\epsilon_2\pi^2} \int_{\frac{M_2}{2}}^\infty \left(x-\frac{\epsilon_1 M_1}{2\epsilon_2}\right)^{-2} \mathrm{d}x\\
&\leq  \frac{\epsilon_1}{\pi^2} M_1 \frac{4}{\left(-\epsilon_1 M_1+\epsilon_2 M_2\right)}
\leq  \frac{\epsilon_1}{\pi^2} M_1 \frac{4}{\left(\epsilon_1 M_1(C(\gamma)-1)\right)}
\leq  \frac{4}{\pi^2(C(\gamma)-1)}.
\end{align*}
Therefore,
\begin{align*}
\norm{P^{\epsilon_2}_{M_2} P^{\epsilon_1}_{M_1} \varphi}^2 &=\norm{P^{\epsilon_1}_{M_1} \varphi}^2 - \norm{\left(P^{\epsilon_2}_{M_2}\right)^\perp P^{\epsilon_1}_{M_1} \varphi}^2
\geq \delta(\gamma)^2-\dfrac{4}{\pi^2(C(\gamma)-1)}
\end{align*}
whenever 
$$
M_2=\left\lceil\dfrac{C(\gamma)M_1\epsilon_1}{\epsilon_2}\right\rceil.
$$
Hence,
\begin{align*}
\inf_{\varphi\in \mathcal{T}_N, \norm{\varphi}=1} \norm{P^{\epsilon_2}_{M_2} \varphi} 
&\geq \sqrt{\delta(\gamma)^2-\dfrac{4}{\pi^2(C(\gamma)-1)}}- \sqrt{1-\delta(\gamma)^2} >\gamma
\end{align*}
by the choice of $\delta(\gamma)$ and $C(\gamma)$ in (\ref{eq:const_cond}).
\end{proof}

\subsection{The proof}
We are now ready to present the proof of Theorem \ref{them;main}.
\begin{proof}[Proof of Theorem \ref{them;main}]
Let $N\leq N_R$  with $R\in\mathbb{N}$ and recall that the reconstruction space $\mathcal{S}$ is defined for sampling density $\epsilon$ such that $0 < \epsilon\leq 1/(T_1+T_2)$. 

We now  fix $\epsilon=1/(T_1+T_2+\lceil a\rceil)$. Suppose that for this fixed $\epsilon$, we can show that given any $\delta \in (0,1)$, there exists $S_\delta \in\mathbb{N}$, independent of $R$, such that for $M_\delta=S_\delta 2^{R+1}/\epsilon$, we have that
$$
\inf_{\varphi\in \mathcal{T}_N, \norm{\varphi}=1} \norm{P_{M_\delta}^\epsilon \varphi} \geq \delta.
$$
Then from Proposition \ref{prop:sampling_rate}, given $\gamma \in (0,1)$ and any sampling density $\epsilon_1 \in\left(0,{1}/(T_1+T_2)\right]$,  by choosing $C(\gamma)\in\bbN$ and $\delta(\gamma)\in (0,1)$ such that 
\begin{align*} 
\sqrt{\delta(\gamma)^2-\dfrac{4}{\pi^2(C(\gamma)-1)}}-\sqrt{1-\delta(\gamma)^2}>\gamma,
\end{align*}
we have that
$$
\inf_{\varphi\in \mathcal{T}_N, \norm{\varphi}=1} \norm{P_M^{\epsilon_1} \varphi} \geq \gamma \qquad \text{whenever}\qquad
M =\left\lceil\dfrac{C S_{\delta(\gamma)}2^{R+1}}{\epsilon_1}\right\rceil.
$$
where by assumption $S_{\delta(\gamma)}\in\bbN$ is such that for $M_{\delta(\gamma)}=S_{\delta(\gamma)} 2^{R+1}/\epsilon$,
$$
\inf_{\varphi\in \mathcal{T}_N, \norm{\varphi}=1} \norm{P_{M_{\delta(\gamma)}}^\epsilon \varphi} \geq \delta(\gamma).
$$
Hence, it is sufficient to prove this theorem for $\epsilon=1/(T_1+T_2+\left\lceil a\right\rceil)$. 

Recall $A_{R,1}$ and $A_{R,2}$ from (\ref{eq:A_R}), then by the choice of $N$ and Lemma \ref{lem:decomposition},
\begin{align}\label{span2}
\mathcal{T}_N  \subset \mathrm{span}\br{\phi_{R,k}: A_{R,1}\leq k\leq A_{R,2}}.
\end{align}
Let $\varphi \in \mathcal{T}_N$ such that $\norm{\varphi} =1$. 
Then, by (\ref{span2}), we have that
\begin{equation}\label{eq;one}
\varphi = \sum_{l=A_{R,1}}^{ A_{R,2}}\alpha_l \phi_{R,l}, \quad \sum_{l=A_{R,1}}^{ A_{R,2}}\abs{\alpha_l}^2 =1.
\end{equation}
Moreover, $\varphi$ is compactly supported in $[-T_1,T_2]$ since it is a linear combination of elements in $\Omega_a$. Thus, by Proposition \ref{prop:formula},
\begin{align*}
\norm{P_M \varphi}^2 &= \sum_{j=-\fl{\frac{M}{2}}}^{\cl{\frac{M}{2}}-1}\abs{\ip{\varphi}{s_j}}^2 =\sum_{j=-\fl{\frac{M}{2}}}^{\cl{\frac{M}{2}}-1}\dfrac{\epsilon}{2^R}\abs{\Phi\left(\dfrac{\epsilon j}{2^R}\right)\hat{\phi}\left(-\dfrac{2\pi\epsilon j}{2^R}\right)}^2
\end{align*}
where
\begin{align}\label{eq;Phi}
\Phi(z) = \sum_{l=A_{R,1}}^{ A_{R,2}} \alpha_l e^{2\pi i lz}.
\end{align} 
Let $L=2^R/\epsilon$, then $L$ is some even integer since ${1}/{\epsilon} =  T_1+T_2+\left\lceil a\right\rceil = 4 \lceil a \rceil -2 \in \bbN$. Furthermore, suppose that ${M}/{2} = SL$ for some $S\in\mathbb{N}$ which we will subsequently determine.  Then:
\begin{align*}
\norm{P_M \varphi}^2 &=\sum_{j=0}^{L-1} \sum_{k=-S}^{S-1}\dfrac{\epsilon}{2^R}\abs{\Phi\left(\dfrac{\epsilon}{2^R}(j+kL)\right)}^2\abs{\hat{\phi}\left(-\dfrac{2\pi \epsilon}{2^R}(j+kL)\right)}^2\\
&=\sum_{j=0}^{L-1}\dfrac{1}{L}\abs{\Phi\left(\dfrac{j}{L}\right)}^2\sum_{k=-S}^{S-1}\abs{\hat{\phi}\left(-\dfrac{2\pi j}{L}-2\pi k\right)}^2.
\end{align*}

By applying Lemma \ref{lem:wavelet_property} to the interval $[-2\pi,0]$, given any $\theta \in (1,\infty)$, we can choose $S\in \mathbb{N}$ such that for all $j=0,\ldots L-1$,
$$
\sum_{k=-{S}}^{S-1}\abs{\hat{\phi}\left(-\dfrac{2\pi j}{L}-2\pi k\right)}^2 \geq \frac{1}{\theta^2}.
$$
Since 
$$
L=\dfrac{2^R}{\epsilon}= 2^{R}(4\left\lceil a\right\rceil-2) > 2^{R}(3\left\lceil a\right\rceil -2)+\lceil a\rceil -1=A_{R,2}-A_{R,1}+1,
$$ Lemma \ref{lem:DFT} (via (\ref{eq;one}) and (\ref{eq;Phi})) implies that $$\sum_{j=0}^{L-1}\dfrac{1}{L}\abs{\Phi\left(\dfrac{j}{L}\right)}^2=1.$$ 
Thus, 
\begin{align*}
\norm{P_M \varphi}^2 \geq \frac{1}{\theta^2} \sum_{j=0}^{L-1}\dfrac{1}{L}\abs{\Phi\left(\dfrac{j}{L}\right)}^2 = \frac{1}{\theta^2}.
\end{align*}
Hence, for $N\leq N_R$ and $M=S 2^{R+1}/\epsilon$, where $S$ depends only on the scaling function $\phi$ and $\theta$,
\begin{align*}
C_{N,M} = \inf_{\varphi\in \mathcal{T}_N, \norm{\varphi}=1} \norm{P_M \varphi} \geq \frac{1}{\theta},
\end{align*}
and the theorem is proven.
\end{proof}

\section{Proof of Theorem \ref{thrm:main2}}\label{s:prf2}
The proof of Theorem \ref{thrm:main2} hinges on the following proposition.

\begin{proposition}\label{prop:exp_blowup}
 Let $N\geq N_R$, and suppose $M=c2^R$ for $c < \dfrac{1}{\epsilon}$, then $C_{N,M} \to 0$ exponentially as $N\to \infty$.
 
\end{proposition}

With this result at hand the proof of Theorem \ref{thrm:main2} is straightforward.
\begin{proof}[Proof of Theorem \ref{thrm:main2}]
Suppose that $\eta_G <\frac{1}{\epsilon \lceil a\rceil}$. Then, by Corollary \ref{c:universal} and Proposition \ref{prop:exp_blowup}, $\kappa(G_N)$ cannot be bounded.
Moreover, from \cite{BAACHOptimality}, for $M=\Theta_G(N_R)$, we have that 
$$\kappa(G_{N_R}) \geq \kappa(F_{N_R,M})\geq \frac{1}{C_{N_R,M}}.$$
Hence, by Proposition \ref{prop:exp_blowup}, $\kappa(G_{N_R})$ becomes exponentially large as $N_R$ grows. 
\end{proof}

The rest of this section is devoted to the proof of Proposition \ref{prop:exp_blowup}, however, before we can state the proof, we need the following results on trigonometric polynomials and Chebyshev polynomials from \cite{Erdelyi92remez-typeinequalities}. 

\begin{proposition}\label{prop:cheb}
 Let $\omega \in [0,\pi]$ and consider the following function, defined over $[-\pi,\pi]$:
\begin{align*}
 Q_{n,\omega}(z) = Q_{2n}\left(\dfrac{\sin(z/2)}{\sin(\omega/2)}\right)
\end{align*}
where $Q_{2n}(x) = \cos(2n\arccos x)$ for $x\in[-1,1]$ is the Chebyshev polynomial of degree $2n$. Then the following holds:
\begin{enumerate}
 \item[(i)] $Q_{n,\omega}$ is a trigonometric polynomial in $z$ of degree $n$, i.e. $Q_{n,\omega}(z) = \sum_{\abs{j}\leq n} \alpha_j e^{izj}$.
\item[(ii)] $\norm{Q_{n,\omega}}_{L^\infty[-\omega,\omega]}=1$. 
\item[(iii)] For $\omega \in [\pi/2,\pi)$, there exists constants $c_1,c_2>0$ such that 
$$
\exp({c_1 n(\pi-\omega)})\leq \norm{Q_{n,\omega}}_{L^\infty[-\pi,\pi]}=Q_{n,\omega}(\pi)\leq \exp({c_2 n(\pi-\omega)}).
$$
\end{enumerate}
\end{proposition}

\begin{proof}[Proof of Proposition \ref{prop:exp_blowup}]
The goal is to use Proposition \ref{prop:cheb}, and the first part of the proof is a setup for that. In particular, let $M=c2^R$ for some $c<1/\epsilon$ and $R\geq 1$. By Lemma \ref{lem:decomposition}, if $0\leq l\leq (2^R-1)\left\lceil a\right\rceil$, then 
$$
\phi_{R,l}\in V_0^{(a)}\oplus W_0^{(a)}\oplus\cdots\oplus W_{R-1}^{(a)}.
$$ 
Hence, for $N\geq N_R$ and 
$p = (2^{R-1}-1)\left\lceil a\right\rceil$, it follows that
 $$\mathcal{T}_N\supset V_0^{(a)}\oplus W_0^{(a)}\oplus\cdots\oplus W_{R-1}^{(a)}\supset \br{\phi_{R,l}:0\leq l\leq 2p}.$$
Thus, we get that 
\begin{align*}
 (C_{N,M})^2 = \mathop{\inf_{\norm{\varphi}=1}}_{\varphi \in \mathcal{T}_N}\norm{P_M\varphi}^2 
 \leq \inf\{\norm{P_M\varphi}^2: \sum_{l=0}^{2p}\abs{\beta_l}^2=1, \varphi =\sum_{l=0}^{2p}\beta_l \phi_{R,l}\}.
\end{align*}
Hence, by Proposition \ref{prop:formula} and the choice of $M = c2^R$, it follows that 
\begin{equation}\label{estimate_C}
\begin{split}
(
C_{N,M})^2 &\leq \inf\br{\sum_{j=-\fl{\frac{M}{2}}}^{\cl{\frac{M}{2}}-1} \dfrac{\epsilon}{2^{R}}\abs{\Phi\left(\dfrac{2\pi\epsilon j}{2^{R}}\right)\hat{\phi}\left(-\dfrac{2\pi\epsilon j }{2^{R}}\right) }^2: \Phi(z) = \sum_{l=0}^{2p} \beta_l e^{ izl}, \sum_{l=0}^{2p}\abs{\beta_l}^2=1}\\
&\leq \norm{\hat{\phi}}_{L^\infty[-\pi c\epsilon,\pi c\epsilon]}^2\inf\br{\sum_{j=-\fl{\frac{M}{2}}}^{\cl{\frac{M}{2}}-1} \dfrac{\epsilon}{2^{R}}\abs{\Phi\left(\dfrac{2\pi\epsilon j}{2^{R}}\right) }^2: \Phi(z) = \sum_{l=0}^{2p} \beta_l e^{ izl}, \norm{\Phi}_{L^2[-\pi,\pi]}^2=1}\\
&= \norm{\hat{\phi}}_{L^\infty[-\pi c\epsilon,\pi c\epsilon]}^2\inf\br{\sum_{j=-\fl{\frac{M}{2}}}^{\cl{\frac{M}{2}}-1} \dfrac{\epsilon}{2^{R}}\abs{\Phi\left(\dfrac{2\pi\epsilon j}{2^{R}}\right) }^2: \Phi(z) = \sum_{\abs{l}\leq p} \beta_l e^{ izl}, \norm{\Phi}_{L^2[-\pi,\pi]}^2=1}.
\end{split}
\end{equation}
The last equality above is a consequence of the following: For $\Phi(z) = \sum_{l=0}^{2p} \beta_l e^{ izl},$
\begin{align*}
\sum_{j=-\fl{\frac{M}{2}}}^{\cl{\frac{M}{2}}-1} \dfrac{\epsilon}{2^{R}}\abs{\Phi\left(\dfrac{2\pi\epsilon j}{2^{R}}\right) }^2 &=\sum_{j=-\fl{\frac{M}{2}}}^{\cl{\frac{M}{2}}-1} \dfrac{\epsilon}{2^{R}}\abs{\sum_{\abs{l}\leq p} \beta_{l+p} e^{2\pi i\epsilon jl/2^R}e^{2\pi i \epsilon jp/2^R} }^2
=\sum_{j=-\fl{\frac{M}{2}}}^{\cl{\frac{M}{2}}-1} \dfrac{\epsilon}{2^{R}}\abs{\sum_{\abs{l}\leq p} \beta_{l+p} e^{2\pi i\epsilon jl/2^R} }^2.
\end{align*}
Note that we have carried out this shift in indices in order to later show that the infimum is taken over a set of functions which include those of the form $Q_{n,\omega}$ defined in Proposition \ref{prop:cheb}. 
From (\ref{estimate_C}) it follows easily that
\begin{align*}
 &(C_{N,M})^2 \leq c\epsilon\norm{\hat{\phi}}_{L^\infty[-\pi c\epsilon,\pi c\epsilon]}^2\inf \br{ \norm{\Phi}_{L^\infty[-\pi c\epsilon,\pi c\epsilon]}^2: \Phi(z) = \sum_{\abs{l}\leq p} \beta_l e^{izl}, \norm{\Phi}_{L^2[-\pi,\pi]}^2=1},
\end{align*}
where we have again used that $M = c2^R$.
Also, by the Cauchy Schwarz inequality, for $\Phi(z) = \sum_{\abs{l}\leq p} \beta_l e^{izl}$, $$\abs{\Phi(z)}^2 \leq (2p +1)\sum_{\abs{l}\leq p} \abs{\beta_l}^2 = (2p +1) \norm{\Phi}^2_{L^2[-\pi,\pi]}.$$
So, 
$$
\norm{\Phi}^2_{L^\infty[-\pi,\pi]}=2p +1 \implies \norm{\Phi}^2_{L^2[-\pi,\pi]}\geq 1.
$$
Thus,
\begin{equation}\label{final_C}
\begin{split}
 (C_{N,M})^2 &\leq c\epsilon\norm{\hat{\phi}}_{L^\infty[-\pi c\epsilon,\pi c\epsilon]}^2\inf \br{ \norm{\Phi}_{L^\infty[-\pi c\epsilon,\pi c\epsilon]}^2: \Phi(z) = \sum_{\abs{l}\leq p}  \beta_l e^{ izl}, \norm{\Phi}^2_{L^2[-\pi,\pi]}\geq 1}\\
 &\leq D_R\inf \br{ \norm{\Phi}_{L^\infty[-\pi c\epsilon,\pi c\epsilon]}^2: \Phi(z) = \sum_{\abs{l}\leq p} \beta_l e^{izl}, \norm{\Phi}^2_{L^\infty[-\pi,\pi]}=1},\\
\end{split}
\end{equation}
where 
$$
D_R = (2p+1)c\epsilon\norm{\hat{\phi}}_{L^\infty[-\pi c\epsilon,\pi c\epsilon]}^2.
$$

Having established (\ref{final_C}) we can now make use of Proposition \ref{prop:cheb}. Indeed,
for $\omega \in [\pi/2,\pi)$, let 
$$
q_{\omega} =\dfrac{Q_{p,\omega}}{\norm{Q_{p,\omega}}_{L^\infty[-\pi,\pi]}},
$$ 
where $Q_{p,\omega}$ is defined in Proposition \ref{prop:cheb}.
Then, by Proposition \ref{prop:cheb},
\begin{equation}\label{q}
q_\omega\in \br{ \Phi: \Phi(z) = \sum_{\abs{l}\leq p} \beta_l e^{izl}, \norm{\Phi}^2_{L^\infty[-\pi,\pi]}=1},\end{equation}
and there exists some constant $\eta>0$, independent of $p$, such that 
\begin{equation}\label{q2}
\norm{q_{\omega}}_{L^\infty[-\omega,\omega]} \leq \dfrac{1}{\norm{Q_{p,\omega}}_{L^\infty[-\pi,\pi]}}\leq \exp({-\eta p(\pi-\omega)}).
\end{equation}
We now split the proof into two cases, and we will show that $C_{N,M} \to 0$ exponentially as $R\to \infty$ when
 $$
\text{Case 1:} \quad 
c\in\left[\dfrac{1}{2\epsilon},\dfrac{1}{\epsilon}\right), \qquad \text{Case 2:} \quad 
c\in\left(0,\dfrac{1}{2\epsilon} \right).
$$

Case 1: By (\ref{final_C}), (\ref{q}) and (\ref{q2}) (and recalling the value of $p = (2^{R-1}-1)\left\lceil a\right\rceil$),
\begin{align*}
 (C_{N,M})^2&\leq D_R \norm{q_{\pi c\epsilon}}_{L^\infty[-\pi c\epsilon,\pi c\epsilon]}^2 \\ 
&\leq (2^R\left\lceil a\right\rceil -2\left\lceil a\right\rceil +1)c\epsilon\norm{\hat{\phi}}_{L^\infty[-\pi c\epsilon,\pi c\epsilon]}^2 \exp({(-\eta\pi(1-c\epsilon)(2^{R}\left\lceil a\right\rceil -2\left\lceil a\right\rceil))})
\end{align*}  
Thus, we have shown that $C_{N,M}$ decays exponentially as $N\to \infty$ in the first case scenario.  

Case 2: Clearly, we still have exponential decay in $C_{N,M}$, since,  again by (\ref{final_C}), (\ref{q}) and (\ref{q2}),
\begin{align*}
 (C_{N,M})^2&\leq D_R \norm{q_{\pi/2}}_{L^\infty[-\pi c\epsilon,\pi c\epsilon]}^2 
\leq D_R \norm{q_{\pi/2}}_{L^\infty[-\pi/2,\pi /2]}^2 \\
&\leq (2^R\left\lceil a\right\rceil -2\left\lceil a\right\rceil +1)c\epsilon\norm{\hat{\phi}}_{L^\infty[-\pi c\epsilon,\pi c\epsilon]}^2\exp({-\eta\pi(2^{R-1}\left\lceil a\right\rceil -\left\lceil a\right\rceil)}).
\end{align*}
\end{proof}

\section{Proof of Theorem \ref{prop:daubechies_case}}\label{s:prf3}
We are now ready to present the proof of Theorem \ref{prop:daubechies_case}.

\begin{remark}\label{rem:range}
In the construction of Daubechies wavelets \cite{daubechies1992ten}, the scaling function $\phi$ is defined such that 
$$
\hat{\phi}(\xi) := \prod_{s=1}^{\infty} m_0\left(\frac{\xi}{2^s}\right)
$$
where 
$$
m_0(\xi)=\left(\frac{1+e^{-i\xi}}{2}\right)^N \mathcal{L}(\xi)
$$
for some $N\in\mathbb{N}$ and $\mathcal{L}$ is such that
$$
\abs{\mathcal{L}(\xi)}^2=\sum_{k=0}^{N-1}\binom{N-1+k}{k}\sin^{2k}\left( \frac{\xi}{2}\right).
$$

Note that in this case, $\abs{m_0(\xi)}>0$ for all $\xi\in(-\pi,\pi)$ and since $\hat{\phi}(0)=1$, there exists $K\in\mathbb{N}$ such that $\abs{\hat{\phi}\left(\xi/2^K\right)}>0$ for all $\xi\in(-2\pi,2\pi)$.
Hence,
\begin{align}\label{eq:non_zero}
\hat{\phi}(\xi)=\hat{\phi}\left(\dfrac{\xi}{2^K}\right)\prod_{s=1}^K  m_0\left(\frac{\xi}{2^s}\right)\neq 0 \qquad\text{for all } \xi\in(-2\pi,2\pi).
\end{align}
\end{remark}

\begin{proof}[Proof of Theorem \ref{prop:daubechies_case}]

Recall that $\epsilon\in (0,1/(T_1+T_2)]$ and from Proposition \ref{prop:exp_blowup}, for all $c<1/\epsilon$, $C_{N_R,M}$ will tend to $0$ exponentially if $M<2^R/\epsilon$. So, for each $\theta\in(1,\infty)$, there exists $R_0\in\mathbb{N}$ such that for all $R\geq R_0$,
$
\Theta(N_R;\theta)\geq \left\lceil \dfrac{2^R}{\epsilon}\right\rceil
$. 
Hence, if it is known that there exists $R_1$ and $\theta\in(1,\infty)$ such that for all $R\geq R_1$
\begin{equation}\label{eq:daub_upper}
\Theta(N_R;\theta)\leq \left\lceil \dfrac{2^R}{\epsilon}\right\rceil
\end{equation}
then for such $\theta$ and all $R\geq \max\br{R_0,R_1}$ we have
$
\Theta(N_R;\theta)= \left\lceil 2^R/\epsilon\right\rceil
$.
So, it remains to show the existence of $\theta \in (1,\infty)$ such that (\ref{eq:daub_upper}) holds.
Let $\varphi\in \mathcal{T}_{N_R}$ be such that $\norm{\varphi}=1$. Then, by Lemma \ref{lem:decomposition},  we have that 
\begin{equation}\label{eq:sum1}
\varphi = \sum_{l= A_{R,1}}^{A_{R,2}} \alpha_l \phi_{R,l}, \qquad \sum_{l= A_{R,1}}^{A_{R,2}} \abs{\alpha_l}^2=1,
\end{equation} where $A_{R,1}$ and $A_{R,2}$ are as defined in (\ref{eq:A_R}). Now, let $M= \lceil 2^R/\epsilon\rceil$. Then, by Proposition \ref{prop:formula}, 
\begin{align}\label{eq;first_bound}
 \norm{P_M \varphi}^2 &=\sum_{j=-\fl{\frac{M}{2}}}^{\cl{\frac{M}{2}}-1}\dfrac{\epsilon}{2^R}\abs{\Phi\left(\dfrac{\epsilon j}{2^R}\right)\hat{\phi}\left(-\dfrac{2\pi\epsilon j}{2^R}\right)}^2 \geq \gamma_1^2 \sum_{j=-\fl{\frac{M}{2}}}^{\cl{\frac{M}{2}}-1}\dfrac{\epsilon}{2^R}\abs{\Phi\left(\dfrac{\epsilon j}{2^R}\right)}^2
\end{align}
where
$
\Phi(z) = \sum_{l= A_{R,1}}^{A_{R,2}} \alpha_l e^{2\pi i l z}
$ and 
\begin{align*}
\gamma_1&=\inf_{\xi\in\left[-\pi\epsilon M 2^{-R},\pi\epsilon M 2^{-R}\right]}\abs{\hat{\phi}(\xi)}
\geq \inf_{\xi\in\left[-(1+\epsilon 2^{-R})\pi,(1+\epsilon 2^{-R})\pi\right]}\abs{\hat{\phi}(\xi)}>0.
\end{align*} 
Note that  $\left[-(1+\epsilon 2^{-R})\pi,(1+\epsilon 2^{-R})\pi\right] \subset (-2\pi,2\pi)$ by the assumption that $\epsilon \leq \frac{1}{T_1+T_2} <1$ and by  (\ref{eq:non_zero}), $\gamma_1>0$. Note also that  we can let $\gamma_1=\inf_{\xi\in [-\pi,\pi]}\abs{\hat{\phi}(\xi)}$ whenever $2^R/\epsilon\in \mathbb{Z}$ since we have set $M=2^R/\epsilon$.
 Hence, it remains to obtain a positive lower bound for 
\begin{align}\label{eq:interm_sum}
\sum_{j=-\fl{\frac{M}{2}}}^{\cl{\frac{M}{2}}-1}\dfrac{\epsilon}{2^R}\abs{\Phi\left(\dfrac{\epsilon j}{2^R}\right)}^2.
\end{align}

We will split the proof into several cases. The case of $a=1$ is treated separately mainly for pedagogical reasons as the proof is simpler in this case. 

{\bf Case 1:} $a=1$ and $1/\epsilon\in\mathbb{N}$.

Since $a=1$, we have that $2^R/\epsilon \geq 2^{R}= A_{R,2}-A_{R,1}+1$ and for $1/\epsilon\in\mathbb{N}$ (in which case, $M=2^R/\epsilon$ is even), Lemma \ref{lem:DFT} gives that 
\begin{align*}
\sum_{j=-\frac{M}{2}}^{\frac{M}{2}-1}\dfrac{\epsilon}{2^R}\abs{\Phi\left(\dfrac{\epsilon j}{2^R}\right)}^2 =1.
\end{align*}
So, given any $R\in\mathbb{N}$ and 
$
\theta \geq \left(\inf_{\xi\in [-\pi,\pi]}\abs{\hat{\phi}(\xi)}\right)^{-1}
$, we have that 
$
\Theta(N_R;\theta)\leq 2^R/\epsilon
$.

{\bf Case 2:} $a=1$ and $1/\epsilon\notin\mathbb{N}$.

In this case we must have $\epsilon<1$, and an application of Theorem \ref{thm:grochenig} to $\Phi$ with $r=\left\lceil 2^R/\epsilon\right\rceil$, $2D=2^R\geq A_{R,2}-A_{R,1}$, $\delta={\epsilon}/{2^R}$ and 
$$
x_j=\frac{\epsilon}{2^R}\left(-\fl{\frac{M}{2}}+j-1\right), \qquad j=1,\ldots,r
$$ gives that
\begin{align*}
\sum_{j=-\fl{\frac{M}{2}}}^{\cl{\frac{M}{2}}-1}\dfrac{\epsilon}{2^R}\abs{\Phi\left(\dfrac{\epsilon j}{2^R}\right)}^2 \geq \left(1-\epsilon\right)^2 >0.
\end{align*}
So, given any $R\in\mathbb{N}$ and  
$$
\theta \geq \left(\left(1-\epsilon\right)\inf_{\xi\in\left[-(1+\epsilon)\pi,(1+\epsilon)\pi\right]}\abs{\hat{\phi}(\xi)}\right)^{-1},
$$ we have that 
$
\Theta(N_R;\theta)\leq 2^R/\epsilon.
$

{\bf Case 3:} $a>1$ and $2^R/\epsilon\in\mathbb{N}$ for some $R$.

When $a>1$, Lemma \ref{lem:DFT} and  Theorem \ref{thm:grochenig} cannot be applied directly because $2^R/\epsilon$ may be less than $A_{R,2}-A_{R,1}+1=2^{R}(3\lceil a\rceil -2)+\lceil a\rceil -1$ and so, we will first decompose $\Phi$ into two other trigonometric polynomials for which we can obtain bounds. 

We now let $R\geq \log_2( \lceil a\rceil -1)$. Since $\phi$ and $\psi$ are continuous and compactly supported, $\norm{\phi}_\infty$ and $\norm{\psi}_\infty$ exist. So, by Lemma \ref{lem:decomposition} (ii) and Proposition \ref{prop:formula},
\begin{align*}
\sum_{j=-\fl{\frac{M}{2}}}^{\cl{\frac{M}{2}}-1}\dfrac{\epsilon}{2^R}\abs{\Phi\left(\dfrac{\epsilon j}{2^R}\right)}^2
=\sum_{j=-\fl{\frac{M}{2}}}^{\cl{\frac{M}{2}}-1} \dfrac{\epsilon}{2^R}\abs{\Phi_1\left(\dfrac{\epsilon j}{2^R}\right)+\Phi_2\left(\dfrac{\epsilon j}{2^R}\right)}^2 
\end{align*}
where 
\begin{align*}
\Phi_1(z) = \sum_{j=A_{R,1}}^{A_{R,2}-\lceil a\rceil}\alpha_j e^{2\pi i z j},\qquad
\Phi_2(z) = \sum_{j=A_{R,2}-\lceil a\rceil+1}^{A_{R,2}}\alpha_j e^{2\pi i z j},\qquad 
\end{align*}
and 
\begin{equation}\label{eq:sum2}
\sum_{j=A_{R,2}-\lceil a\rceil+1}^{A_{R,2}} \abs{\alpha_j}^2  \leq \frac{\left(\norm{\phi}_\infty+\norm{\psi}_\infty\right)^2 \lceil a\rceil(\lceil a\rceil+1)}{2^{R+1}}.
\end{equation}
So, as argued in (\ref{eq;first_bound}),
$$
\norm{P_M \varphi}^2 \geq \gamma_1^2 \left(C^2_{\Phi_1}+C^2_{\Phi_2} -2C_{\Phi_1}C_{\Phi_2}\right)
$$
where 
\begin{equation}\label{eq;C}
C_{\Phi_s} = \sqrt{\sum_{j=-\fl{\frac{M}{2}}}^{\cl{\frac{M}{2}}-1}\dfrac{\epsilon}{2^R} \abs{\Phi_s\left(\dfrac{\epsilon j}{2^R}\right)}^2}, \qquad s=1,2.
\end{equation} 

If $2^R/\epsilon\in\mathbb{N}$ for some $R$, then we may apply Lemma \ref{lem:DFT} to $\Phi_1$ since
$$
(A_{R,2}-A_{R,1})-\lceil a\rceil +1 =2^R(3\lceil a\rceil -2) \leq 2^R/\epsilon,
$$
and to $\Phi_2$ since
$$
A_{R,2}-(A_{R,2}-\cl{a}+1) +1 = \cl{a}\leq 2^R/\epsilon.
$$
We thus obtain
$$
C_{\Phi_1}^2=\sum_{j=A_{R,1}}^{A_{R,2}-\lceil a\rceil} \abs{\alpha_j}^2, \qquad
C_{\Phi_2}^2= \sum_{j=A_{R,2}-\lceil a\rceil+1}^{A_{R,2}} \abs{\alpha_j}^2.
$$ 
Note that
\begin{align*}
C^2_{\Phi_1}+C^2_{\Phi_2} -2C_{\Phi_1}C_{\Phi_2}
&= \sum_{j=A_{R,1}}^{A_{R,2}}\abs{\alpha_j}^2 -2 \left(\sum_{j=A_{R,1}}^{A_{R,2}-\cl{a}}\abs{\alpha_j}^2\right)^{1/2}\left(\sum_{j=A_{R,2}-\cl{a}+1}^{A_{R,2}}\abs{\alpha_j}^2\right)^{1/2}\\
&\geq 1-2\left(\frac{\left(\norm{\phi}_\infty+\norm{\psi}_\infty\right)(\lceil a\rceil+1)}{2^{(R+1)/2}} \right)
\end{align*}
by (\ref{eq:sum1}) and (\ref{eq:sum2}).
Hence, for all $\mu\in(0,1)$, there exists $R_0$ such that for all $R\geq R_0$,
$$
\norm{P_M\varphi}^2 \geq \inf_{\xi\in[-\pi,\pi]}\abs{\hat{\phi}(\xi)}^2\left(1- \frac{\left(\norm{\phi}_\infty+\norm{\psi}_\infty\right)(\lceil a\rceil+1)}{2^{(R-1)/2}} \right) > \mu \inf_{\xi\in[-\pi,\pi]}\abs{\hat{\phi}(\xi)}^2,
$$
and so given any 
$$
\theta > \left(\inf_{\xi\in[-\pi,\pi]}\abs{\hat{\phi}(\xi)}\right)^{-1},
$$ there exists $R_0$ such that for all $R\geq R_0$,
$$
\Theta(N_R;\theta)\leq 2^R/\epsilon.
$$

{\bf Case 4:} $a>1$ and $2^R/\epsilon\not \in \mathbb{N}$ for all $R\in\mathbb{N}$.

In this case, $\epsilon<1/(3\lceil a\rceil -2)$ and as in Case 3, obtaining appropriate estimates for $C_{\Phi_1}$ and $C_{\Phi_2}$ defined in (\ref{eq;C}) will provide the required lower bound for (\ref{eq:interm_sum}).
 
In the case of $C_{\Phi_1}$, applying Theorem \ref{thm:grochenig} to $\Phi_1$ with $r=\left\lceil 2^R/\epsilon\right\rceil$, $2D=2^R(3\cl{a}-2)-2= (A_{R,2}-\cl{a})-A_{R,1}$, $\delta={\epsilon}/{2^R}$ and 
$$
x_j=\frac{\epsilon}{2^R}\left(-\fl{\frac{M}{2}}+j-1\right), \qquad j=1,\ldots,r
$$ gives that
\begin{align*}
 (1-\delta_1) \sum_{j=A_{R,1}}^{A_{R,2}-\lceil a\rceil} \abs{\alpha_j}^2 \leq \left(\sum_{j=1}^r \nu_j \abs{\Phi_1(x_j)}^2\right)^{\frac{1}{2}} \leq (1+\delta_1) \sum_{j=A_{R,1}}^{A_{R,2}-\lceil a\rceil} \abs{\alpha_j}^2
\end{align*}
where $\delta_1 =\epsilon(3\lceil a\rceil -2-1/2^{R-1}) < \epsilon(3\lceil a\rceil -2)<1$, $\nu_j =(x_{j+1}-x_{j-1})/2$ and $x_0 = x_r-1$. Note that $\epsilon/2^{R+1}\leq\nu_j\leq 2^R/\epsilon$.
Hence, by (\ref{eq;C}), 
$$
(1-\delta_1)^2 \sum_{j=A_{R,1}}^{A_{R,2}-\lceil a\rceil} \abs{\alpha_j}^2 \leq C_{\Phi_1}^2 \leq 2(1+\delta_1)^2 \sum_{j=A_{R,1}}^{A_{R,2}-\lceil a\rceil} \abs{\alpha_j}^2.
$$

In the case of $C_{\Phi_2}$, applying Theorem \ref{thm:grochenig} to $\Phi_2$ with $r=\left\lceil 2^R/\epsilon\right\rceil$, $2D=2\cl{(\cl{a}-1)/2}\geq A_{R,2}-(A_{R,2}-\cl{a}+1)$, $\delta={\epsilon}/{2^R}$ and 
$$
x_j=\frac{\epsilon}{2^R}\left(-\fl{\frac{M}{2}}+j-1\right), \qquad j=1,\ldots,r
$$ gives that
\begin{align*}
 (1-\delta_2) \sum_{j=A_{R,2}-\cl{a}+1}^{A_{R,2}} \abs{\alpha_j}^2 \leq \left(\sum_{j=1}^r \nu_j \abs{\Phi_1(x_j)}^2\right)^{\frac{1}{2}} \leq (1+\delta_2) \sum_{j=A_{R,2}-\cl{a}+1}^{A_{R,2}} \abs{\alpha_j}^2
\end{align*}
where $\delta_2 \leq\epsilon(\cl{a}+1)/2^R<1$, $\nu_j =(x_{j+1}-x_{j-1})/2$ and $x_0 = x_r-1$. Again, $\epsilon/2^{R+1}\leq\nu_j\leq 2^R/\epsilon$.
So,
$$
(1-\delta_2)^2 \sum_{j=A_{R,2}-\lceil a\rceil+1}^{A_{R,2}} \abs{\alpha_j}^2 \leq C_{\Phi_2}^2 \leq 2(1+\delta_2)^2 \sum_{j=A_{R,2}-\lceil a\rceil+1}^{A_{R,2}} \abs{\alpha_j}^2.
$$
Hence,
\begin{align*}
\norm{P_M\varphi}^2 &\geq \inf_{\xi\in[-(1+\epsilon 2^{-R})\pi,(1+\epsilon 2^{-R})\pi]}\abs{\hat{\phi}(\xi)}^2\left((1-\delta_1)^2-(1+\delta_2)(1+\delta_1) \frac{\left(\norm{\phi}_\infty+\norm{\psi}_\infty\right)(\lceil a\rceil+1)}{2^{(R-3)/2}} \right)\\
&\to (1-\epsilon(3\lceil a\rceil -2))^2\inf_{\xi\in[-\pi,\pi]}\abs{\hat{\phi}(\xi)}^2>0 \quad\text{as} \quad R\to\infty.
\end{align*}
So, for all $\mu\in(0,1)$, there exists $R_0$ such that for all $R\geq R_0$,
\begin{align*}
\norm{P_M\varphi}^2 &\geq \mu(1-\epsilon(3\lceil a\rceil -2))^2\inf_{\xi\in[-\pi,\pi]}\abs{\hat{\phi}(\xi)}^2>0.
\end{align*}
and for all 
$$
\theta > \left((1-\epsilon(3\lceil a\rceil -2))\inf_{\xi\in[-\pi,\pi]}\abs{\hat{\phi}(\xi)}\right)^{-1},
$$ there exists $R_0$ such that for all $R\geq R_0$, 
$
\Theta(N_R;\theta)\leq 2^R/\epsilon.
$
\end{proof}

\section{Numerical Examples}\label{s:numexp}
In this section we provide numerical examples to illustrate the behaviour of the stable sampling rate as well as demonstrating sharpness of our estimates. We also show that, because of the linearity of the stable sampling rate, any convergence properties of a series expansion of a function in a particular wavelet basis will be inherited (up to a constant) by the generalized sampling reconstruction based on Fourier samples.  In other words, as discussed in Section \ref{ss:optimal}, generalized sampling is, up to a constant, an oracle for the wavelet reconstruction problem.

\subsection{Sharpness of the stable sampling rate estimates}
Before we demonstrate the sharpness of our estimates numerically, let us recall the result from Theorem  \ref{prop:daubechies_case}. In particular, 
for $N_R=2^R\left\lceil a\right\rceil +(R+1)(\left\lceil a\right\rceil -1)$ and when $2^R/\epsilon \in \mathbb{Z}$, then for all sufficiently large $R$
\begin{align} \label{eq:daub_ssr2}
\Theta(N_R;\theta) = \dfrac{2^R}{\epsilon},
\end{align}
where 
\begin{align}\label{eq:gamma2}
\theta > \left(\inf_{\xi\in [-\pi,\pi]}\abs{\hat{\phi}(\xi)}\right)^{-1},
\end{align}
and $\phi$ is the scaling function of the wavelet.
Recall also the asymptotic result
\begin{equation}\label{eq;asympt}
\lim_{R\to\infty}\dfrac{\Theta(N_R;\theta)}{N_R}=\dfrac{1}{\epsilon\lceil a\rceil}.
\end{equation}
In this section we demonstrate these sharp results numerically. We consider the Haar wavelet (supported on $[0,1]$), the Daubechies-4 wavelet (supported in $[0,3]$), and the Daubechies-6 wavelet (supported in $[0,5]$). 

\begin{figure}
\centering
{\includegraphics[scale=0.57]{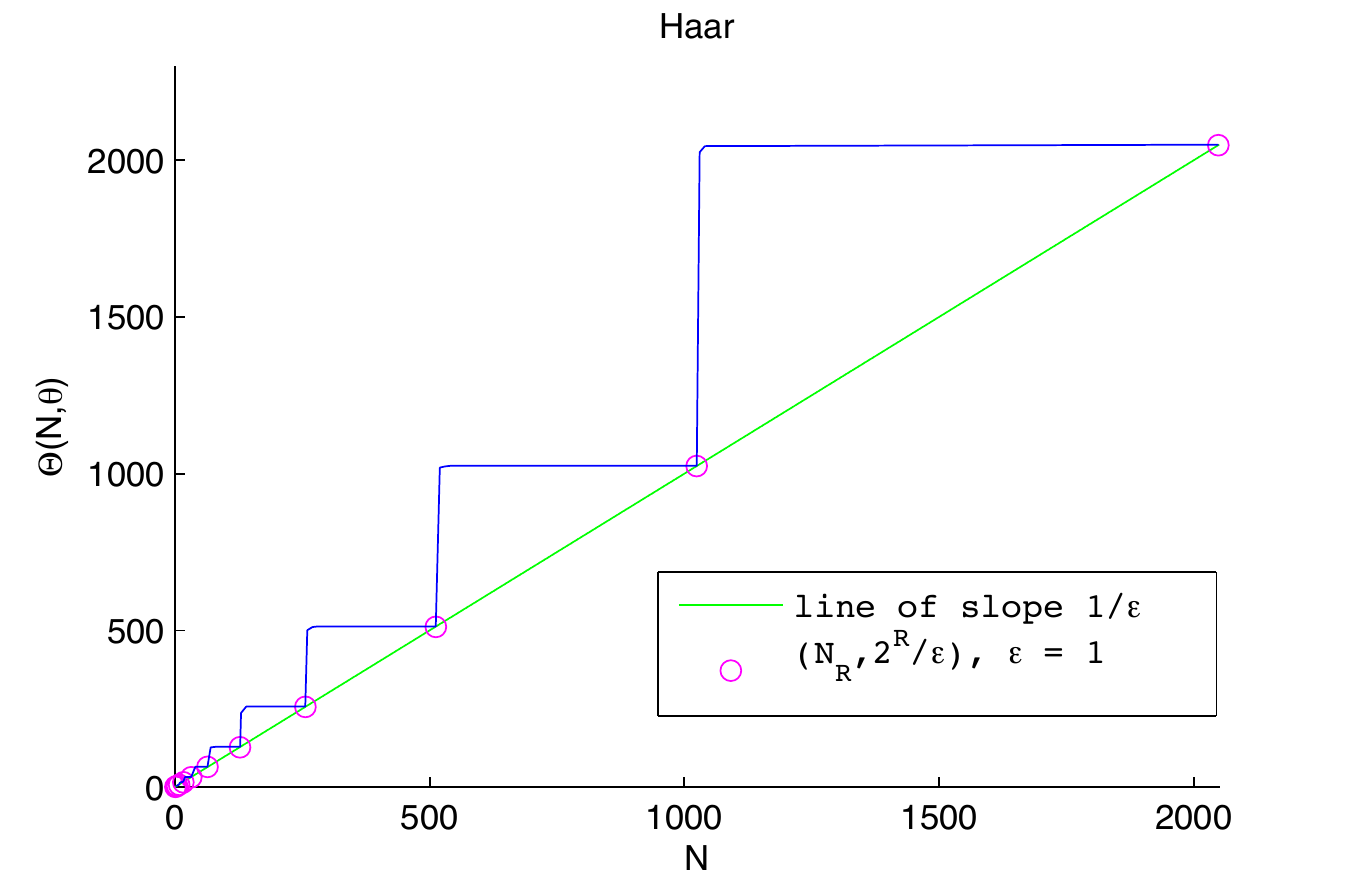}
\includegraphics[scale=0.57]{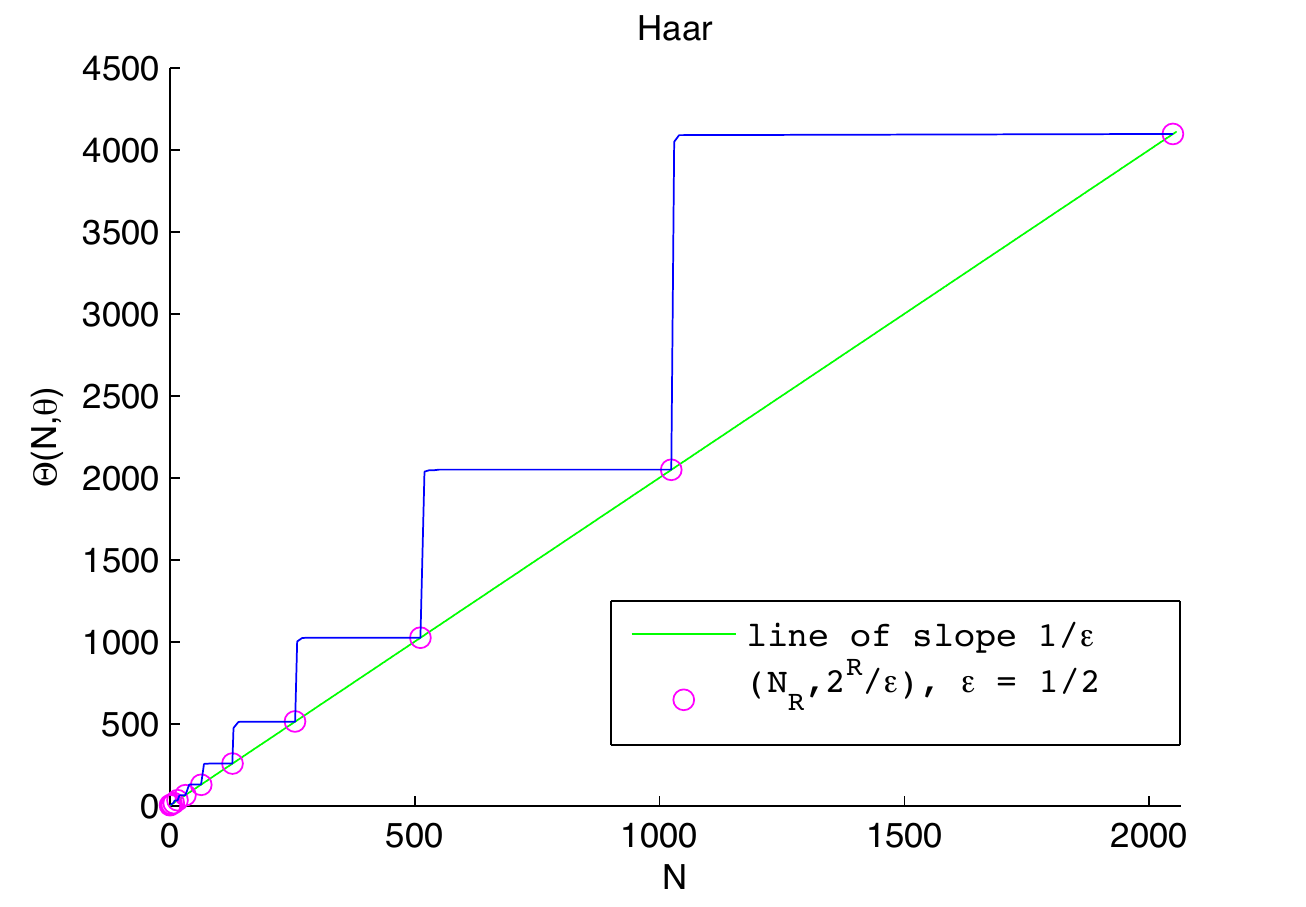}\\
}
\caption{The figure displays the stable sampling rate $\Theta(N;\theta)$ in blue for the Haar wavelet with Fourier samples for $\theta = \pi/2$ at a sampling density $\epsilon = 1$ (left) and $\epsilon = 1/2$ (right).
\label{fig:haar}}
\end{figure}

For the Haar wavelet, the Fourier sampling density must be $\epsilon\leq 1$.  Since 
$$
\left(\inf_{\xi\in [-\pi,\pi]}\abs{\hat{\phi}(\xi)}\right)^{-1} = \frac{\pi}{2},
$$
in this case, from the proof of Theorem \ref{prop:daubechies_case}, we see that (\ref{eq:daub_ssr2}) applies whenever $\theta \geq \frac{\pi}{2}$.

Figure \ref{fig:haar} shows the growth of $\Theta\left(N,\pi/2\right)$ for sampling densities $\epsilon = 1$ 
and $\epsilon = 1/2$ respectively. We observe from the figure that 
$$
\Theta\left(N_R,\pi/2\right) = 2^{R}, \qquad 
\Theta\left(N_R,\pi/2\right) = 2^{R+1}, \qquad R \in \mathbb{N}
$$ respectively, exactly as suggested in (\ref{eq:daub_ssr2}). Moreover, by (\ref{eq;asympt}), we have that
$$
\Theta(N_R;\pi/2) \sim N_R\frac{1}{\epsilon \lceil a\rceil}, 
$$
which is verified in Figure \ref{fig:haar} via the green line.

\begin{figure}
\centering
{\includegraphics[scale=0.57]{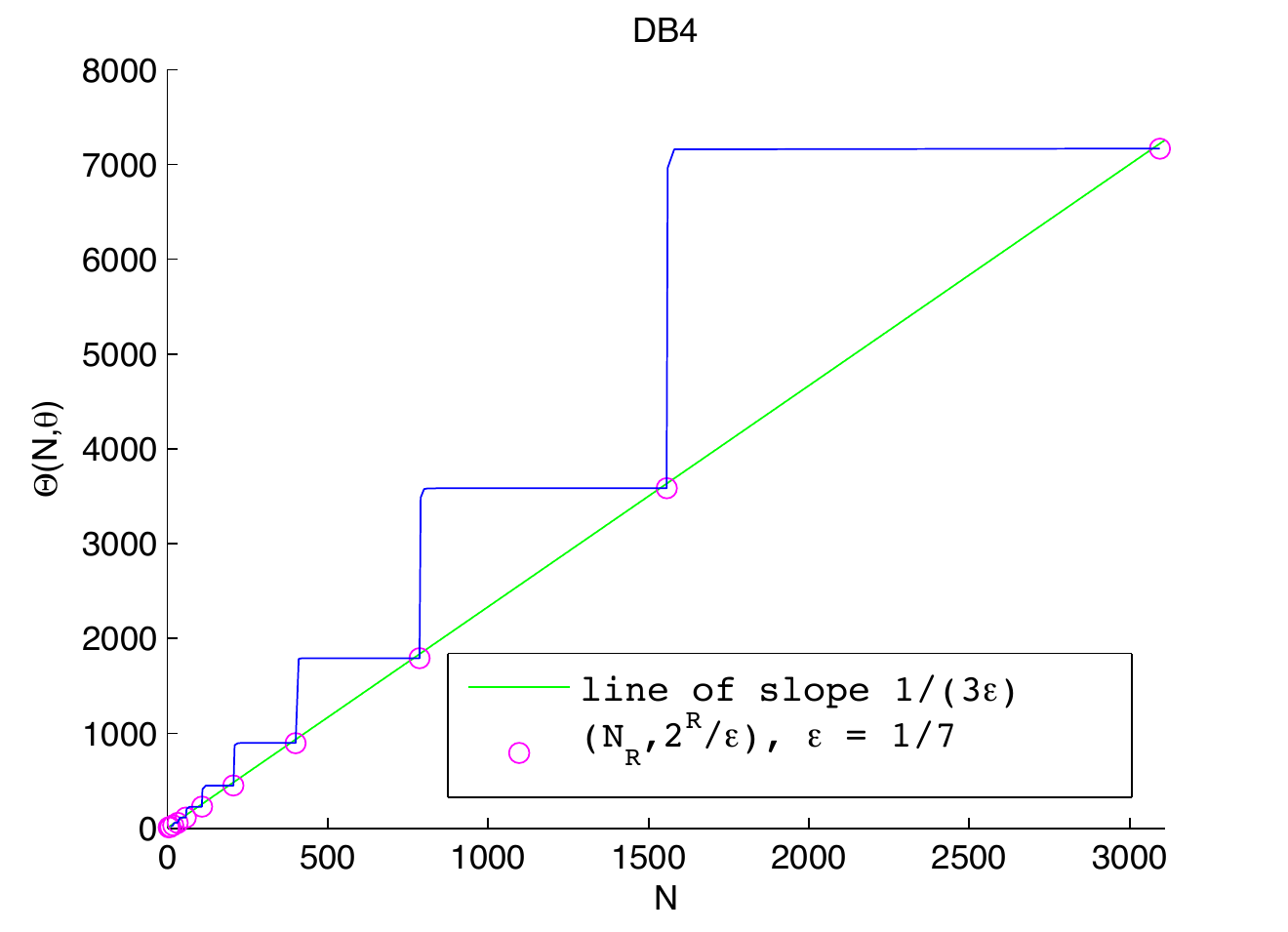}
\includegraphics[scale=0.57]{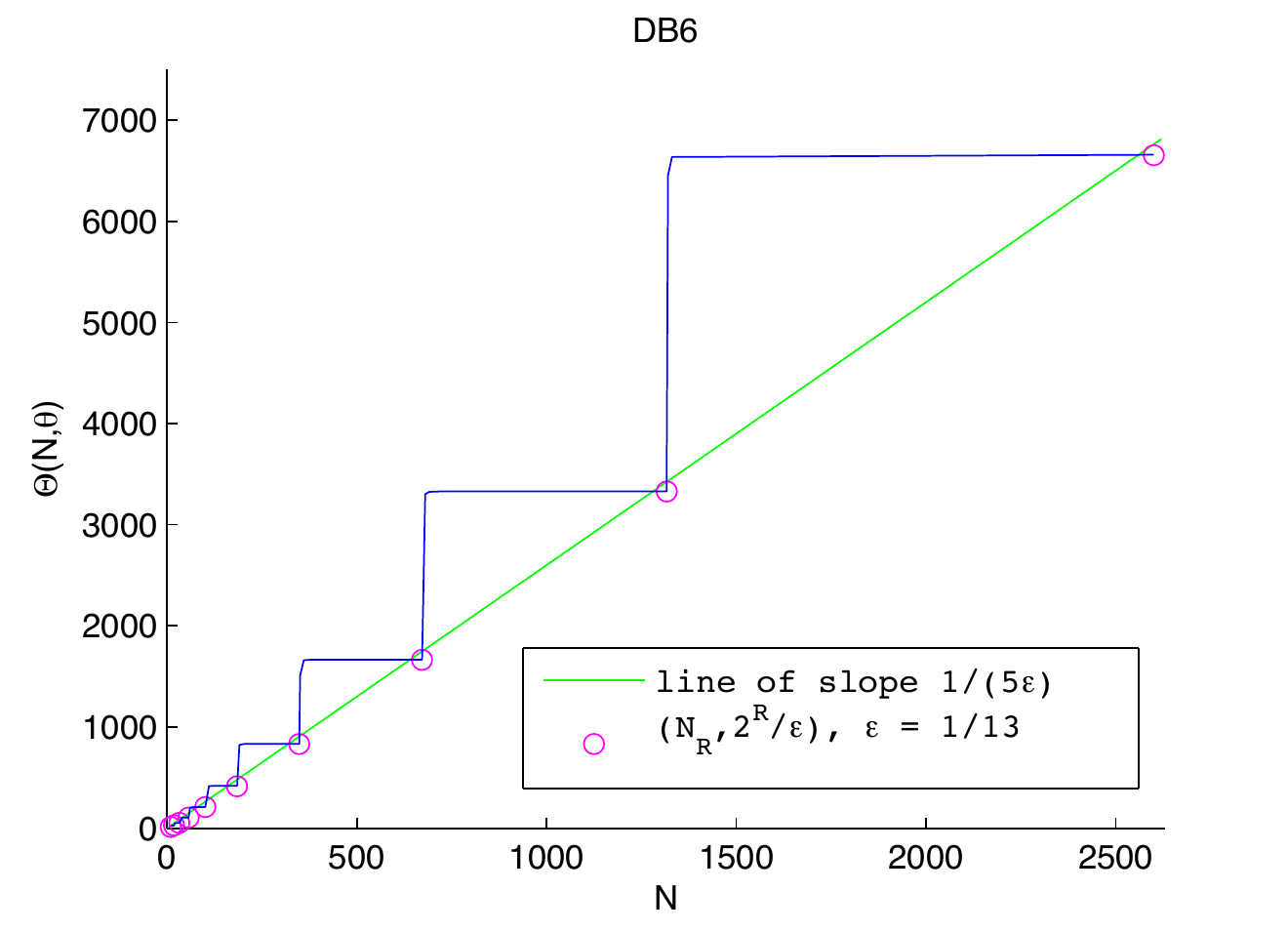}\\
}
\caption{The figure displays the stable sampling rate $\Theta(N;\theta_1)$ and $\Theta(N;\theta_2)$ in blue for the Daubechies-4 wavelet (left) and the Daubechies-6 wavelet (right) with Fourier samples at a sampling density $\epsilon = 1/7$  and $\epsilon = 1/13$ respectively.\label{fig:haar05}}
\end{figure}

In the case of the DB4 and DB6 wavelets, the Fourier sampling space must be of sampling density $\epsilon\leq 1/7$ and 
$\epsilon \leq 1/13$ respectively . Computationally we may observe that 
$$
\theta_1^{-1} = 0.684 < \inf_{x\in[-\pi,\pi]}\abs{\hat\phi_{DB4}(x)}  , \qquad 
\theta_2^{-1} = 0.698 < \inf_{x\in[-\pi,\pi]}\abs{\hat\phi_{DB6}(x)},
$$
where $\phi_{DB4}$ and $\phi_{DB6}$ are the scaling function of the DB4 and DB6 wavelets respectively. So, again, as displayed in Figure \ref{fig:haar05}, we have
$$
\Theta\left(N_R; \theta_1\right) = 7\cdot 2^R, \qquad 
\Theta\left(N_R;\theta_2\right) = 13 \cdot2^R, \qquad R \in \mathbb{N},
$$
which confirms (\ref{eq:daub_ssr2}). Moreover, by (\ref{eq;asympt}), we have that
$$
\Theta(N_R;\theta_1 ) \sim N_R\frac{1}{3\epsilon}, \qquad \Theta(N_R;\theta_2) \sim N_R\frac{1}{5\epsilon}
$$
which is verified in Figure \ref{fig:haar05} via the green line.

\begin{remark}
Note that 
$$
\Theta (N_R;\theta) < \Theta(N;\theta) \leq \Theta(N_{R+1};\theta),\quad N_R < N \leq N_{R+1}.
$$
The staircase effect witnessed in the figures suggests that the upper bound is in fact an equality.  Hence, although the stable sampling rate is linear for all $N$, from the point of view of the stable sampling rate at least, there is nothing to be gained from allowing $N \neq N_R$.
\end{remark}

\subsection{Generalized sampling and function reconstruction}
In this section we demonstrate the power of generalized sampling as recovery scheme of wavelet coefficients in practice. Given the result on the stable sampling rate above we have now full control over how to balance the number of Fourier samples versus the number of wavelet coefficients in order to get a stable and quasi-optimal reconstruction. This combination of quasi-optimality and the linearity of the stable sampling rate means that any decay in the wavelet coefficients of the underlying signal is preserved in the generalized sampling reconstruction.
 
In these experiments with Daubechies wavelets we will use the predicted  value  from (\ref{eq;asympt}), namely, the number of samples $M$ should asymptotically satisfy 
$$
M = \frac{N}{\epsilon \left\lceil a\right\rceil},
$$
where $N$ is the number of coefficients to be computed, $\epsilon$ is the sampling density and $a$ is the maximum value of the support of the mother wavelet.

We will also demonstrate, as predicted by Theorem \ref{thrm:main2}, that failure of satisfying the stable sampling rate gives a completely unstable and even non-convergent reconstruction. In this case we will chose the disastrous value
$$
M = cN, \qquad c < \frac{1}{\epsilon \left\lceil a\right\rceil},
$$
which causes the condition number of the algorithm to blow up exponentially. It also makes the constant in the error bound blow up at the same rate and thus one gets a non-convergent method.

\begin{table}
\begin{center}
{\small{
\begin{tabular}{|c| c| c| c|c|}

\thickhline
 
$(M,N,\alpha)$   	& 	 $\|f-f_M\|_{L^{2}}$	& 	 $\|f-\tilde f_{N,M}\|_{L^{2}}$  	&	 $- \frac{\log \| f - \tilde{f}_{N,M} \|}{\log N}$	&	Wavelet	\T\B\\\hline
$(906, 348,2)$	& 	$6.3\times 10^{-4}$	& 	$8.9\times 10^{-5}$	& 	1.59	& 	DB 3	\T\\
$(1748,672,2)$	& 	$2.9\times 10^{-4}$	& 	$3.3\times 10^{-5}$	& 	1.59	& 	DB 3	\\
$(3422, 1316,2)$	& 	$1.6\times 10^{-4}$	& 	$1.2\times 10^{-5}$	& 	1.58	& 	DB 3	\B\\\hline
$(934,400,2.5)$	&	$2.3\times 10^{-3}$	&	$3.1\times 10^{-6}$	&	2.12	&	DB 2	\T\\
$(1834,786,2.5)$	& 	$1.2\times 10^{-3}$	& 	$8.1\times 10^{-7}$	& 	2.10	& 	DB 2	\\
$(3632,1556,2.5)$	& 	$6.3\times 10^{-4}$	& 	$2.0\times 10^{-7}$	& 	2.10	& 	DB 2	\B\\\hline
$(256,256,3)$	& 	$1.4\times 10^{-2}$	& 	$4.2\times 10^{-7}$	& 	2.65	& 	Haar	\T\\
$(512,512,3)$	& 	$1.2\times 10^{-2}$	& 	$7.5\times 10^{-8}$	& 	2.63	& 	Haar	\\
$(1024,1024,3)$	& 	$1.2\times 10^{-2}$	& 	$1.3\times 10^{-8}$	& 	2.62	& 	Haar	\B\\\hline

\end{tabular}
   }}
  \end{center}
  \caption{The table shows the error of the reconstructions based on classical Fourier series, $f_M$, as well as generalized sampling 
  $\tilde f_{N,M}$ with different types of wavelets. Note that both $f_M$ and $f_{N,M}$ use exactly the same samples.}
\label{t:numexp1}
 \end{table}

The test functions will be of the form
\begin{equation}\label{eq;decay}
f = \sum_{j=1}^{3\times 10^3} \beta_j \varphi_j, \qquad \beta_j = j^{-\alpha}, \quad \alpha > 1,
\end{equation}
where the $\varphi_j$s are different types of Daubechies wavelets. We let $\tilde{f}_{N,M}$ denote the function that is constructed with 
generalized sampling using $M$ Fourier coefficients as samples and then reconstructing by computing $N$ approximate wavelet coefficients. In other words, $\tilde{f}_{N,M}$ is the solution to 
\begin{align}
 \ip{P_M \tilde{f}_{N,M}}{\varphi_j}=\ip{P_M f}{\varphi_j}, \ \ j=1,\ldots,N,
\end{align}
where $P_M$ is the projection onto the sampling space $\mathcal{S}_M$, where $\mathcal{S}_M$ is defined in (\ref{eq:S_m}).
As a comparison we will use the truncated Fourier series
$$
f_M = P_Mf = \sum_{j=1}^M \langle f,s_j\rangle s_j.
$$ 
We will sometimes assume that the samples $ \langle f,s_j\rangle$ are contaminated with noise and thus we observe
$$
\xi = \{ \langle f,s_1\rangle, \hdots,  \langle f,s_M\rangle\} + v, \qquad \|v\| = \varepsilon,
$$
for some noise level $\varepsilon \geq 0.$
Note that $f_M$ and $\tilde{f}_{N,M}$ use exactly the same information sampled.

\begin{table}
\begin{center}
{\small{
\begin{tabular}{|c| c| c| c| c|}

\thickhline
 
$(M,N,\alpha)$   	& 	 $\|f-f_M\|_{L^{2}}$	& 	 $\|f-\tilde f_{N,M}\|_{L^{2}}$  	&	Noise Level $\varepsilon$ 	&	Wavelet	\T\B\\\hline
$(934,400,2.5)$	&	$1.0\times 10^{-1}$	&	$9.7\times 10^{-2}$	&	$1.0\times 10^{-1}$	&	DB 4	\T\\
$(1834,786,2.5)$	& 	$1.0\times 10^{-2}$	& 	$9.7\times 10^{-3}$	& 	$1.0\times 10^{-2}$	& 	DB 4	\\
$(3632,1556,2.5)$	& 	$1.2\times 10^{-3}$	& 	$9.8\times 10^{-4}$	& 	$1.0\times 10^{-3}$	& 	DB 4	\B\\\hline
$(256,256,3)$	& 	$1.3\times 10^{-2}$	& 	$1.2\times 10^{-4}$	& 	$1.0\times 10^{-4}$	& 	Haar	\T\\
$(512,512,3)$	& 	$1.2\times 10^{-2}$	& 	$1.2\times 10^{-5}$	& 	$1.0\times 10^{-5}$	& 	Haar	\\
$(1024,1024,3)$	& 	$1.2\times 10^{-2}$	& 	$1.2\times 10^{-6}$	& 	$1.0\times 10^{-6}$	& 	Haar	\B\\\hline

\end{tabular}
   }}
  \end{center}
   \caption{The table shows the error of the reconstructions based on classical Fourier series, $f_M$, as well as generalized sampling 
  $\tilde f_{N,M}$ with different types of wavelets, where the samples are contaminated with noise. Note that both $f_M$ and $f_{N,M}$ use exactly the same samples.}
\label{t:numexp2}
 \end{table}

\begin{table}
\begin{center}
{\small{
\begin{tabular}{|c| c| c| c| c| c|}

\thickhline
 
$(M,\alpha)$   	& 	 $\|f-f_M\|_{L^{2}}$	& 	$\|f-\tilde f_{M/c,M}\|_{L^{2}}$ 	&	 $\|f-\tilde f_{M/c_1,M}\|_{L^{2}}$	&			Noise Level $\varepsilon$ 	&	Wavelet	\T\B\\\hline
$(482,3)$	&	$4.7\times 10^{-3}$	&	$7.3\times 10^{-7}$	&	$2.8\times 10^{-2}$	&			0	&	DB 4	\T\\
$(934,3)$	& 	$2.4\times 10^{-3}$	& 	$1.4\times 10^{-7}$	& 	$5.4\times 10^{-2}$	& 			0	& 	DB 4	\\
$(1834,3)$	& 	$1.2\times 10^{-3}$	& 	$2.6\times 10^{-8}$	& 	$1.4\times 10^{-2}$	& 			0	& 	DB 4	\B\\\hline
$(482,3)$	& 	$4.7\times 10^{-3}$	& 	$9.6\times 10^{-6}$	& 	$6.7\times 10^{2}$	& 			$1.0\times 10^{-5}$	& 	DB 4	\T\\
$(934,3)$	& 	$2.4\times 10^{-3}$	& 	$9.5\times 10^{-6}$	& 	$4.7 \times 10^3$	& 			$1.0\times 10^{-5}$	& 	DB 4	\\
$(1834,3)$	& 	$1.2\times 10^{-3}$	& 	$9.7\times 10^{-6}$	& 	$1.9\times 10^3$	& 			$1.0\times 10^{-5}$	& 	DB 4	\B\\\hline

\end{tabular}
   }}
  \end{center}
   \caption{The table shows the error of the reconstructions based on classical Fourier series, $f_M$, as well as generalized sampling 
  $\tilde f_{N,M}$ with $N = M/c$ and $N= M/c_1$, with noiseless and noisy data. Note that $f_M$ and $f_{N,M}$ use exactly the same samples.}
\label{t:numexp3}
 \end{table}

The fact that 
$$
\|f-f_M\| = \|P_M^{\perp}f\|, \qquad \|f-\tilde f_{N,M}\| \leq \theta \|Q_N^{\perp}f\|,
$$
together with (\ref{eq;decay}) show that the reconstruction created by generalized sampling will asymptotically outperform the reconstruction based on the truncated Fourier series on the types of functions described in (\ref{eq;decay}).
In particular, since $\tilde{f}_{N,M}$ is quasi-optimal, we have that 
$$
- \frac{\log \| f - \tilde{f}_{N,M} \|}{\log N} \approx \alpha - \frac{1}{2}
$$
for large $N$. This is verified in Table \ref{t:numexp1}.

Also, observe in Table \ref{t:numexp2} the predicted stability of generalized sampling. In particular, the condition number of generalized sampling is equal to $\theta$ which in the case of this experiment is $\pi/2$ for the Haar case and $1.46$ for the DB4.  As expected, the error in generalized sampling is of the same order of magnitude as the noise level. Note that the reconstruction based on the truncated Fourier series is also stable and its error in the upper part of the table also follows the noise level closely. However, in the lower half of the table, the noise level is much smaller than the error caused by the slow convergence of the truncated Fourier series and thus its error is dominated by the error from the tail of the Fourier series.

In Table \ref{t:numexp3} we demonstrate that if the number of samples $M$ does not satisfy the stable sampling rate we get an unstable and non-convergent method. In particular, we compare the choices 
$$
M = cN, \quad c = \frac{1}{\epsilon \left\lceil a\right\rceil}, \qquad M = c_1N, \quad c_1 = 0.95c. 
$$ 
As verified in Table \ref{t:numexp3} the latter choice gives disastrous results.

\section{Conclusions and future work}\label{s:conclusion}
The aim of this paper has been to show that generalized sampling solves the problem of computing one-dimensional wavelet coefficients in a stable and accurate manner from Fourier samples.  In particular, we have proved that the stable sampling rate is linear for all wavelets, and thus generalized sampling is, up to a constant factor, an optimal method for this problem.  Furthermore, we have shown that, for the class of perfect reconstruction methods, any attempt to lower the stable sampling ratio necessarily results in exponential ill-conditioning.

Perhaps the most important direction for future work is the extension of this analysis to the higher-dimensional setting.  We expect that much of the analysis carried out in this paper can be generalized in this way, and this currently work in progress.  Higher dimensions also opens the possibility for using more exotic approximation systems, such as contourlets \cite{Vetterli, Do},  curvelets \cite{Cand, Candes2002} and shearlets \cite{Gitta, Gitta2, GItta3}.  This is another topic for future work.

Besides the theory, the main hurdle to overcome in passing to higher dimensions is that of computational complexity.  As discussed in \cite{BAACHAccRecov}, this is $\mathcal{O}(N M)$ in general (i.e. $\mathcal{O}(N^2)$ whenever the stable sampling rate is linear, such as in the wavelet case), since one is required to solve a dense $M \times N$ well-conditioned least-squares problem. In two or more dimensions this value becomes prohibitively large.  However, for wavelets at least, the corresponding matrix is extremely structured. In the Haar wavelet case, for example, it can be decomposed using a combination of the discrete wavelet and discrete Fourier transform.  Hence the computational cost reduces to $\mathcal{O}(N \log N)$.  It is therefore highly likely that for general wavelets the complexity of computing the reconstruction can be similarly reduced to only $\mathcal{O}(N \log N)$, paving the way for implementable algorithms in higher dimensions.

Another topic we have not addressed is that of sparsity.  The generalized sampling framework studied in this paper guarantees recovery of all signals in a wavelet basis from their Fourier samples. However, suppose now that the signal to be recovered is in fact sparse in the wavelet domain, or compressible (i.e. well approximated by a sparse signal).  Can this property be exploited to reduce the number of Fourier samples used in recovering the signal?

An abstract framework for sparsity-exploiting generalized sampling was recently developed in \cite{BAACHGSCS}.  Note that this is intimately related to the field of compressed sensing \cite{candesCSMag,EldarKutyniokCSBook,FornasierRauhutCS}.  However, unlike the standard compressed sensing framework, which models signals as finite length vectors in vector spaces, the framework developed in \cite{BAACHGSCS} models signals as elements of separable, infinite-dimensional Hilbert spaces. As discussed in \cite{BAACHGSCS}, the infinite-dimensional model can often be more faithful to the original problem, leading to significant potential benefits.  For example, in the MRI problem -- which is best modelled by the continuous, as opposed to the discrete, Fourier transform -- it allows one to avoid the issues raised in Remark \ref{r:DFT}.

The aim of future work in this direction is to combine the results of this paper with the framework of \cite{BAACHGSCS} so as to obtain a full theory for wavelet reconstructions of compressible signals from Fourier samples.  In particular, the analogue of the stable sampling rate in \cite{BAACHGSCS}, known as the \textit{balancing property}, must be first analysed. Moreover, compressed sensing relies on so-called incoherence between sampling and reconstruction bases. This must also be estimated.

Another open problem involves the question of Fourier samples taken non-uniformly.  In this paper we have considered only Fourier samples taken  on a regular lattice.  However, non-uniform sampling patterns are more common in applications.  The question of generalized sampling for non-uniform Fourier samples was considered previously in \cite{BAACHOptimality} within the setting of Fourier frames.  We believe that the key results proved herein regarding the behaviour of the stable sampling rate can be extended to this case.


\addcontentsline{toc}{section}{References}
\bibliographystyle{abbrv}
\bibliography{References}

\begin{thebibliography}{10}

\bibitem{BAACHGSCS}
B.~Adcock and A.~C. Hansen.
\newblock Generalized sampling and infinite-dimensional compressed sensing.
\newblock {\em Technical report NA2011/02, DAMTP, University of Cambridge},
  2011.

\bibitem{BAACHSampTA}
B.~Adcock and A.~C. Hansen.
\newblock Reduced consistency sampling in {H}ilbert spaces.
\newblock In {\em Proceedings of the 9th International Conference on Sampling
  Theory and Applications}, 2011.

\bibitem{BAACHShannon}
B.~Adcock and A.~C. Hansen.
\newblock A generalized sampling theorem for stable reconstructions in
  arbitrary bases.
\newblock {\em J. Fourier Anal. Appl.}, 18(4):685--716, 2012.

\bibitem{BAACHAccRecov}
B.~Adcock and A.~C. Hansen.
\newblock Stable reconstructions in {H}ilbert spaces and the resolution of the
  {G}ibbs phenomenon.
\newblock {\em Appl. Comput. Harmon. Anal.}, 32(3):357--388, 2012.

\bibitem{AHHTillposed}
B.~Adcock, A.~C. Hansen, E.~Herrholz, and G.~Teschke.
\newblock Generalized sampling: extension to frames and ill-posed problems.
\newblock {\em Inverse Problems. (to appear)}, 2011.

\bibitem{BAACHOptimality}
B.~Adcock, A.~C. Hansen, and C.~Poon.
\newblock Beyond consistent reconstructions: optimality and sharp bounds for
  generalized sampling, and application to the uniform resampling problem.
\newblock {\em Preprint}, 2012.

\bibitem{BlumensathUnionSubspace}
T.~Blumensath.
\newblock Sampling theorems for signals from the union of finite-dimensional
  linear subspaces.
\newblock {\em IEEE Trans. Inform. Theory}, 55(4):1872--1882, 2009.

\bibitem{candesCSMag}
E.~J. Cand{\`e}s.
\newblock An introduction to compressive sensing.
\newblock {\em IEEE Signal Process. Mag.}, 25(2):21--30, 2008.

\bibitem{Cand}
E.~J. Cand{\`e}s and D.~L. Donoho.
\newblock Recovering edges in ill-posed inverse problems: optimality of
  curvelet frames.
\newblock {\em Ann. Statist.}, 30(3):784--842, 2002.

\bibitem{Candes2002}
E.~J. Cand{\`e}s and D.~L. Donoho.
\newblock New tight frames of curvelets and optimal representations of objects
  with piecewise {$C^2$} singularities.
\newblock {\em Comm. Pure Appl. Math.}, 57(2):219--266, 2004.

\bibitem{chui1992compactly}
C.~Chui and J.~Wang.
\newblock On compactly supported spline wavelets and a duality principle.
\newblock {\em Trans. Amer. Math. Soc}, 330(2):903--915, 1992.

\bibitem{cohen2006biorthogonal}
A.~Cohen, I.~Daubechies, and J.~Feauveau.
\newblock Biorthogonal bases of compactly supported wavelets.
\newblock {\em Communications on pure and applied mathematics}, 45(5):485--560,
  2006.

\bibitem{Gitta}
S.~Dahlke, G.~Kutyniok, P.~Maass, C.~Sagiv, H.-G. Stark, and G.~Teschke.
\newblock The uncertainty principle associated with the continuous shearlet
  transform.
\newblock {\em Int. J. Wavelets Multiresolut. Inf. Process.}, 6(2):157--181,
  2008.

\bibitem{Gitta2}
S.~Dahlke, G.~Kutyniok, G.~Steidl, and G.~Teschke.
\newblock Shearlet coorbit spaces and associated banach frames.
\newblock {\em Applied and Computational Harmonic Analysis}, 27(2):195--214,
  2009.

\bibitem{daubechies1992ten}
I.~Daubechies.
\newblock {\em Ten Lectures on Wavelets}.
\newblock Cbms-Nsf Regional Conference Series in Applied Mathematics. Society
  for Industrial and Applied Mathematics, 1992.

\bibitem{daubechies2004iterative}
I.~Daubechies, M.~Defrise, and C.~De~Mol.
\newblock An iterative thresholding algorithm for linear inverse problems with
  a sparsity constraint.
\newblock {\em Communications on pure and applied mathematics},
  57(11):1413--1457, 2004.

\bibitem{Vetterli}
M.~N. Do and M.~Vetterli.
\newblock The {C}ontourlet {T}ransform: {A}n {E}fficient {D}irectional
  {M}ultiresolution {I}mage {R}epresentation.
\newblock {\em {IEEE} {T}ransactions on {I}mage {P}rocessing},
  14(12):2091--2106, 2005.

\bibitem{EldarRobConsistSamp}
T.~Dvorkind and Y.~C. Eldar.
\newblock Robust and consistent sampling.
\newblock {\em IEEE Signal Process. Letters}, 16(9):739--742, 2009.

\bibitem{eldar2003FAA}
Y.~C. Eldar.
\newblock Sampling with arbitrary sampling and reconstruction spaces and
  oblique dual frame vectors.
\newblock {\em J. Fourier Anal. Appl.}, 9(1):77--96, 2003.

\bibitem{eldar2003sampling}
Y.~C. Eldar.
\newblock Sampling without input constraints: Consistent reconstruction in
  arbitrary spaces.
\newblock In A.~I. Zayed and J.~J. Benedetto, editors, {\em Sampling, Wavelets
  and Tomography}, pages 33--60. Boston, MA: Birkh{\"a}user, 2004.

\bibitem{EldarUnionSubspace}
Y.~C. Eldar.
\newblock Robust recovery of signals from a structured union of subspaces.
\newblock {\em IEEE Trans. Inform. Theory}, 55(11):5302--5316, 2009.

\bibitem{EldarMinimax}
Y.~C. Eldar and T.~Dvorkind.
\newblock A minimum squared-error framework for generalized sampling.
\newblock {\em IEEE Trans. Signal Process.}, 54(6):2155--2167, 2006.

\bibitem{EldarKutyniokCSBook}
Y.~C. Eldar and G.~Kutyniok, editors.
\newblock {\em Compressed Sensing: Theory and Applications}.
\newblock Cambridge University Press, 2012.

\bibitem{eldar2005general}
Y.~C. Eldar and T.~Werther.
\newblock General framework for consistent sampling in {H}ilbert spaces.
\newblock {\em Int. J. Wavelets Multiresolut. Inf. Process.}, 3(3):347, 2005.

\bibitem{Erdelyi92remez-typeinequalities}
T.~Erd{\'e}lyi.
\newblock Remez-type inequalities on the size of generalized polynomials.
\newblock {\em J. London Math. Soc}, 45:255--264, 1992.

\bibitem{Grochenig}
H.~G. Feichtinger, K.~Gr\"{o}chenig, and T.~Strohmer.
\newblock Efficient numerical methods in non-uniform sampling theory.
\newblock {\em Numerische Mathematik}, 69:423--440, 1995.
\newblock 10.1007/s002110050101.

\bibitem{FornasierRauhutCS}
M.~Fornasier and H.~Rauhut.
\newblock Compressive sensing.
\newblock In {\em Handbook of Mathematical Methods in Imaging}, pages 187--228.
  Springer, 2011.

\bibitem{GelmanWood3DWaveletEncoding}
N.~Gelman and M.~L. Wood.
\newblock Wavelet encoding for 3-d gradient echo {MR}-imaging.
\newblock {\em Magn. Reson. Med.}, 36:613--19, 1996.

\bibitem{guerquin2011fast}
M.~Guerquin-Kern, M.~Haberlin, K.~Pruessmann, and M.~Unser.
\newblock A fast wavelet-based reconstruction method for magnetic resonance
  imaging.
\newblock {\em Medical Imaging, IEEE Transactions on}, 30(9):1649--1660, 2011.

\bibitem{hansen2011}
A.~C. Hansen.
\newblock On the solvability complexity index, the n-pseudospectrum and
  approximations of spectra of operators.
\newblock {\em J. Amer. Math. Soc.}, 24(1):81--124, 2011.

\bibitem{HealyWeaverWaveletIEEE}
D.~M. Healy and J.~B. Weaver.
\newblock Two applications of wavelet transforms in {M}agnetic {R}esonance
  {I}maging.
\newblock {\em IEEE Trans. Inform. Theory}, 38(2):840--862, 1992.

\bibitem{eugenio1996first}
E.~Hern{\'a}ndez and G.~Weiss.
\newblock {\em A First Course on Wavelets}.
\newblock Studies in Advanced Mathematics. CRC Press, 1996.

\bibitem{UnserHirabayashiConsist}
A.~Hirabayashi and M.~Unser.
\newblock Consistent sampling and signal recovery.
\newblock {\em IEEE Trans. Signal Process.}, 55(8):4104--4115, 2007.

\bibitem{hrycakIPRM}
T.~Hrycak and K.~Gr\"{o}chenig.
\newblock Pseudospectral {F}ourier reconstruction with the modified inverse
  polynomial reconstruction method.
\newblock {\em J. Comput. Phys.}, 229(3):933--946, 2010.

\bibitem{GItta3}
G.~Kutyniok, J.~Lemvig, and W.-Q. Lim.
\newblock Compactly supported shearlets.
\newblock In M.~Neamtu and L.~Schumaker, editors, {\em Approximation Theory
  XIII: San Antonio 2010}, volume~13 of {\em Springer Proceedings in
  Mathematics}, pages 163--186. Springer New York, 2012.

\bibitem{KyriakosEtAlGeneralizedEncoding}
W.~E. Kyriakos, W.~S. Hoge, and D.~Mitsouras.
\newblock Generalized encoding through the use of selective excitation in
  accelerated parallel {MRI}.
\newblock {\em NMR Biomed.}, 19:379--392, 2006.

\bibitem{LaineWaveletsBiomed}
A.~F. Laine.
\newblock Wavelets in temporal and spatial processing of biomedical images.
\newblock {\em Annu. Rev. Biomed. Eng.}, 02:511--550, 2000.

\bibitem{LuDoUnionSubspace}
Y.~M. Lu and M.~N. Do.
\newblock A theory for sampling signals from a union of subspaces.
\newblock {\em IEEE Trans. Signal Process.}, 56(6):2334--2345, 2008.

\bibitem{NowakWaveletDenoise}
R.~Nowak.
\newblock Wavelet-based {R}ician noise removal for {M}agnetic {R}esonance
  {I}maging.
\newblock {\em IEEE Trans. Image Proc.}, 8:1408--19, 1998.

\bibitem{PanychWaveletEncoding}
L.~P. Panych.
\newblock Theoretical comparison of {F}ourier and wavelet encoding in
  {M}agnetic {R}esonance {I}maging.
\newblock {\em IEEE Trans. Med. Imaging}, 15(2):141--153, 1996.

\bibitem{PanychEtAlWaveletEncoded}
L.~P. Panych, P.~D. Jakab, and F.~A. Jolesz.
\newblock Implementation of wavelet-encoded {MR} imaging.
\newblock {\em J. Magn. Reson. Imaging}, 3:649--55, 1993.

\bibitem{Do}
D.~D.-Y. Po and M.~N. Do.
\newblock Directional multiscale modeling of images using the contourlet
  transform.
\newblock {\em Trans. Img. Proc.}, 15(6):1610--1620, June 2006.

\bibitem{pruessmann1999sense}
K.~P. Pruessmann, M.~Weiger, M.~B. Scheidegger, P.~Boesiger, et~al.
\newblock Sense: sensitivity encoding for fast mri.
\newblock {\em Magnetic Resonance in Medicine}, 42(5):952--962, 1999.

\bibitem{UnserSplinesFit}
M.~Unser.
\newblock Splines: A perfect fit for signal and image processing.
\newblock {\em IEEE Signal Process. Mag.}, 16(6):22--38, 1999.

\bibitem{unser2000sampling}
M.~Unser.
\newblock Sampling--50 years after {S}hannon.
\newblock {\em Proc. IEEE}, 88(4):569--587, 2000.

\bibitem{unser1994general}
M.~Unser and A.~Aldroubi.
\newblock A general sampling theory for nonideal acquisition devices.
\newblock {\em IEEE Trans. Signal Process.}, 42(11):2915--2925, 1994.

\bibitem{UnserAldroubiWaveletReview}
M.~Unser and A.~Aldroubi.
\newblock A review of wavelets in biomedical applications.
\newblock {\em Proc. IEEE}, 84(4):626--638, 1996.

\bibitem{unser1992asymptotic}
M.~Unser, A.~Aldroubi, and M.~Eden.
\newblock On the asymptotic convergence of< e1> b</e1>-spline wavelets to gabor
  functions.
\newblock {\em Information Theory, IEEE Transactions on}, 38(2):864--872, 1992.

\bibitem{UnserAldroubiLaineEditorial}
M.~Unser, A.~Aldroubi, and A.~F. Laine.
\newblock Guest editorial: wavelets in medical imaging.
\newblock {\em IEEE Trans. Med. Imaging}, 22(3):285--288, 2003.

\bibitem{unserzerubia}
M.~Unser and J.~Zerubia.
\newblock A generalized sampling theory without band-limiting constraints.
\newblock {\em IEEE Trans. Circuits Syst. II.}, 45(8):959--969, 1998.

\bibitem{WeaverEtAlWaveletFiltering}
J.~B. Weaver, Y.~Xu, D.~M. Healy, and J.~R. Driscoll.
\newblock Filtering {MR} images in the wavelet transform domain.
\newblock {\em Magn. Reson. Med.}, 21:288--295, 1991.

\bibitem{WeaverEtAlWaveletEncoding}
J.~B. Weaver, Y.~Xu, D.~M. Healy, and J.~R. Driscoll.
\newblock Wavelet-encoded {MR} imaging.
\newblock {\em Magn. Reson. Med.}, 24:275--287, 1992.

\end{thebibliography}

\end{document}